\documentclass{article}
\usepackage[margin=0.75in]{geometry}

\usepackage{amsthm}
\usepackage{tikz}
\usepackage{mathtools}
\usepackage{aliascnt}
\usepackage{hyperref}
\usepackage{cleveref}
\usepackage{amsfonts}
\usepackage{enumitem}
\usepackage{ifthen}
\usetikzlibrary{decorations.markings,decorations.pathreplacing,calligraphy}

\newcommand{\N}{\mathbb{N}}
\newcommand{\Z}{\mathbb{Z}}

\newcommand{\Q}{\mathbb{Q}}
\newcommand{\jac}[1]{\mathrm{Jac}\left({#1}\right)}
\newcommand{\opdiv}[1]{\mathrm{Div}\left({#1}\right)}
\newcommand{\pic}[1]{\mathrm{Pic}\left({#1}\right)}

\newcommand{\val}[1]{\mathrm{val}\left({#1}\right)}
\newcommand{\Inv}{\operatorname{Inv}}
\newcommand{\inv}{\operatorname{inv}}
\newcommand{\eaf}[1]{\widetilde{\Sigma}_{#1}}
\newcommand{\Supp}{\mathrm{Supp}}
\newcommand{\set}[1]{\left\{{#1}\right\}}
\DeclarePairedDelimiter\floor{\lfloor}{\rfloor}

\newcommand{\gen}[1]{\langle #1 \rangle}
\newcommand{\id}{\mathrm{id}}
\newcommand{\tauD}{\tau^{u,v}_D}
\newcommand{\tauDi}[1]{\tau^{u_{#1},v_{#1}}_{D_{#1}}}
\newcommand{\Pic}{\operatorname{Pic}}
\newcommand{\asp}{\operatorname{ASP}}
\newcommand{\sci}{\operatorname{sci}}
\renewcommand{\min}[1]{\operatorname{min}\left\{ #1 \right\} }
\renewcommand{\max}[1]{\operatorname{max}\left\{#1\right\}}

\newif\ifshowcomments
\showcommentstrue
\newcommand{\iou}[1][]{
    \ifcsname showcomments\endcsname
        \ifthenelse{\equal{\showcomments}{1}}{\showcommentstrue}{\showcommentsfalse}
    \fi
    \ifshowcomments
        \ifthenelse{\equal{#1}{}}{{\color{blue}\{IOU\}}}%
        {{\color{blue}\{IOU: #1\}}}%
    \fi
}

\tikzset{-dot-/.style={decoration={
      markings,
      mark=at position #1 with {\fill circle (2pt);}},postaction={decorate}}}

\definecolor{dust}{HTML}{E1BE6a}
\definecolor{shallows}{HTML}{005A52}
\definecolor{ocean}{HTML}{191386}

\title{Twice-marked banana graphs \& Brill--Noether generality}
\author{Nathan Pflueger \& Noah Solomon}

\theoremstyle{plain} 
\newtheorem{theorem}{Theorem}[section]
\crefname{theorem}{theorem}{theorems}
\Crefname{theorem}{Theorem}{Theorems}

\newaliascnt{conjecture}{theorem}
\newtheorem{conjecture}[conjecture]{Conjecture}
\aliascntresetthe{conjecture}
\newaliascnt{coro}{theorem}
\newtheorem{coro}[coro]{Corollary}
\aliascntresetthe{coro}
\newaliascnt{lemma}{theorem}
\newtheorem{lemma}[lemma]{Lemma}
\aliascntresetthe{lemma}
\newaliascnt{prop}{theorem}
\newtheorem{prop}[prop]{Proposition}
\aliascntresetthe{prop}
\crefname{conjecture}{conjecture}{conjectures}
\Crefname{conjecture}{Conjecture}{Conjectures}
\crefname{coro}{corollary}{corollaries}
\Crefname{coro}{Corollary}{Corollaries}
\crefname{lemma}{lemma}{lemmas}
\Crefname{lemma}{Lemma}{Lemmas}
\crefname{prop}{proposition}{propositions}
\Crefname{prop}{Proposition}{Propositions}
\newaliascnt{example}{theorem}
\newtheorem{example}[example]{Example}
\aliascntresetthe{example}
\crefname{example}{example}{examples}
\Crefname{example}{Example}{Examples}

\theoremstyle{definition}
\newaliascnt{defi}{theorem}
\newtheorem{defi}[defi]{Definition}
\aliascntresetthe{defi}
\crefname{defi}{definition}{definitions}
\Crefname{defi}{Definition}{Definitions}

\theoremstyle{remark}
\newaliascnt{rema}{theorem}
\newtheorem{rema}[rema]{Remark}
\aliascntresetthe{rema}
\newaliascnt{nota}{theorem}

\aliascntresetthe{nota}
\crefname{rema}{remark}{remarks}
\Crefname{rema}{Remark}{Remarks}
\crefname{nota}{notation}{notations}
\Crefname{nota}{Notation}{Notations}

\begin{document}

\maketitle

\begin{abstract}
We analyze a family of graphs known as banana graphs, with two marked vertices, through the lens of Hurwitz--Brill--Noether theory.
As an application, we construct explicit new examples of finite graphs which are Brill--Noether general. These are the first such examples since the analysis of chains of loops by Cools, Draisma, Payne and Robeva. The graphs constructed are chains of loops and ``theta graphs,'' which are banana graphs of genus $2$. We also demonstrate that almost all banana graphs of genus at least $3$ cannot be used for this purpose, due either to failure of a submodularity condition or to the presence of far too many inversions, in certain permutations associated to divisors called transmission permutations. 
\end{abstract}

\section{Introduction}
This paper offers a case study in Brill--Noether theory of graphs. Its contents are complementary to \cite{Pfl22}; our aim is to perform some explicit computations and examples with the tools developed in that paper to obtain novel examples and shine light on intriguing phenomena that arise. In particular, we construct the first explicit examples of \emph{Brill--Noether general graphs} in all genera other than the famous chains of loops found in \cite{CDPR12}. 
See Section \ref{ssec:conventions} for conventions on what sort of graphs we consider.
Our focus is on an enriched form of Brill--Noether theory taking two marked vertices into account; we perform a detailed analysis of banana graphs from this point of view, and use this analysis for our other constructions.

Brill--Noether theory of graphs is a purely combinatorial subject born out of a tantalizing analogy with algebraic geometry. This analogy was brought into sharp relief by Baker and Norine \cite{BakNor07}, who proved a graph-theoretic analog of the classical Riemann--Roch formula. The essence of the analogy is that a configuration of \emph{chips} on a finite graph, up to an equivalence relation generated by \emph{chip-firing moves}, are analogous to line bundles on smooth algebraic curves. In light of this analogy, we refer to chip configurations on a finite graph as \emph{divisors}, and the chip-firing equivalence relation as \emph{linear equivalence}. A key innovation of \cite{BakNor07} is a simple and useful definition of the \emph{rank} $r(D)$ of a divisor, which, roughly speaking, is a graph-theoretic analog of the dimension of a projective space $\mathbb{P}^r$ to which a line bundle defines a map from an algebraic curve. We refer the reader to \cite{BakNor07} or the expository book \cite{corry2018divisors} for terminology and background on divisors and rank, and the Riemann--Roch theorem.

We consider three versions of Brill--Noether theory of graphs in this paper, considering graphs with zero, one or two marked vertices. There are two reasons for adding marked vertices. First, doing so provides a tool to study Brill--Noether theory of graphs without marked points, by gluing graphs at marked points. This is what enables our new constructions of Brill--Noether general graphs. Second, the enriched question lends itself to interesting examples and computations in low genus; we will see interesting behavior even in genus $2$ when two vertices are marked.

We explain each situation in turn.

\subsection{Brill--Noether general graphs}

\begin{figure}
    \newcommand{\markpt}[1]{\draw[fill=black] (#1) circle(0.05cm);}
    \newcommand{\bigmarkpt}[1]{\draw[fill=black] (#1) circle(0.1cm);}
    \centering
    \begin{tikzpicture}[scale=0.74]
        \begin{scope}
           \draw (0,0) -- (1,0) -- (1,1) -- (0,1) -- cycle;
           \bigmarkpt{0,0}
           \draw (0,0) node[left] {$u = u_1$};
           \bigmarkpt{1,0}
           \draw (1,0) node[above right] {$v_1 = u_2$};
           \markpt{1,1}
           \markpt{0,1}
           
           \draw [thick,decorate, decoration = {brace}] (0,1.2) --  (1,1.2) node[midway, above] {$k_1 = 4$};
        \end{scope}
        
        \begin{scope}[xshift=2cm,yshift=-0.5cm]
            \draw (0,0) -- (-1,0.5) -- (-2,-0.5) -- (-1,-1.5) -- (0,-1) -- cycle;
            \draw (0,0) -- (1,0.5) -- (2,-0.5) -- (1,-1.5) -- (0,-1);
            \markpt{0,0};
            \markpt{-1,0.5}
            \markpt{-2,-0.5}
            \markpt{-1,-1.5}
            \markpt{0,-1}
            \bigmarkpt{1,0.5}
            \draw (1,0.5) node[below] {$v_2$};
            \markpt{2,-0.5}
            \markpt{1,-1.5}
            \draw [thick,decorate, decoration = {brace}] (2,-2) --  (-2,-2) node[midway, below] {$k_2 = 4$};
        \end{scope}
        
        \draw (3,0) -- (4,0.25);
        
        \begin{scope}[xshift=4.5cm, yshift=-0.75cm]
        \draw (0,0) -- (1,0) -- (1.5,1) -- (0.5,1.5) -- (-0.5,1) -- cycle;
        \markpt{0,0}
        \bigmarkpt{1,0}
        \draw (1,0) node[below] {$v_3$};
        \markpt{1.5,1}
        \markpt{0.5,1.5}
        \bigmarkpt{-0.5,1}
        \draw (-0.5,1) node[above left] {$u_3$};
        \draw [thick,decorate, decoration={brace}] (-0.5,1.7) --  (1.5,1.7) node[midway, above] {$k_3=5$};
        \end{scope}
        
        \draw (5.5,-0.75) -- (6.5,-1);
        
        \begin{scope}[xshift=8.5cm,yshift=0cm]
           \draw (0,-1) -- (0,-3);
           \draw (0,-1) -- (-1,-0.5) -- (-2,-1) -- (-2,-2) -- (-1,-3) -- (0,-3);
           \draw (0,-1) -- (0.2,-0.2) -- (1,0.5) -- (2,0.5) -- (3,0) -- (3.5,-1) -- (3.5,-2) -- (3,-3) -- (2,-3.5) -- (1,-3.5) -- (0,-3);
           \markpt{0.2,-0.2}
           \markpt{0,-1}
           \markpt{0,-2}
           \markpt{-1,-0.5}
           \bigmarkpt{-2,-1}
           \draw (-2,-1) node[above] {$u_4$};
           \markpt{-2,-2}
           \markpt{-1,-3}
           \markpt{0,-3}
           \markpt{1,0.5}
           \markpt{2,0.5}
           \bigmarkpt{3,0}
           \draw (3,0) node[below left] {$v_4 = u_5$};
           \markpt{3.5,-1}
           \markpt{3.5,-2}
           \markpt{3,-3}
           \markpt{2,-3.5}
           \markpt{1,-3.5}
           \draw [thick,decorate, decoration = {brace}] (4,-4) --  (-2,-4) node[midway, below] {$k_4 = 5$};
        \end{scope}
        
        \begin{scope}[xshift=13.5cm,yshift=2cm]
        \draw (0,0) -- (0,-2);
        \draw (0,0) -- (-1,0.5) -- (-2,0) -- (-2.5,-1) -- (-2,-2) -- (-1,-2.5) -- (0,-2);
        \markpt{0,0}
        \markpt{0,-1}
        \markpt{0,-2}
        \markpt{-1,0.5}
        \markpt{-2,0}
        \markpt{-2.5,-1}
        \markpt{-2,-2}
        \markpt{-1,-2.5}
        \draw (0,0) -- (0.667,-0.667) -- (0.667,-1.333) -- (0,-2);
        \markpt{0.667,-0.667}
        \bigmarkpt{0.667,-1.333}
        \draw (0.667,-1.333) node[right] {$v_5 = v$};
        \draw [thick,decorate, decoration = {brace}] (-2.5,0.7) --  (0.667,0.7) node[midway, above] {$k_5 = 3$};
        \end{scope}
    \end{tikzpicture}
    \caption{A new example of a Brill--Noether general graph. See \Cref{eg:bng}.}
    \label{fig:bngenlexample}
\end{figure}

The driving question of this paper may be stated informally as follows: which finite graphs (perhaps with marked vertices), most closely resemble a ``typical'' algebraic curve of the same genus (perhaps with marked points), from the standpoint of ranks of divisors?
For now, we do not incorporate any marked points.
To make this informal question a bit more precise, we introduce the following term.

\begin{defi}
For a graph or smooth algebraic curve, the \emph{divisor census} is the set of all pairs $(d,r)$ of nonnegative integers for which there exists a divisor $D$ with $\deg D = d$ and $r(D) \geq r$.
\end{defi}

Much of the richness of the geometry of algebraic curves stems from the fact that not every curve of the same genus has the same divisor census. Nonetheless, the celebrated \emph{Brill--Noether theorem} \cite{gh80} says in part that there is one specific divisor census that is ``typical:'' if a genus $g$ curve is chosen at random from the moduli space of curves, then with probability $1$ its divisor census consists of those pairs $(d,r)$ for which the following \emph{Brill--Noether number} is nonnegative.
$$
\rho(g,r,d) = g - (r+1)(g-d+r) \geq 0
$$
Therefore, one may regard a graph with this same divisor census as resembling a ``typical'' algebraic curve. Nonetheless, at the outset of this story it was by no means obvious that any such graphs exist. Indeed, Baker made the following conjecture, paraphrased from \cite[Conjecture 3.9]{BakerSpec} in our terminology.

\begin{conjecture}[Brill--Noether conjecture for graphs \cite{BakerSpec}]
\label{conj:bn}
For every genus $g \geq 0$,
\begin{enumerate}[label = \arabic*)]
    \item Every genus $g$ graph contains every pair $(d,r)$ with $\rho(g,r,d) \geq 0$ in its divisor census.
    \item There exists a genus $g$ graph such that every $(d,r)$ in the divisor census satisfies $\rho(g,r,d) \geq 0$ (proved in \cite{CDPR12}).
\end{enumerate}
\end{conjecture}

Part $1)$ of \Cref{conj:bn} is still open outside of small genus \cite{atanasovRanganathan}, although the analogous statement for \emph{metric} graphs was established in \cite{BakerSpec}. Part $2)$ was proved in \cite{CDPR12} using the now-famous example of \emph{chains of loops}. 

\begin{defi}
Graphs satisfying part $2)$ of \Cref{conj:bn} are called \emph{Brill--Noether general}.
\end{defi}

\begin{rema}
Although this form of Brill--Noether generality was discussed in the original \cite{BakerSpec}, many sources use a stronger definition of Brill--Noether generality, which requires a \emph{dimension} statement as well. For this definition, one works with the metric graph corresponding to $G$, and requires that the locus $W^r_d$ of degree $d$ divisor classes of rank at least $r$, where $r > d-g$, has local dimension exactly (equivalently, at most) $\rho(g,r,d)$ everywhere. We opt for the simpler form of Brill--Noether generality herein, so that we can work purely graph-theoretically. However, we strongly believe that all results about Brill--Noether generality of marked graphs generalize exactly as stated if this dimension requirement is added. Note that the ``existential'' form of Brill--Noether generality for \emph{once-marked} graphs we work with here is still enough to deduce Brill--Noether generality (in the stronger, dimension sense) of algebraic curves specializing to the graph in question. Indeed, this was the approach of \cite{CDPR12}, which worked, in effect, with a simplified version of Brill--Noether generality of marked graphs.
\end{rema}

What \cite{CDPR12} demonstrated is that, assuming certain genericity conditions on the path lengths of the loops, chains of loops are Brill--Noether general in this sense.

Nevertheless, a question has remained since \cite{CDPR12}: which \emph{other} families of graphs include Brill--Noether general graphs? This question has proved remarkably stubborn. Some negative results are known, e.g. \cite{jensenNotDense} identifies homeomorphism classes of graphs containing no Brill--Noether general graphs, but the chain of loops remained for a long time the only family of graphs where explicit Brill--Noether general graphs are known in every genus.

\begin{rema}
It follows from the semicontinuity of Brill--Noether rank \cite{lpp,len14} that the locus of Brill--Noether general \emph{metric} graphs is open in the moduli space of metric graphs (although it is not dense, in contrast to the algebraic setting). Since a chain of loops (without bridges) may deform to a chain of loops and theta graphs, it follows there exist some such metric graphs that are Brill--Noether general. 
Choosing an example with rational edge lengths, one can deduce the existence of Brill--Noether general finite graphs of this type. Since this argument is topological in nature, it does not give specific examples. The novelty of our paper is in part that we can construct specific examples, as well as examples where the middle edge length is long compared to the others (indeed, the set of possible ratios $n_0 : n_1 : n_2$ is dense in the positive part of $\mathbb{P}^2_{\mathbb{R}}$).
\end{rema}

Since the publication of \cite{CDPR12}, chains of loops have proved to be fruitful in proving theorems about general algebraic curves, including a proof of the Gieseker--Petri theorem \cite{JensenPayneTP1}, work on the maximal rank conjecture \cite{JensenPayneTP2}, the study of Prym varieties \cite{CLRW,LenUlirsch}, and Hurwitz--Brill--Noether theory \cite{PflKGonal, JensenRanganathan, CookPowellJensen}. Several of these applications are summarized in the surveys \cite{jensenNotices, jensen2021recent}. They have also proved to be a useful setting for studying Poincar\'e series and a tropical analog of Lang's conjecture \cite{manjunath2020poincar}. Conceptually, one may interpret this as follows: the fact that chains of loops are Brill--Noether general is evidence that they are excellent stand-ins for ``typical'' algebraic curves; therefore they are excellent graphs to use when attempting to prove theorems about algebraic curves via combinatorial means.

Since the publication of \cite{CDPR12}, and in light of how useful chains of loops have proven to be, a question presents itself: \emph{what other graphs are Brill--Noether general?} 
We take herein a first step beyond the familiar landscape of chains of loops, providing the first explicit constructions of Brill--Noether general graphs other than chains of loops. The newly minted graphs have forms like the example in \Cref{fig:bngenlexample} and \Cref{eg:bng}. See \Cref{thm:thetaSimple} and \Cref{thm:bngChain} for the general construction.

The reader will observe the similarity with chains of loops: we have merely coalesced some pairs of adjacent loops into theta graphs, on which we must place certain constraints. While this is only a step outside the realm of chains of loops, it nonetheless breaks the boundary of this class of graphs, and we hope it will provide clues for future exploration.

\subsection{Once-marked graphs and Weierstrass partitions}

We now move on to enriched forms of Brill--Noether theory of graphs, in which we keep track of marked vertices.
The choice $(G,v)$ of a graph with a chosen vertex is called a \emph{(once-)marked graph}, while a choice $(G,u,v)$ of a graph and two chosen vertices is a \emph{twice-marked graph}. We will often assume that $u$ and $v$ are not linearly equivalent as divisors, which we will write $u \not\sim v$, but this is not required.  We first consider once-marked graphs.

Given a once-marked graph $(G,v)$, we enlarge our census questionnaire as follows: when examining a divisor $D$ on $G$, we record not just $r(D)$ itself, but also the ranks $r(D+\ell v)$, $\ell \in \Z$, of all divisors obtained by adding a multiple of the marked vertex.

Although this in principle requires recording infinitely many ranks, we can make our lives easier by only recording the ``excess'' beyond the minimum rank predicted by Riemann--Roch. In this way, all these ranks $r(D + \ell v)$ may be summarized in a finite combinatorial object, called the \emph{Weierstrass partition} of $D$, and denoted $\lambda(D,v)$. Here by a \emph{partition} we mean a nonincreasing sequence $(\lambda_i(D,v))_{i \geq 0}$ of nonnegative integers, only finitely many of which are nonzero. The precise definition is as follows. This is phrased differently from the original definition given in \cite{Pfl17}, but is readily checked to be equivalent.

\begin{defi}
Let $(G,v)$ be a genus $g$ graph with a marked vertex $v$. For any divisor $D$ and integer $i \geq 0$, let
$$s_i(D,v) = \operatorname{min} \left\{ \ell \in \Z:\ r(D+\ell v) \geq i \right\}.$$
Note that Riemann--Roch implies that $s_i(D,v) \leq i + g - \deg D$ for all $i \geq 0$, with equality for $i \gg 0$.
These numbers are commonly interpreted as pole orders (although some may be negative).
The \emph{Weierstrass partition} of $D$ with respect to $v$ is the nonincreasing sequence of nonnegative integers 
$\lambda(D,v) = (\lambda_0(D,v), \lambda_1(D,v), \ldots)$
defined by $$\lambda_i(D,v) = i  + g - \deg D- s_i(D,v).$$
The (finite) sum $\displaystyle \sum_{i=0}^\infty \lambda_i(D,v)$ is denoted $| \lambda(D,v) |$.
\end{defi}

With this definition in hand, we may inquire about the following more refined census.

\begin{defi}
For a once-marked graph $(G,v)$, the \emph{divisor census} is the set of all partitions $\lambda$ for which there exists a divisor $D$ with $\lambda_i(D,v) \geq \lambda_i$ for all $i \geq 0$.
\end{defi}

This census is related to the census of $G$ (with no marked points) in a simple way: $(d,r)$ belongs to the census of $G$ if and only if either $r < d-g$ or the census of $(G,v)$ includes the ``rectangular'' partition $(g-d+r,g-d+r,\ldots,g-d+r,0,\ldots)$, where there are $r+1$ copies of $g-d+r$. In fact, the analog of the Brill--Noether number is $g - |\lambda|$.

\begin{conjecture}[Brill--Noether existence conjecture for once-marked graphs]\label{conj:bnOnceMarked}
For \emph{any} once-marked graph $(G,v)$ of genus $g$, every partition $\lambda$ with $|\lambda| \leq g$ is in the divisor census. 
\end{conjecture}

\begin{defi}
A once-marked graph $(G,v)$ is called \emph{Brill--Noether general} if $|\lambda(D,v)| \leq g$ for all divisors $D$ on $G$. In other words, every partition in the divisor census has size at most $g$.
\end{defi}

\Cref{conj:bnOnceMarked} implies \Cref{conj:bn}. Like \Cref{conj:bn}, it is known to hold for \emph{metric} graphs, although the only known proof requires algebraic geometry and intersection theory; see Propositions 4.2 and 5.1 of \cite{Pfl17}.

With suitable genericity hypotheses, a chain of loops with a vertex marked at one end is Brill--Noether general in this sense, as proved in \cite{Pfl17}. This paper constructs new examples of Brill--Noether general marked graphs, consisting of chains that mix loops with theta graphs (\Cref{eg:bng}).

\subsection{Twice-marked graphs and transmission permutations}

The situation with two marked points is studied in \cite{Pfl22}. Given a twice-marked graph $(G,u,v)$, we enlarge our census further, and inquire about all ranks $r(D+au+bv)$ for $a,b \in \Z$. As in the once-marked situation, we hope to record this in a finite combinatorial datum, called a \emph{transmission permutation} and denoted $\tauD$. When it exists, $\tauD$ encodes all of these ranks, via the equation

\begin{equation}\label{eq-tauDIntroDefn}
    r(D + au + bv) + 1 = \# \{ i \geq -b:\ \tauD(i) \leq a \}.
\end{equation}

An equivalent description of $\tauD$ is given in \Cref{def-tauD}.
Two complexities emerge in the twice-marked situation. First, transmission permutations do not always exist: $\tauD$ exists if and only if $D$ satisfies a convexity condition called \emph{submodularity}. 
Submodularity requires that for all divisors $D' = D+au+bv$, if $r(D') = r(D'-v)$, then also $r(D'-u) = r(D'-u-v)$. 
In geometric language: if $|D'|$ has a base point at $v$, then so must $|D'-u|$. An equivalent definition is given in \Cref{ssec:conventions}.
Any divisor on an algebraic curve is submodular, so it is natural to regard those twice-marked graphs on which all divisors are submodular as better analogs of twice-marked algebraic curves. 
Second, on a finite graph, the class $[u-v]$ has finite order. That is, $ku \sim kv$ for some positive integer $k$. The minimum such $k$ is called the \emph{torsion order} of $(G,u,v)$ (this is \Cref{def-TwMkGraph}). The torsion order has a profound effect on transmission permutations: $\tauD$ (when it exists) always satisfies the periodicity property $\tauD(n+k) = \tauD(n)+k$ for all $n \in \Z$. We will call such permutations \emph{extended $k$-affine}, and denote the group of such permutations by $\eaf{k}$. As explained in \cite{Pfl22}, a useful analog of the Brill--Noether number is $g - \inv_k(\tauD)$. Here, $\inv_k$ denotes the length in the affine symmetric group, or equivalently the number of inversions modulo $k$. More precisely, $\inv_k(\tauD)$ is the number of pairs $(i,j)$ such that $i<j$, $\tauD(i) > \tauD(j)$, and $0 \leq i < k$. We will call these \emph{$k$-inversions} of $\tau$. See Section \ref{ssec:conventions} for details.

\begin{defi}
A genus $g$ twice-marked graph $(G,u,v)$ for which $ku \sim kv$ is said to have \emph{$k$-general transmission} if all divisors $D$ are submodular, and satisfy $\inv_k(\tauD) \leq g$.
\end{defi}

Although this definition requires only that $k u \sim kv$, i.e. that $k$ is divisible by the torsion order, we will see in \Cref{lem:kgtImpliesTorsionOrder} that it implies when $g \ge 1$ that $k$ is \emph{exactly} the torsion order (the $g=0$ case is trivial: a twice-marked tree has $k$-general transmission for every $k$, which can be useful in inductive constructions). So in practice we can sometimes be lax in specifying which $k$ is intended when we say that a twice-marked graph has $k$-general transmission: in positive genus, the only $k$ that could be intended is the torsion order.

\begin{example}\label{eg:cycle}
    If $G$ is a cycle graph, with two marked points $u,v$ joined by two paths of length $a$ and $b$, then the torsion order of $(G,u,v)$ is $k =\frac{a+b}{\gcd(a,b)}$ and $(G,u,v)$ has $k$-general transmission \cite[\S 2.1]{Pfl22}.
\end{example}

Unfortunately, the relationship between $k$-general transmission and Brill--Noether generality is not as simple as one would like. One way to view the difficulty is that \emph{no specific torsion order is ``typical.''} So the definition of $k$-general transmission accepts this and instead specifies the most generic situation given a specific torsion order. Furthermore, for sufficiently large torsion order, namely $k \geq \frac12 g + 1$, we can in fact deduce Brill--Noether generality; this is proved in \Cref{prop:kgt-bngenl}.

One can also view the study of $k$-general transmission as a part of \emph{Hurwitz}--Brill--Noether theory, which studies the geometry of linear series on general curves of a fixed gonality $k$. This point of view is explained in \cite{Pfl17}, though we do not discuss it in detail here. See \cite{larsonHBN,larsonLarsonVogt} for the main results of Hurwitz--Brill--Noether theory for algebraic curves.

We emphasize that $k$-general transmission neither implies Brill--Noether generality nor is implied by it in general. Nevertheless, the two notions are related.
We explain in Section \ref{sec:mixedTorsion} how twice-marked graphs with $k$-general transmission may be used to construct Brill--Noether general graphs and once-marked graphs. For now, it suffices to remark that, when working to understand which twice-marked graphs best represent twice-marked algebraic curves, it is natural to consider three categories, in order of how well the twice-marked graphs resemble general twice-marked algebraic curves.

\begin{enumerate}[label = \arabic*)]
    \item Twice-marked graphs with non-submodular divisors.
    \item Twice-marked graphs with all divisors submodular, but some with $\inv_k(\tauD) > g$.
    \item Twice-marked graphs of torsion order $k$ with $k$-general transmission.
\end{enumerate}

Twice-marked graphs in the first category are particularly poor avatars of twice-marked algebraic curves, since \emph{every} twice-marked algebraic curve, with no genericity assumptions, has all divisors submodular. Therefore one should use caution using such graphs to obtain intuition about curves, and our results give some examples of such behavior.

\subsection{Summary of results}

This paper is a case study in this classification of twice-marked graphs. Our objects of interest are twice-marked \emph{banana graphs}, which we aim to classify into the three categories stated above.
For genus $2$ banana graphs, also called \emph{theta graphs} (as they look topologically like a $\theta$ symbol), we obtain quite precise results. We defer the precise statements until after developing the necessary notation, but the theorem below summarizes some easy-to-state consequences.

\begin{theorem}\label{thm:thetaSimple}
Let $(G,u,v)$ be a theta graph with two marked points.
\begin{enumerate}[label = \arabic*)]
    \item If $u$ and $v$ are located on the interiors of distinct strands of $G$ (as defined in Section \ref{ssec:conventions}), then all divisors on $G$ are submodular.
    \item If $(G,u,v)$ is \emph{evenly marked}, meaning that $u$ and $v$ divide their strands into two segments with the same ratio $\frac{a}{b} \in \Q$ (see \Cref{defn:evenlyMarked} for a precise definition), then $(G,u,v)$ has $k$-general transmission, where 
    $k$ is the torsion order of $(G,u,v)$.
\end{enumerate}
\end{theorem}

The two parts of \Cref{thm:thetaSimple} follow directly from \Cref{thm-NonSubmodGenus2} and \Cref{cor:evenlyMarkedKGT}, respectively. Part 1) becomes bi-conditional once a few cases are added, and part 2) is an example where we can apply a ``non-recurrence criterion,'' \Cref{thm:kgtThetas}, that is tractable to verify for evenly marked theta graphs. We are optimistic that non-recurrence should be tractable to verify in other cases as well, and that a complete classification of genus $2$ twice-marked graphs may be possible, but some new ideas are needed.

\Cref{thm:thetaSimple} provides a large family of new genus $2$ twice-marked graphs with $k$-general transmission for any value $k \ge 2$. We can then use these to construct many new examples of Brill--Noether general graphs and once-marked graphs by coupling these constructions, and \Cref{eg:cycle}, with the following theorem, proved in Section \ref{sec:mixedTorsion}.

\begin{theorem}
\label{thm:bngChain}
Let $(G_i, u_i, v_i)$, for $i=1,2, \ldots, \ell$, be a sequence of $\ell$ twice-marked graphs, and $(G,u,v) = (G,u_1, v_\ell)$ the iterated vertex gluing (see Section \ref{ssec:conventions}). Let $g_i$ and $k_i$ be the genus of $G_i$ and torsion order of $(G_i,u_i,v_i)$, respectively.  Suppose that each $(G_i, u_i,v_i)$ has $k_i$-general transmission.
\begin{enumerate}[label = \arabic*)]
    \item If $k_i > g_1 + g_2 + \ldots + g_i$ for all $i$, then $(G,v)$ is a Brill--Noether general marked graph.
    \item If $k_i > \min{g_1 + g_2 + \ldots + g_i, g_i + g_{i+1} + \ldots + g_\ell }$ for all $i$, then $G$ is a Brill--Noether general graph.
\end{enumerate}
\end{theorem}

\begin{rema}
As a degenerate case, \Cref{thm:bngChain} shows that a single twice-marked graph $(G,u,v)$ with $k$-general transmission is Brill--Noether general provided that $k > g$. In fact, we prove in \Cref{prop:kgt-bngenl} that $k \geq \frac12 g + 1$ is sufficient. Moreover, that stronger bound is sharp. We strongly believe that one can refine the bound in \Cref{thm:bngChain} so that it gives this sharper bound in the $\ell = 1$ case, but we have not attempted to do so here in order to simplify the statements and arguments.
\end{rema}

\begin{example}
\label{eg:bng}
    Consider the graph in \Cref{fig:bngenlexample}. Regard it as a twice-marked graph with the leftmost and rightmost thick vertices marked. This graph is a chain of $5$ twice-marked graphs $(G_i, u_i, v_i)$, $1 \leq i \leq 5$. Some attachments are vertex gluings, while others include a bridge; as we discuss below this does not affect Brill--Noether generality. The five twice-marked graphs are as follows. We use some notation to be officially defined later in the paper.
    
    $(G_1,u_1,v_1)$ is a cycle ($g_1 = 1$), with the marked points joined by paths of length $3$ and $1$. By \Cref{eg:cycle}, the torsion order is $k_1 = \frac{3+1}{\gcd(3,1)} = 4$, and we have $4$-general transmission.
    
    $(G_2,u_2,v_2) = (\theta_{4,1,4}, x_1, z_1)$ is an evenly marked theta graph ($g_2 = 2$), with edges divided in the ratio $\frac{1}{3}$, so by \Cref{thm:thetaSimple} it has torsion order $k_2 = \frac{1+3}{\gcd(1,3)} = 4$ and $4$-general transmission.
    
    $(G_3, u_3, v_3)$ is a cycle  with $\frac{3+2}{\gcd(3,2)} = 5$-general transmission.
    
    $(G_4,u_4,v_4) = (\theta_{5,2,10}, x_2, z_4)$ is an evenly marked theta with $\frac{2+3}{\gcd(2,3)} = \frac{4+6}{\gcd(4,6)} = 5$-general transmission. Note that in this case, the marked strands do not have \emph{equal} length, but they are cut in the same ratio.
    
    $(G_5, u_5, v_5) = (\theta_{6,2,3}, x_4, z_2)$ is an evenly marked theta with $\frac{4+2}{\gcd(4,2)} = \frac{2+1}{\gcd(2,1)} = 3$-general transmission.
    
    The genera are $(g_1, g_2, g_3,g_4,g_5) = (1,2,1,2,2)$, and the torsion orders are $(k_1, k_2, k_3, k_4, k_5) = (4,4,5,5,3)$. Since $k_1 > g_1$, $k_2 > g_1+g_2$, $k_3 > g_1+g_2+g_3$, $k_4 > g_4 + g_5$, and $k_5 > g_5$, it follows from \Cref{thm:bngChain} that this graph is Brill--Noether general.
    
    It is straightforward to construct many graphs of this type, blending cycles and evenly-marked thetas to taste.
\end{example}

Our original motivation for this paper was to understand theta graphs, with the goal of classifying genus $2$ twice-marked curves. In the course of this work, it became clear that our methods generalize readily to banana graphs of all genera, so our remaining results concern banana graphs of genus $g \geq 3$. However, a stark difference emerges in our results: banana graphs of genus $g \geq 3$ almost never have $k$-general transmission, and in fact usually have non-submodular divisors. 

\begin{theorem} \label{thm:bananaSimple}
Let $(G,u,v)$ be a banana graph of genus $g \geq 3$, marked at two vertices $u,v$, at least one of which lies at least distance $2$ from both multivalent vertices, and neither of which is the midpoint of a length-$2$ strand. Then there exist non-submodular divisors on $(G,u,v)$.
\end{theorem}

\Cref{thm:bananaSimple} follows immediately from the more precise \Cref{thm-NSMForBanana}. It leaves only a few special cases left where $k$-general transmission might be possible. We rule most of these out as well, and prove the following theorem as part of \Cref{cor:bananasWithKGT}.

\begin{theorem}\label{thm:bananas}
    A twice-marked banana graph of genus $g \geq 3$ does not have $k$-general transmission for any $k\geq 3$.
\end{theorem}

In fact, we show in Section \ref{sec:kgt} that most banana graphs of genus $g \geq 3$ are extremely far from $k$-general transmission: even in cases where all divisors are submodular, not only do these graphs possess transmission permutations with more than $g$ $k$-inversions, they possess transmission permutations whose number of $k$-inversions grows at least quadratically in $g$. We defer the precise statements and hypotheses to Section \ref{ssec:mostBananas}.

Sections \ref{ssec:mostBananas} and \ref{ssec:mpAutos} also contain several computations of transmission permutations on banana graphs displaying some intriguing structure. Though we do not need to exploit all of this structure for our results, we have shown it in some detail due to some intriguing patterns that emerge, that may be suggestive of additional structure within transmission permutations. We hope this case study may be informative for future work on the topic.

\begin{rema}
Readers familiar with the moduli space of tropical curves may not be surprised that genus $g \geq 3$ banana graphs do not behave like general graphs, since they belong to a highly special stratum of this moduli space. Indeed, a genus $g$ banana graph has $g+1$ strands, hence such graphs occupy a $(g+1)$-dimensional stratum in the moduli space, i.e. a stratum of codimension $(3g-3) -(g+1) = 2g-4$. The codimension is the same after marking two points. So for $g > 2$ banana graphs are ``special.''

Another important way in which banana graphs are special is that they are always hyperelliptic: they possess a degree $2$ divisor of rank $1$, consisting of the two non-bivalent vertices. For $g \geq 3$ this shows that they are not Brill--Noether general, since $\rho(g,1,2) = 2-g$.
\end{rema}

\subsection{Revisions to the published manuscript and AI disclosure}

The present manuscript is a revision of the manuscript published in \emph{Algebraic Combinatorics} Volume 8 (2025) no. 5, pp. 1415--1457, incorporating several corrections discovered after publication. Most importantly, we have accounted for the following exceptional family: twice-marked banana graphs in which one marked point is the midpoint of a length-$2$ strand, and the other marked point is interior to a different strand. This exceptional family is an edge case in the following results, which have been revised accordingly: \Cref{thm:bananaSimple}, \Cref{thm-NSMForBanana}, \Cref{prop-bananTorsion}, and \Cref{cor:bananasWithKGT}. We have also added \Cref{ssec:midpointfamily} to address $k$-general transmission in the exceptional family. We have also corrected some formulas and their ranges in \Cref{ssec:mostBananas}, and several other misprints. The paper's main results, as corrected herein, have now been formally verified in Lean; see the repository
\url{https://github.com/npflueger/brill-noether-graphs} for details.
In addition, our construction of Brill--Noether general graphs as chains of cycles and thetas is now registered at the Palomar registry \cite{palomar-bnchains}.

Generative AI was used to aid this revision, as follows. After discovering the missing edge case, we used AI tools to identify the exceptional family, the necessary corrections to the manuscript, and several misprints. The formal verification in Lean was also carried out primarily by AI. The manuscript itself was revised by the authors directly. No AI tools were used in the development or preparation of the original published manuscript.

\section{Background}
\label{sec:background}

\subsection{Conventions and useful facts}\label{ssec:conventions}
We opt for the conventions that $[k]$ is the set $\set{0,\dots,k-1}$ and $\N:=\set{1,2,\dots}$. 

\begin{defi}[Graph]\label{def-Graph}
We assume the convention that a \textit{graph} $G$ is finite, connected, and loopless, possibly with parallel edges. The set of vertices will be denoted $V(G)$ and the set of edges $E(G)$. The \textit{valence} of a vertex, denoted $\val{v}$ is the number of edges incident to $v$. We refer to vertices $v$ with $\val{v} \geq 3$ as \emph{multivalent vertices}. We take the genus $g$ of a graph to be $\#E(G)- \#V(G)+1$, i.e. the number of edges that must be removed to produce a spanning tree.
\end{defi}

We next give a very brief crash course on divisor theory for finite graphs. The reader interested in a more comprehensive treatment should consult \cite{corry2018divisors} or \cite{jensenNotices}. Given a finite graph $G$ the group of divisors $\mathrm{Div}(G)$ is the free abelian group $\Z^{V(G)}$ on the vertices. 
A divisor is commonly regarded as a \emph{chip configuration,} i.e. an allocation of poker chips to the vertices, in which negative chips, interpreted as debt, are permitted. The \emph{degree} of a divisor is the total number of chips. 
Taking the quotient by the image of the graph Laplacian yields the Picard group $\pic{G}$. We refer to two divisors in the same class of $\pic{G}$ as being linearly equivalent and write $D_1 \sim D_2$. Any divisor equivalent to $0$ is called principal. 
Linearly equivalent divisors have the same degree, so $\Pic(G)$ is graded by degree into cosets $\Pic^d(G)$.
The degree $0$ coset is a subgroup called the Jacobian, written $\jac{G}$. A divisor is called effective if all coefficients are non-negative. In other words, an effective divisor is a chip configuration with no debt. We write $D \geq 0$ to mean that $D$ is effective, or more generally $D \geq E$ to mean that $D-E$ is effective, and $|D|$ for $\{E \ge 0 : E \sim D\}$. There is a function $r:\opdiv{G} \to \Z_{\geq -1}$ which is characterized as follows. We let $r(D)$ be $-1$ on any divisor not equivalent to any effective divisor. If for every effective divisor $E$ of degree $e$, $D-E$ is equivalent to an effective divisor, then $r(D) \geq e$. A divisor is called special if $r(D) > \max{\deg(D)-g, -1}$ and non-special otherwise. There is a distinguished divisor called the canonical divisor, written $K_G$, which is given by $\sum_{v \in V(G)}(\val{v}-2)v$. We refer to $D$ and $K_G-D$ as dual divisors. In \cite{BakNor07}, the authors proved the fundamental Riemann--Roch theorem relating the rank of a divisor to that of its dual:
\begin{align*}
    r(D) - r(K_G-D) = \deg(D)-g + 1.
\end{align*}

\begin{defi}
    If $(G_1,u_1,v_1),(G_2,u_2,v_2)$ are twice-marked graphs, we may obtain a new twice-marked graph $(G,u_1,v_2)$ by taking the disjoint union of $G_1$ and $G_2$ and identifying $v_1$ and $u_2$. We refer to the new graph as the \textit{vertex gluing of $(G_1,u_1,v_1)$ and $(G_2,u_2,v_2)$}.
    
    Given a sequence $\Big( (G_i, u_i, v_i) \Big)_{1 \leq i \leq \ell}$ of $\ell$ twice-marked graphs, the \emph{iterated vertex gluing} is the graph obtained by taking $(G_1,u_1,v_1)$, then successively forming the vertex gluing with $(G_2, u_2, v_2), \ldots ,(G_\ell, u_\ell, v_\ell)$.
\end{defi}

From the standpoint of Brill--Noether theory on finite graphs, it suffices to study graphs which are bridgeless (i.e. $2$-edge connected). Every graph $G$ admits a retraction to a bridgeless graph which gives a rank preserving isomorphism between the Picard groups, so we are not missing anything by focusing solely on bridgeless graphs. A convenience of this setting is the following lemma.

\begin{lemma}\label{lem-bridgelessFacts}
    If $G$ is a bridgeless graph of positive genus, then
    \begin{enumerate}[label = \arabic*)]
        \item If $u,v \in V(G)$ then $u = v$ if and only if $u \sim v$;
        \item There is a bijection between rank $0$ divisors in $\Pic^1(G)$ and vertices in $V(G)$.
    \end{enumerate} 
\end{lemma}

Throughout this paper, we will sometimes replace graphs with their bridgeless contraction without further comment. In particular, when performing vertex gluing, we may just as well attach the glued vertices by a path of any length (see e.g. \Cref{fig:bngenlexample}) with no meaningful changes to the divisor theory.

\begin{defi}[Banana-Graph]\label{def-BanGraph}
    For $n_0,\dots,n_g \in \N$ we define the \textit{banana graph} $B_{n_0,\dots,n_g}$ to be the graph obtained by connecting two vertices with $g+1$ paths of length $n_0,\dots,n_g$. The resulting graph has genus $g$. In the genus $1$ case this is simply a cycle graph of length $n_0+n_1$. In the case where $g = 2$ we refer to such graphs as \textit{theta graphs} and denote them $\theta_{n_0,n_1,n_2}$. Label one of the multivalent vertices $v_{0,0}$ and the other $v_{0,n_0}$. Label the vertex a distance $i$ from $v_{0,0}$ along the $\alpha$-th path as $v_{\alpha,i}$. Note the multivalent vertices have more than one label:
    \begin{align*}
        v_{\alpha,0} = v_{\beta,0} \text{ and } v_{\alpha,n_\alpha} = v_{\beta,n_\beta} \forall \alpha,\beta \in \set{0,\dots,g}.
    \end{align*}
    All other vertices are uniquely labeled. We define $\overline{v_{\alpha,i}}:= v_{\alpha,n_\alpha-i}$ (in the genus $2$ case, this coincides with the usual dual $K_G-v_{\alpha,i}$ up to linear equivalence). In the special case of theta graphs we adopt the notation that $x_i:= v_{0,i}, y_i := v_{1,i},$ and $z_i:= v_{2,i}$ to simplify typography. We refer to the collection of vertices $v_{\alpha,0},v_{\alpha,1},\dots,v_{\alpha,n_\alpha}$ as the $\alpha$-th \textit{strand} of  our graph. The multivalent vertices are thus on all strands simultaneously.
\end{defi}

\begin{defi}\label{def-delt}
We define $\delta(P)$ for a proposition $P$ to be the indicator function which is $1$ when $P$ holds and $0$ when $P$ does not.
\end{defi}

\begin{defi}\label{def-TwMkGraph}
A \textit{twice-marked} graph $(G,u,v)$ is a graph with a choice of two distinguished vertices $u,v$. The torsion order of $(G,u,v)$ is the order of $[u-v]$ as an element of $\jac{G}$, i.e. the minimum $k \in \N$ such that $ku \sim kv$.
\end{defi}

\begin{defi}\label{def-Twist}
A \textit{twist} of a divisor $D$ on a twice-marked graph $(G,u,v)$ is a divisor of the form $D+au+bv$, $a,b \in \Z$.
\end{defi}

\begin{defi}\label{def-Delt}
For a divisor $D$ on $(G,u,v)$ we define
\begin{align*}
    \Delta(D):=r(D) - r(D-u) - r(D-v) + r(D-u-v).
\end{align*}
Although suppressed in the notation, this function depends on the choice of marked points $u,v$.
\end{defi}

\begin{defi}[\cite{Pfl22}]\label{def-submod}
A divisor $D$ on $(G,u,v)$ is \textit{submodular} with respect to $u,v$ if $\Delta(D')\geq 0$ for all twists $D' = D+au + bv$.
\end{defi}

\begin{defi}\label{def-EA}
    A \textit{permutation} is a bijection $\tau:\Z\to \Z$. Given $k\in \N$, permutations satisfying $\tau(n+k)= \tau(n) + k$ for all $n \in \Z$ form a group denoted $\eaf{k}$, referred to as the \textit{extended affine symmetric group}.
\end{defi}

\begin{defi}[\cite{Pfl22}]\label{def-tauD}
    Given a twice-marked graph $(G,u,v)$, let $D$ be a divisor in $\opdiv{G}$. If it exists, the \emph{transmission permutation} of $D$, denoted $\tauD$ is the unique bijection $\tauD:\ \Z \to \Z$ which satisfies, for all $a,b \in \Z$,
    \begin{align*}
        \delta(\tauD(b) = a) = \Delta(D+au-bv).
    \end{align*}
\end{defi}

\begin{lemma}[{\cite[Remark 1.5, Proposition 2.3]{Pfl22}}] \label{lem:tauChars} 
A divisor $D$ on $(G,u,v)$ has a well-defined transmission permutation if and only if it is submodular. If $(G,u,v)$ has torsion order $k$, or more generally if $ku \sim k v$, then $\tauD \in \eaf{k}$. The transmission permutation is also characterized by either of the following two equivalent equations. For all $a,b \in \Z$:
\begin{eqnarray*}
r(D + au - bv) + 1 &=& \# \{ \ell \geq b:\ \tauD(\ell) \leq a \},\\
r(K_G - D - au + bv) + 1 &=& \# \{ \ell < b:\ \tauD(\ell) >a \}.
\end{eqnarray*}
\end{lemma}

\begin{defi}[\cite{Pfl22}]\label{def-inv}
    Given a permutation $\tau$, an inversion is a pair $(a,b) \in \Z^2$ such that $a<b$ and $\tau(a)>\tau(b)$. For a given $k \in \Z$, we say that two inversions $(a,b),(a',b')$ are $k$\textit{-equivalent} if $a-a' = b-b'$ and $a \equiv a' \mod k$. We write $\Inv_k(\tau)$ for the set of $k$-equivalence classes of inversions of $\tau$ and $\inv_k(\tau)$ for $\# \Inv_k(\tau)$. We call the equivalence classes \textit{$k$-inversions}.
\end{defi}

\subsection{Jacobians of Banana Graphs}
In this section we characterize the Jacobians of Banana graphs of genus $\geq 2$.

\begin{prop}\label{prop-JacBanana}
    Given $n_0,\dots,n_{g} \in \Z_{\ge 1}$, there is an isomorphism
    \begin{align*}
        \jac{B_{n_0,\dots,n_g}}\cong &\Z^{g+1}/\gen{(1,1,\dots,1),(n_0,-n_1,0,\dots),(n_0,0,-n_2,0,\dots,0),\\&\dots,(n_0,0,\dots,0,-n_{g})}.
    \end{align*}
    under which the coset $[a_0, \ldots , a_g]$ of $(a_0, \ldots, a_g) \in \Z^{g+1}$ is identified with the divisor class $\sum_{\alpha=0}^g [ v_{\alpha,a_\alpha} - v_{0,0} ]$, assuming that $0 \leq a_\alpha \leq n_\alpha$ for all $\alpha$ (so that $v_{\alpha,a_\alpha}$ is a well-defined vertex).
\end{prop}

\begin{proof}
\newcommand{\tphi}{\widetilde{\phi}}
    Here and throughout the sequel we write $(a_0,\dots,a_g)$ for an element of $\Z^{g+1}$ and $[a_0,\dots,a_g]$ for an element of the quotient. Let $J:= \jac{B_{n_0,\dots,n_g}}$ and $$K:=\gen{(1,1,\dots,1),(n_0,-n_1,0,\dots),(n_0,0,-n_2,0,\dots,0),\dots,(n_0,0,\dots,0,-n_g)}.$$ 
    We claim that there is an isomorphism $\phi:\Z^{g+1}/K \to J$ characterized by the formula 
     \begin{align} \label{eq:phiDefn}
        [a_0,\dots,a_g]\mapsto \sum_{\alpha = 0}^g a_\alpha [(v_{\alpha,1}-v_{0,0})].
    \end{align}
    We must check that $\phi$ is well-defined, injective, and surjective. We check all three by first considering a map 
    $\tphi: \Z^{g+1} \to \operatorname{Div}^0(G)$ defined by $(a_0, \ldots, a_g) \mapsto \sum_{\alpha=0}^g a_\alpha (v_{\alpha,1} - v_{0,0})$. This is transparently a homomorphism. If we verify that $\tphi$ sends all generators of $K$ to principal divisors, then it will follow that $\tphi$ induces a well-defined homomorphism $\phi$ obeying the desired formula. To see this, first note that $\tphi(1,\ldots,1)$ is the principal divisor obtained by a single chip-firing move at $v_{0,0}$. Second, note that we have
    \begin{align}\label{eq:multDiff}
    a_\alpha (v_{\alpha,1} - v_{0,0}) \sim v_{\alpha, a_\alpha} - v_{0,0} \mbox{ for all } 0 \leq a_\alpha \leq n_\alpha.
    \end{align}
    This follows by induction and the observation that $v_{\alpha,i+1} - v_{\alpha,i} \sim v_{\alpha,i} - v_{\alpha, i-1}$ for all $1 \leq i \leq n_\alpha-1$, which results from chip-firing the vertex $v_{\alpha,i}$. We can explicitly characterize the set of chip-firing moves realizing this equivalence as follows: for each vertex $i \in \set{1,\dots,a_\alpha-1}$, chip-fire the vertex $v_{\alpha,i}$ $(a_\alpha-i)$-many  times. This implies that, if $e_\alpha \in \Z^{g+1}$ denotes the $\alpha$-th standard basis vector, we have $\tphi(n_\alpha e_\alpha) \sim v_{\alpha, n_\alpha} - v_{0,0} = v_{0,n_0} - v_{0,0}$. This does not depend on $\alpha$, so for any two strands $\alpha, \beta$ we have $\tphi(n_\alpha e_\alpha - n_\beta e_\beta) \sim 0$. It follows that all other generators of $K$ are sent by $\tphi$ to principal divisors. Hence $\tphi$ induces a well-defined homomorphism $\phi: \Z^{g+1} / K \to J$ obeying Equation \eqref{eq:phiDefn}.
    
    To check injectivity, we use useful coset representatives of $\Z^{g+1}/K$. Every element of $\Z^{g+1}/K$ has a representative $(a_0,\dots,a_g)$ (in fact, a unique representative) satisfying the following three conditions.
    \begin{enumerate}[label = \arabic*)]
        \item $0\leq a_\alpha \leq n_\alpha$ for all $\alpha \in \set{0,\dots,g}$,
        \item there exists $\alpha$ such that $a_\alpha = 0$, and
        \item If $a_\alpha = n_\alpha$ and $a_\beta = 0$, then $\alpha < \beta$.
    \end{enumerate}
    Call an element $(a_0, \ldots, a_g)$ satisfying these conditions \emph{reduced}.
    To reduce an arbitrary element of $\Z^{g+1}$ to an element obeying the first two conditions, one can repeatedly follow these steps: add a multiple of $(1,\ldots,1)$ so that all coordinates are nonnegative and at least one is zero; reduce any coordinate $a_\alpha > n_\alpha$ by $n_\alpha$ while replacing a coordinate $a_\beta = 0$ by $n_\beta$. Then to satisfy the third, one can ``left-justify'' all coordinates with $a_\alpha = n_\alpha$ while moving $0$s to the right.
    
    Now, if $(a_0, \ldots, a_g)$ is reduced, Equation \eqref{eq:multDiff} implies that
    \begin{align} \label{eq:phiReduced}
    \phi( [a_0, \ldots, a_g]) = \sum_{\alpha=0}^g [ v_{\alpha, a_\alpha} - v_{0,0} ].
    \end{align}
    It follows from Dhar's burning algorithm \cite{Dhar90} (or \cite[\S 3.4]{corry2018divisors} for a more recent treatment) that $\sum_{\alpha=0}^g (v_{\alpha, a_\alpha} - v_{0,0})$ is in fact a $v_{0,0}$-reduced divisor: the number of chips at $v_{0,n_0}$ is strictly less than the number of strands with no interior chips, so both multivalent vertices burn; all other strands have at most one chip in their interior and burn from both ends. Therefore this divisor is not principal unless $a_0 = a_1 = \ldots a_g = 0$. This shows that $\phi$ is injective. 
    
    Finally, note that $J$ is generated by the classes $[v_{\alpha,a_\alpha} - v_{0,0}]$, which are all in the image of $\phi$ by Equation \eqref{eq:phiReduced}. So $\phi$ is surjective, and therefore an isomorphism.
\end{proof}

\begin{rema}
    While we require that each strand must have non-zero length to make later arguments easier, it's worth considering the limiting case where one of the strands has length $0$. Combinatorially this should correspond to passing from a banana graph of genus $g$ to a bouquet of $g$ cycles. Indeed, if $n_0 = 0$ then we get
    \begin{align*}
        \Z^{g+1}/\gen{(1,\dots,1),(0,n_1,0,\dots,0),(0,0,n_2,0,\dots,0),\dots,(0,\dots,0,n_g)} \cong \prod_{\alpha = 1}^g\Z/n_\alpha\Z.
    \end{align*}
    Recall that the Jacobian of a graph $G$ is isomorphic as a group to the direct product of the Jacobians of each of the $2$-vertex connected components, and that cycles have cyclic Jacobians of order equal to their length. Hence taking the limiting case of our presentation of $\jac{B_{n_0,\dots,n_g}}$ as $n_0 \to 0$ produces a presentation for the Jacobian of the bouquet of $g$ cycles.
\end{rema}

In genus $2$, we can restrict even further and go down another dimension.

\begin{prop} \label{prop-JacTheta}
    For a theta graph $\theta_{a,b,c}$ we have $\jac{\theta_{a,b,c}} \cong \Z^2/\gen{(a+c,c),(-a,b)}$.
\end{prop}

\begin{proof}
    
As we established above,
\begin{align*}
    \jac{\theta_{a,b,c}} \cong \Z^3/\gen{(a,0,-c),(a,-b,0),(1,1,1)}.
\end{align*}

There is a map $\phi:\Z^3 \to \Z^2/\gen{(a+c,c),(-a,b)}$ given by $(x,y,z)\mapsto (x-z,y-z)$. Clearly $$N:=\gen{(a,0,-c),(a,-b,0),(1,1,1)} \subseteq \ker{\phi}$$ so $\phi$ descends to a surjection from the quotient: $\Tilde{\phi}:\Z^3/N\to \Z^2/\gen{(a+c,c),(-a,b)}$. For injectivity note that if $(x,y,z) \in \ker{\phi}$ then this implies that 
\begin{align*}
    (x-z,y-z) = m(a,-b)+n(a+c,c).
\end{align*}
From this we can write:
\begin{align*}
    (x,y,z) &= (z+ma+n(a+c),z-mb+nc,z) \\&= (z+nc)(1,1,1) + m(a,-b,0)+n(a,0,-c) \in N.
\end{align*}
\end{proof}

\begin{defi}[Abel--Jacobi Map]
For a graph $G$ with a fixed base point $v_0$ we define the \textit{Abel--Jacobi map} $S_{v_0}:V(G)\to \jac{G}$ by $v \mapsto [v-v_0]$.
\end{defi}

\begin{rema}
    An advantage of this two-dimensional discrete torus presentation of the Jacobian is that it can be easily identified with the lattice points of a parallelogram as in \Cref{ParFunDomain} or a hexagon as in \Cref{fig:hexFunDom}, each with the opposing parallel sides identified.
    These figures can also be inferred from Figure $1$ of \cite{ABKS}; in particular the hexagon presentation is naturally divided into three parallelograms corresponding to three different types of degree $2$ break divisors.
    In both representations, it is easy to see a copy of $V(G)$ sitting inside of $\jac{G}$ through the Abel--Jacobi map, first studied in \cite{BakNor07}.
    Thus the number of elements of the Jacobian is given by counting lattice points inside these polygons, accounting for identified boundary points. By Pick's theorem, this value equals the area of either of these polygons, namely $ab+ ac+bc$. These identification space representations of the Jacobian are helpful for thinking about progressions of the form $n(u-v)$ which correspond to straight paths along the torus. 
\end{rema}

\begin{figure}
    \centering
    \pgfmathsetmacro\a{3}
    \pgfmathsetmacro\b{4}
    \pgfmathsetmacro\c{5}
    \pgfmathsetmacro\w{int(\a + \a + \c)}
    \pgfmathsetmacro\h{int(\b + \c)}
    \pgfmathsetmacro\adec{int(\a-1)}
    \pgfmathsetmacro\bdec{int(\b-1)}
    \pgfmathsetmacro\cdec{int(\c-1)}
    \begin{tikzpicture}
    \begin{scope}[scale = 0.75] 
        \node [inner sep=0pt](1) at (0,0) [label = {270:$(0,0)$}]{};
        \node [inner sep=0pt] (2) at (-\a,\b) [label = {180:$(-a,b)$}]{};
        \node [inner sep=0pt](3) at (\c,\b+\c) [label = {90:$(c,b+c)$}]{};
        \node [inner sep=0pt](4) at (\a+\c,\c) [label = {0:$(a+c,c)$}]{};
    
        \foreach \x in {0,...,\a} {
            \ifnum\x=\a{
                \node[right] (x\x) at (-\a +\x,\b) {$x_{\x} = y_{\b} = z_{\c}$};
            }
            \else{
                \ifnum\x=0{
                    \node (x\x) at (-\a +\x,\b) {};
                }
                \else{
                    \node (x\x) at (-\a +\x,\b) {$x_{\x}$};
                }\fi  
        	   
           }\fi
        }
        \foreach \x [evaluate = \x as \xinc using int(\x+1)]in {0,...,\adec}{
            \draw [dotted,thick] (x\x) -- (x\xinc);
        }
        
        \foreach \x in {0,...,\b} {
            \ifnum\x=\b{
                \node (y\x) at (0,\x) {$ $};
            }
            \else{
                \ifnum\x=0{
                    \node (y\x) at (0,\x) {};
                }
                \else{
                    \node (y\x) at (0,\x) {$y_{\x}$};
                }\fi  
        	   
           }\fi
        }
    
        \foreach \x [evaluate = \x as \xinc using int(\x+1)]in {0,...,\bdec}{
            \draw [dotted,thick] (y\x) -- (y\xinc);
        }
        
        \foreach \x in {0,...,\c} {
            \ifnum\x=\c{
                \node (z\x) at (\c -\x,\b + \c - \x) {$ $};
            }
            \else{
                \ifnum\x=0{
                    \node (z\x) at (\c -\x,\b + \c - \x) {};
                }
                \else{
                    \node (z\x) at (\c -\x,\b + \c - \x) {$z_{\x}$};
                }\fi  
        	   
           }\fi
        }
    
        \foreach \x [evaluate = \x as \xinc using int(\x+1)]in {0,...,\cdec}{
            \draw [dotted,thick] (z\x) -- (z\xinc);
        }
        
        \foreach \x in {0,...,\w}{
            \foreach \y in {0,...,\h}{
                \node (\x+\y+50) at (\x -\a,\y-0.0075) {$\cdot$}; 
            }
        }
        \draw[opacity=0.5,fill = ocean] (1.center)--(2.center)--(3.center)--(4.center)--(1.center);
    
        \draw [thick](1.center) -- (2)
        (1.center) -- (4);
        \draw [thick,dashed] (2.center) -- (3.center)
        (3.center) -- (4.center);
    
        \draw[->,opacity = 0.25] (1) -- (\a+\c+1,0);
        \draw[->,opacity = 0.25] (1) -- (-\a-1,0);
        \draw[->,opacity = 0.25] (1) -- (0,\b+\c+1);
    \end{scope}
    \end{tikzpicture}
    \caption{A fundamental domain of $\jac{\theta_{\a,\b,\c}}$. The dashed upper left and upper right side indicate identification with the corresponding opposite side. The Abel--Jacobi embedding $S_{x_0}:V(G) \hookrightarrow \jac{G}$ is also visualized through the labeled vertices $x_i,y_i,z_i$ and dotted edges connecting them.}
    \label{ParFunDomain}
\end{figure}

\begin{figure}
    \pgfmathsetmacro\a{3}
    \pgfmathsetmacro\b{4}
    \pgfmathsetmacro\c{5}
    \pgfmathsetmacro\w{int(\a + \c)}
    \pgfmathsetmacro\h{int(\b + \c)}
    \pgfmathsetmacro\adec{int(\a-1)}
    \pgfmathsetmacro\bdec{int(\b-1)}
    \pgfmathsetmacro\cdec{int(\c-1)}
    \newif\ifdrawedges
    \drawedgestrue
    \centering
    \begin{tikzpicture}

    \begin{scope}[scale = 0.75 ]
    \node (O) at (0,0) {};
    \draw[->,opacity = 0.25] (O) -- (\c+1,0);
    \draw[->,opacity = 0.25] (O) -- (-\a-1,0);
    \draw[->,opacity = 0.25] (O) -- (0,\b+\c+1);

        \foreach \x in {0,...,\w}{
            \foreach \y in {0,...,\h}{
                \node (\x+\y+50) at (\x -\a,\y-0.0075) {$\cdot$}; 
            }
        }
    \begin{scope}[shift= {(-\a,0)}]
        
        \node [inner sep=0pt](hex1) at (0,0) [label = {215:$(-a,0)$}]{};
        \node [inner sep=0pt](hex2) at (0,\b) [label = {225:$(-a,b)$}]{};
        \node [inner sep=0pt](hex3) at (\c,\b+\c) [label = {135:$(c-a,b+c)$}]{};
        \node [inner sep=0pt](hex4) at (\a+\c,\b+\c) [label = {45:$(c,b+c)$}]{};
        \node [inner sep=0pt](hex5) at (\a+\c,\c) [label = {0:$(c,c)$}]{};
        \node [inner sep=0pt](hex6) at (\a,0) [label = {320:$(0,0)$}]{};

        \draw[opacity=0.5,fill = ocean] (hex1.center)--(hex2.center)--(hex3.center)--(hex4.center)--(hex5.center)--(hex6.center)--(hex1.center);

        \foreach \x in {0,...,\a} {
            \ifnum\x=\a{
                \node[right] (x\x) at (\x,\b) {$x_{\x} = y_{\b} = z_{\c}$};
            }
            \else{
                \ifnum\x=0{
                    \node (x\x) at (\x,\b) {};
                }
                \else{
                    \node (x\x) at (\x,\b) {$x_{\x}$};
                }\fi  
               
           }\fi
        }

        \foreach \x in {0,...,\b} {
            \ifnum\x=\b{
                \node (y\x) at (\a,\x) {$ $};
            }
            \else{
                \ifnum\x=0{
                    \node (y\x) at (\a,\x) {};
                }
                \else{
                    \node (y\x) at (\a,\x) {$y_{\x}$};
                }\fi  
               
           }\fi
        }
        
        \foreach \x in {0,...,\c} {
            \ifnum\x=\c{
                \node (z\x) at (\a+\c -\x,\b + \c - \x) {$ $};
            }
            \else{
                \ifnum\x=0{
                    \node (z\x) at (\a+\c -\x,\b + \c - \x) {};
                }
                \else{
                    \node (z\x) at (\a+\c -\x,\b + \c - \x) {$z_{\x}$};
                }\fi  
               
           }\fi
        }
        \ifdrawedges{%
            \foreach \x [evaluate = \x as \xinc using int(\x+1)]in {0,...,\adec}{
                \draw [dotted,thick] (x\x) -- (x\xinc);
            }
            \foreach \x [evaluate = \x as \xinc using int(\x+1)]in {0,...,\bdec}{
                \draw [dotted,thick] (y\x) -- (y\xinc);
            }
            \foreach \x [evaluate = \x as \xinc using int(\x+1)]in {0,...,\cdec}{
                \draw [dotted,thick] (z\x) -- (z\xinc);
            }
        }\fi
        \draw [thick](hex1.center) -- (hex6.center)
        (hex6.center) -- (hex5.center)
        (hex5.center) -- (hex4.center);
        \draw [thick,dashed] (hex1.center) -- (hex2.center)
        (hex2.center) -- (hex3.center)
        (hex3.center) -- (hex4.center);
    \end{scope}
    \end{scope}
    \end{tikzpicture}
    \caption{A depiction of an alternative fundamental domain of $\jac{\theta_{\a,\b,\c}}$. As above the dashed sides indicate identification with the corresponding opposite side and the Abel--Jacobi embedding $S_{x_{0}}$ is visualized inside the diagram.}
    \label{fig:hexFunDom}
\end{figure}

\begin{figure}
    \centering
    \pgfmathsetmacro\a{3}
    \pgfmathsetmacro\b{4}
    \pgfmathsetmacro\c{5}
    \pgfmathsetmacro\w{int(2*\a + \c))}
    \pgfmathsetmacro\h{int(\b + \c)}
    \pgfmathsetmacro\adec{int(\a-1)}
    \pgfmathsetmacro\bdec{int(\b-1)}
    \pgfmathsetmacro\cdec{int(\c-1)}
    \newif\ifdrawedges
    \drawedgestrue
    \pgfmathsetmacro\spacing{8.5}
    \begin{tikzpicture}
    \begin{scope}[scale = 0.35]
    
        \node [inner sep=0pt](1) at (0,0) [label = {270:$(0,0)$}]{};
        \node [inner sep=0pt] (2) at (-\a,\b) [label = {180:$(-a,b)$}]{};
        \node [inner sep=0pt](3) at (\c,\b+\c) [label = {135:$(c,b+c)$}]{};
        \node [inner sep=0pt](4) at (\a+\c,\c) [label = {45:$(a+c,c)$}]{};
        \node [inner sep=0pt](hex2) at (0,\b) [label = {0:$(0,b)$}]{};
    
        \draw[opacity=0.5,fill = shallows] (1.center)--(hex2.center)--(3.center)--(4.center)--(1.center);
        \draw[opacity=0.5,fill = dust] (2.center)--(3.center)--(hex2.center) --(2.center);
        \draw[opacity=0.5,fill = ocean] (2.center)--(hex2.center)--(1.center) --(2.center);

        \draw[very thick,<->] (\w-\a+.25*\spacing,\h/2) -- (\w -\a+.75*\spacing,\h/2);

        \foreach \x in {0,...,\a} {
            \ifnum\x=\a{
                \node (x\x) at (-\a +\x,\b) {$x_{\x}$};
            }
            \else{
                \ifnum\x=0{
                    \node (x\x) at (-\a +\x,\b) {};
                }
                \else{
                    \node (x\x) at (-\a +\x,\b) {$x_{\x}$};
                }\fi  
        	   
           }\fi
        }

        \foreach \x in {0,...,\b} {
            \ifnum\x=\b{
                \node (y\x) at (0,\x) {$ $};
            }
            \else{
                \ifnum\x=0{
                    \node (y\x) at (0,\x) {};
                }
                \else{
                    \node (y\x) at (0,\x) {$y_{\x}$};
                }\fi  
        	   
           }\fi
        }
        
        \foreach \x in {0,...,\c} {
            \ifnum\x=\c{
                \node (z\x) at (\c -\x,\b + \c - \x) {$ $};
            }
            \else{
                \ifnum\x=0{
                    \node (z\x) at (\c -\x,\b + \c - \x) {};
                }
                \else{
                    \node (z\x) at (\c -\x,\b + \c - \x) {$z_{\x}$};
                }\fi  
        	   
           }\fi
        }
        \ifdrawedges{%
            \foreach \x [evaluate = \x as \xinc using int(\x+1)]in {0,...,\adec}{
                \draw [dotted,thick] (x\x) -- (x\xinc);
            }
            \foreach \x [evaluate = \x as \xinc using int(\x+1)]in {0,...,\bdec}{
                \draw [dotted,thick] (y\x) -- (y\xinc);
            }
            \foreach \x [evaluate = \x as \xinc using int(\x+1)]in {0,...,\cdec}{
                \draw [dotted,thick] (z\x) -- (z\xinc);
            }
        }\fi

        \begin{scope}[shift = {(\w-\a+\spacing,0)}]
            \node [inner sep=0pt](hex1) at (0,0) [label = {215:$(-a,0)$}]{};
            \node [inner sep=0pt](hex2) at (0,\b) [label = {225:$(-a,b)$}]{};
            \node [inner sep=0pt](hex3) at (\c,\b+\c) [label = {135:$(c-a,b+c)$}]{};
            \node [inner sep=0pt](hex4) at (\a+\c,\b+\c) [label = {45:$(c,b+c)$}]{};
            \node [inner sep=0pt](hex5) at (\a+\c,\c) [label = {0:$(c,c)$}]{};
            \node [inner sep=0pt](hex6) at (\a,0) [label = {320:$(0,0)$}]{};

            \draw[opacity=0.5,fill = shallows] (hex1.center)--(hex2.center)--(hex3.center)--(hex5.center)--(hex1.center);
            \draw[opacity=0.5,fill = dust] (hex1.center)--(hex5.center)--(hex6.center) --(hex1.center);
            \draw[opacity=0.5,fill = ocean] (hex3.center)--(hex4.center)--(hex5.center) --(hex3.center);

            \foreach \x in {0,...,\a} {
                \ifnum\x=\a{
                    \node (x\x) at (\x,\b) {$x_{\x}$};
                }
                \else{
                    \ifnum\x=0{
                        \node (x\x) at (\x,\b) {};
                    }
                    \else{
                        \node (x\x) at (\x,\b) {$x_{\x}$};
                    }\fi  
            	   
               }\fi
            }

            \foreach \x in {0,...,\b} {
                \ifnum\x=\b{
                    \node (y\x) at (\a,\x) {$ $};
                }
                \else{
                    \ifnum\x=0{
                        \node (y\x) at (\a,\x) {};
                    }
                    \else{
                        \node (y\x) at (\a,\x) {$y_{\x}$};
                    }\fi  
            	   
               }\fi
            }
            
            \foreach \x in {0,...,\c} {
                \ifnum\x=\c{
                    \node (z\x) at (\a+\c -\x,\b + \c - \x) {$ $};
                }
                \else{
                    \ifnum\x=0{
                        \node (z\x) at (\a+\c -\x,\b + \c - \x) {};
                    }
                    \else{
                        \node (z\x) at (\a+\c -\x,\b + \c - \x) {$z_{\x}$};
                    }\fi  
            	   
               }\fi
            }
            \ifdrawedges{%
                \foreach \x [evaluate = \x as \xinc using int(\x+1)]in {0,...,\adec}{
                    \draw [dotted,thick] (x\x) -- (x\xinc);
                }
                \foreach \x [evaluate = \x as \xinc using int(\x+1)]in {0,...,\bdec}{
                    \draw [dotted,thick] (y\x) -- (y\xinc);
                }
                \foreach \x [evaluate = \x as \xinc using int(\x+1)]in {0,...,\cdec}{
                    \draw [dotted,thick] (z\x) -- (z\xinc);
                }
            }\fi
        \end{scope}
    \end{scope}
    
    \end{tikzpicture}
    \caption{The transformation between the parallelogram and hexagon fundamental domains for $\jac{\theta_{\a,\b,\c}}$ along with the Abel--Jacobi embedding $S_{x_0}$. Note that this transformation has two steps, the relocation of the triangles along the vectors $(-a,b)$ and $(a+c,c)$ respectively, and the translation of the whole diagram along $(a,0)$. As above opposite sides of each diagram are identified.}
    \label{fig:ParToHex}
\end{figure}

\subsection{Rank Computations on Banana Graphs}

In this section we develop some results which expedite rank computations on banana graphs. We first recall the notion of a rank determining set, introduced by Luo in \cite{RDS}.

\begin{defi}
    Given a set $A \subseteq V(G)$ we define $r_A(D)$ to be $-1$ if $|D| = \emptyset$ and $r_A(D) \geq r$ if $|D-E| \neq \emptyset$ for every effective divisor of degree $r$ supported on $A$. A set $A$ is a \textit{rank determining set} if $r_A(D) = r(D)$ for all divisors $D$.
\end{defi}

\begin{lemma}\label{lem:g12}
For any banana graph $G= B_{n_0,\dots,n_g}$ with $g \ge 1$, the divisor $v_{0,0} + v_{0,n_0}$ has rank $1$.
\end{lemma}

\begin{proof}
For all $\alpha$ and indices $i,j$ with $0 \leq i < j \leq n_\alpha$, chip-firing the set $\{v_{\alpha,k}:\ i <  k < j \}$  shows that 
$$v_{\alpha,i+1} + v_{\alpha,j-1} \sim v_{\alpha,i} + v_{\alpha,j}.$$
It follows by induction that, for all $0 \leq i \leq n_\alpha$, we have
$$v_{0,0} + v_{0,n_0} = v_{\alpha,0} + v_{\alpha,n_\alpha} \sim v_{\alpha,i} + v_{\alpha,n_\alpha-i}.$$
Hence every vertex $v_{\alpha,i}$ is supported on some divisor linearly equivalent to $v_{0,0} + v_{0,n_0}$, so $r(v_{0,0} + v_{0,n_0}) \ge 1$. The rank is at most $1$ since $v_{0,0} \not\sim v_{0,n_0}$.
\end{proof}

\begin{lemma}\label{lem-BananaRDS}
    For any banana graph $G= B_{n_0,\dots,n_g}$ the set $\set{v_{0,0},v_{0,n_0}}$ is a rank determining set.
\end{lemma}

\begin{proof}
    Let $W:=\set{v_{0,0},v_{0,n_0}}$. By \cite[Proposition 3.1]{RDS} the lemma is equivalent to the claim that $r_W(D) \geq 1$ implies $r(D)\geq 1$ for all $D \in \opdiv{G}$. Fix a divisor $D$ and suppose $r_W(D) \geq 1$. If $d:=\deg D >g$ then by Riemann--Roch we have that $r(D) \geq d-g \geq 1$, so assume $d\leq g$. 
    
    We claim that our assumption that $r_W(D)\geq 1$ implies that $D \sim E + v_{0,0} + v_{0,n_0}$, where $E$ is an effective divisor. First, note that $d\geq 2$, since otherwise $D$ is linearly equivalent to both $v_{0,0}$ and $v_{0,n_0}$. By definition of the restricted rank function this means that $D-v_{0,0}$ is equivalent to an effective divisor, call it $E$. Rearranging we have $D \sim E + v_{0,0}$. If $|E-v_{0,n_0}|$ is non-empty, then there is some effective divisor $E'$ such that $E-v_{0,n_0} \sim E'$, hence $D \sim E' + v_{0,0}+v_{0,n_0}$. Then $r(D)\geq r(v_{0,0} + v_{0,n_0}) \geq 1$ (by \Cref{lem:g12}). 
    
    The remaining case is when $D$ is not equivalent to an effective divisor containing $W$ in its support, i.e. $|E-v_{0,n_0}| = \emptyset$, and we will show that this leads to a contradiction. We can assume that $E$ is $v_{0,0}$-reduced, so up to reordering the strands, this implies that we can write
    \begin{align*}
        E \sim av_{0,0} + \sum_{\alpha = 0}^{d-2-a}v_{\alpha,a_\alpha},
    \end{align*}
    where $a \in \set{0,\dots,d-1}$ and $0<a_\alpha<n_\alpha$. See \Cref{lem-v00reduced} below for a precise statement and proof. In other words we can assume that $E$ has $a$ many chips at $v_{0,0}$ and one chip on each of the interiors of the first $d-1-a$ strands. Since $D \sim E+v_{0,0}$, we get that $D \sim (a+1)v_{0,0} + \sum_{\alpha = 0}^{d-2-a}v_{\alpha,a_\alpha}$. Note that because $d\leq g$, there are $g+1-(d-1-a) = a+2 + (g-d) > a+1$ strands with no chips on their interior. So if we consider $D-v_{0,n_0}$ and run Dhar's burning algorithm to compute the $v_{0,n_0}-$reduced representative, the vertex $v_{0,0}$ will burn, and hence the whole graph will burn. Thus $D- v_{0,n_0}$ is $v_{0,n_0}$-reduced and ineffective, which implies $r_{W}(D)  = 0$, yielding our contradiction.
\end{proof}

\begin{rema}
    The upshot of \Cref{lem-BananaRDS} is that the divisors with the highest rank relative to their degree are those with the most chips evenly concentrated at the multivalent vertices (in some representative of the divisor class). In particular if $u = v_{0,0}$ and $v = v_{0,n_0}$ the sequence
    \begin{align*}
        0,u,u+v,2u+v,2u+2v,3u+2v,3u+3v,\dots,(g-1)u+(g-2)v,(g-1)u+(g-1)v
    \end{align*}
    is a sequence of divisors of degree $0,1,\dots, 2g-2$ each of which has maximal rank. Another similar sequence can be obtained by transposing $u,v$ above.
\end{rema}

As the following lemma notes, $v_{0,0}$-reduced divisors on banana graphs have a particularly nice form. In the case of degree $g$ this form exactly coincides with that of a break divisor and moreover for effective divisors of degree at most $g$, this form coincides with the more general notion of a \textit{semibreak} divisor, introduced in \cite{GSTSemibreak}, which is an effective divisor $E$ satisfying $E\leq B$ for some break divisor $B$. In the context of banana graphs, a semibreak divisor is an effective divisor of degree $\leq g$ with at most one chip on the interior of each strand.

\begin{lemma} \label{lem-v00reduced}
    If $D \in \opdiv{B_{n_0,\dots,n_g}}$ is $v_{0,0}$-reduced then $D = av_{0,0} + bv_{0,n_0} + E$ where $E$ is an effective divisor with at most one chip on each strand and no chips at either multivalent vertex, and $0\leq b \leq g - \deg E$. Conversely, every divisor of this form is $v_{0,0}$-reduced. As with all reduced divisors, $r(D)\geq 0$ if and only if $a\geq 0$.
\end{lemma}

\begin{proof}
Applying Dhar's burning algorithm from $v_{0,0}$, the fire is initially blocked by all chips on the interior of a strand, and $v_{0,n_0}$ burns if and only if at least $b+1$ strands are not blocked. If so, the fire then consumes the blocked strands if and only if the blocking chip is the only interior chip of that strand. So a $v_{0,0}$-reduced divisor must have the stated form, and there are at least $b+1$ unblocked strands if and only if $g+1 - \deg E \ge b + 1$.
\end{proof}

 Whenever we write $D = av_{0,0} + bv_{0,n_0} + E$, we assume implicitly that $E$ is an effective divisor with at most one chip on each strand and no chips at either multivalent vertex. Any divisor is linearly equivalent to a divisor of this form.
 We use these representatives to compute ranks of any effective divisor on a banana graph in the following corollary.

\begin{coro}\label{cor-BanaRankComp}
    If $D = av_{0,0} + bv_{0,n_0} + E$ has the form described above, and $a,b$ satisfy $a,b \geq -1$ and either
    \[ 0 \leq a \leq g - \deg E \quad\mbox{or} \quad 0 \leq b \leq g - \deg E,\]
    then
    \begin{align*}
        r(D) &= \min{a,b} + \max{0,\max{a,b}-(g-\deg E)}\\
            &= \max{ \min{a,b}, \deg D - g }.
    \end{align*}
\end{coro}

The moral of \Cref{cor-BanaRankComp} is that the rank of a divisor on a banana graph depends only on the number of chips on the multivalent vertices and the number of chips on the interior of the strands, but not their positions on those strands. This contrasts the chain of $g$ loops where rank is sensitive to the locations of chips within each cycle.

\begin{proof}
    For convenience, define $e = \deg E$. First, we reduce to the effective case. If $D$ is ineffective, assume without loss of generality that $a = -1$, and therefore $0 \leq b \leq g - e$. Then $D$ is $v_{0,0}$-reduced, and therefore has rank $-1$. The claimed formula for $r(D)$ reads $-1 + \max{0, b - g + e }$, which is indeed $-1$ since $b \leq g-e$.

    Now assume $a,b \geq 0$. By symmetry, we may assume that $a \geq b$, and therefore $b \leq g - e$. So $D$ is effective and $v_{0,0}$-reduced. If we also have $a \leq g-e$, then $D$ is also $v_{0,n_0}$-reduced, and it follows readily from \Cref{lem-BananaRDS} that $r(D) = \min{a,b}$, which matches the claimed formula since $\max{a,b} \leq g-e$. 
    
    So we may assume further that $a > g-e$. In this case, the desired formula reads $r(D) = \deg D - g$. Riemann--Roch implies that this is a lower bound on $r(D)$, so it suffices to show that $r(D) \leq \deg D - g$. We can do so constructively: define an effective divisor $L = (a -g + e)v_{0,0} + (b+1) v_{0,n_0}$, of degree $\deg D -g + 1$. Then $D - L = (g-e) v_{0,0} - v_{0,n_0} + E$, which is $v_{0,n_0}$-reduced and ineffective. This shows $r(D) < \deg D - g + 1$, as desired.
\end{proof}

Using \Cref{cor-BanaRankComp} we can isolate some specific characteristics of the $\Delta$ function on various twice-marked banana graphs that we will later use to find values of some transmission permutations. These results can be generalized in a number of ways but we focus on only the cases we actually need to streamline the argument.
\begin{coro}\label{cor-BananaDeltaComps}
    Given a twice-marked banana graph $(B_{n_0,\dots,n_g},u,v)$, we have the following statements about $\Delta$.
    \begin{enumerate}[label = \arabic*)]
        \item If $(u,v) = (v_{0,0},v_{0,n_0})$ and $a \geq 0$ then 
        \begin{align*}
            \Delta(a(v_{0,0}+v_{0,n_0})) = \delta(a\leq g).
        \end{align*}
        \item If $(u,v) = (v_{0,0},v_{0,n_0-1})$ with $0\leq b <a \leq g$ and $0<n<n_0-1$ then
        \begin{align*}
            \Delta(av_{0,n_0}) = \Delta(b(v_{0,0}+v_{0,n_0}) + v_{0,n}) = 1 \quad \text{and} \quad \Delta(av_{0,0}+bv_{0,n_0}+v) = \delta(a = g).
        \end{align*}
        \item If $(u,v) = (v_{0,1},v_{1,n_1-1})$ with $0\leq a\leq g-2$, $2\leq m \leq n_0-1,$ and $1 \leq n \leq n_1-2$ then
        \begin{align*}
            \Delta(a(v_{0,0}+v_{0,n_0})+v_{0,m}+v_{1,n}) = 1.
        \end{align*}
        If $0 \leq b <g$ then
        \begin{align*}
            \Delta(bv_{0,n_0}+v_{0,m}) &= 1
        \end{align*} 
        and
        \begin{align*}\Delta((g-1)v_{0,0}+bv_{0,n_0}+v+v_{0,m}) &= \Delta(b(v_{0,0}+v_{0,n_0})+v_{0,m}+v_{1,n})\\&=\delta(b < g-1).
        \end{align*}
    \end{enumerate}
\end{coro}

\begin{proof}
    \newcommand{\topv}{v_{0,0}}
    \newcommand{\botv}{v_{0,n_0}}

    \newcommand{\tbc}[2]{(#1,#2)}
    
   There is a fair amount of casework, which is summarized in the tables below. To simplify notation, we use the following abbreviation: $(a,b)$ denotes the divisor $a v_{0,0} + b v_{0,n_0}$ supported on the multivalent vertices. In each table, we specify divisors linearly equivalent to the four divisors $D,\, D-u,\, D-v,\, D-u-v$ of the standard form $(a,b) + E$, and compute its rank using \Cref{cor-BanaRankComp}. We also state relevant bounds on any variables in the first column, and compute $\Delta$ in the last.\\

    \noindent\textbf{Case 1:} $(u,v) = (v_{0,0},v_{0,n_0})$. In this case, we want $\Delta( (a,a) )$ for all $a \geq 0$. We have the following table for the cases $a \leq g$.\\

    $\begin{array}{|l|l|l|l|l|l|}
    \hline
    \mathrm{hypotheses} & D\sim & D-u\sim & D-v\sim & D-u-v\sim & \Delta(D)\\ \hline\hline
    0 \leq a \leq g
    & \tbc{a}{a} %
    & \tbc{a-1}{a} %
    & \tbc{a}{a-1} %
    & \tbc{a-1}{a-1} %
    & \empty
    \\
    \empty
    & r = a
    & r= a-1
    & r= a-1
    & r= a-1
    & 1
    \\\hline

    \end{array}$
    \vspace{0.5cm}

    Note that when $a=0$ the divisor $(a-1,a-1) = (-1,-1)$ does not strictly match the hypotheses of \Cref{cor-BanaRankComp}, but we can still conclude that its rank is $a-1 = -1$ since it is less than $(-1,0)$, which already has rank $-1$. This reasoning will be used in the cases below without comment.

    On the other hand, when $a > g$ all four divisors $D,D-u,D-v,D-u-v$ have degree at least $2g$ and are therefore nonspecial by Riemann--Roch, so $\Delta(D) = 0$.\\

    \noindent\textbf{Case 2:} $(u,v) = (\topv, v_{0,n_0-1})$. In this case, we are interested in the three types of divisors summarized in the table below.\\
    
{\footnotesize
    \hspace{-.5cm}$\begin{array}{|l|l|l|l|l|l|}
    \hline
    \mathrm{hypotheses} & D\sim & D-u\sim & D-v\sim & D-u-v\sim & \Delta(D)\\\hline
    1 \leq a \leq g
    & \tbc{0}{a} %
    & \tbc{-1}{a} %
    & \tbc{-1}{a-1} + v_{0,1}
    & \tbc{-2}{a-1} + v_{0,1}
    &
    \\
    & r = 0
    & r = -1
    & r = -1
    & r = -1 
    & 1
    \\\hline

    0 \leq b \leq g-1
    & \tbc{b}{b} + v_{0,n} 
    & \tbc{b-1}{b} + v_{0,n}
    & \tbc{b}{b-1} + v_{0,n+1}
    & \tbc{b-1}{b-1} + v_{0,n+1}
    &
    \\
    1 \leq n \leq n_0-2
    & r = b
    & r = b-1
    & r = b-1
    & r = b-1
    & 1
    \\\hline

    0 \leq b < a \leq g
    & \tbc{a}{b} + v_{0,n_0-1}
    & \tbc{a-1}{b} + v_{0,n_0-1}
    & \tbc{a}{b}
    & \tbc{a-1}{b}
    &
    \\
    & r = b + \delta(a=g)
    & r = b
    & r = b
    & r = b
    & \delta(a=g)
    \\\hline
    \end{array}$\\
        \vspace{0.5cm}
}

    \noindent\textbf{Case 3:} $(u,v) = (v_{0,1},v_{1,n_1-1})$.
    Throughout this case, we assume $m,n$ satisfy $2 \leq m \leq n_0-1$ and $1 \leq n \leq n_1-2$.  We are interested in the four types of divisors summarized in the table below.\\

{\footnotesize
    \hspace{-.5cm}$\begin{array}{|l|l|l|l|l|l|}
        \hline
        \mathrm{hypotheses} & D\sim & D-u\sim & D-v\sim & D-u-v\sim & \Delta(D)\\\hline
        0 \leq a \leq g-2
            & \tbc{a}{a}  
            & \tbc{a-1}{a} 
            & \tbc{a}{a-1} 
            & \tbc{a-1}{a-1}
            & \empty\\
        \empty 
            & + v_{0,m} + v_{1,n}
            & + v_{0,m-1} + v_{1,n}
            & + v_{0,m} + v_{1,n+1}
            & + v_{0,m-1} + v_{1,n+1}
            & \empty\\
        \empty
            & r = a
            & r = a-1
            & r = a-1
            & r = a-1
            & 1\\\hline
        0 \leq b \leq g-1
            & \tbc{0}{b} + v_{0,m}
            & \tbc{-1}{b} + v_{0,m-1}
            & \tbc{-1}{b-1}
            & \tbc{-2}{b-1}
            & \empty\\
        \empty 
            &
            &
            & + v_{0,m} + v_{1,1}
            & + v_{0,m-1} + v_{1,1}
            & \empty\\
        \empty
            & r = 0
            & r = -1
            & r = -1
            & r = -1
            & 1\\\hline
        0 \leq b \leq g-1
            & \tbc{g-1}{b}
            & \tbc{g-2}{b}
            & \tbc{g-1}{b}
            & \tbc{g-2}{b}
            & \empty\\
        \empty 
            & + v_{0,m} + v_{1,n_1-1}
            & + v_{0,m-1} + v_{1,n_1-1}
            & + v_{0,m}
            & + v_{0,m-1}
            & \empty\\
        \empty 
            & r = b + 1
            & r = b
            & r = b
            & r = b - \delta(b = g-1)
            & 1 - \delta(b = g-1)\\\hline
        0 \leq b \leq g-1
            & \tbc{b}{b}
            & \tbc{b-1}{b}
            & \tbc{b}{b-1}
            & \tbc{b-1}{b-1}
            & \empty\\
        \empty
            & + v_{0,m} + v_{1,n}
            & + v_{0,m-1} + v_{1,n}
            & + v_{0,m} + v_{1,n+1}
            & + v_{0,m-1} + v_{1,n+1}
            & \empty\\
        \empty 
            & r = b
            & r = b-1
            & r = b-1
            & r = b-1
            & 1 - \delta(b = g-1) \\
        \empty 
            & + \delta(b = g-1)
            & + \delta(b = g-1)
            & + \delta(b = g-1)
            & \empty
            & \empty\\\hline
    \end{array}$\\
}

        \vspace{0.5cm}
        
    Putting these cases together gives the various formulas stated.
\end{proof}

\section{Submodularity on Banana Graphs}
\label{sec:submodularity}

As established in \cite{Pfl22}, and mentioned in \Cref{lem:tauChars}, the existence of transmission permutations is governed by a convexity condition called \textit{submodularity}. A key question for this approach to Brill--Noether theory is: for which graphs are all divisors submodular? In this section we study the special case of banana graphs. The key results are that in genus 2, there exists a relatively straightforward condition on the marked points such that every divisor is submodular. For genus $3$ and above, we can construct a non-submodular divisor on any twice-marked banana graph with only a few exceptions.

\begin{defi}\label{def-Supp}
For a divisor $D$, the \textit{support complex} of $D$ is the set of vertices to which $D$ can transmit chips while remaining effective. Formally,
\begin{align*}
    \Supp(D) = \set{v \in V(G):r(D-v)\geq 0}.
\end{align*}
Note that if $|D| = \emptyset$ then $\Supp(D) = \emptyset$.
\end{defi}

The following criterion substantially narrows the search for non-submodular divisors on genus $2$ graphs.

\begin{lemma}\label{lemm-rank0supp}
\begin{enumerate}[label = \arabic*)]
\item
    If $(G,u,v)$ is a twice-marked graph of genus $2$ and $u \not\sim v$, then any divisor $D$ with $\Delta(D)<0$ has degree $2$ and rank $0$. 
\item 
If $(G,u,v)$ is a twice-marked graph of \emph{any genus}, and $D$ is a divisor of rank $0$, then $\Delta(D)<0$ if and only if $v \in \Supp(D)\setminus \Supp(D-u)$ and $u \in \Supp(D) \setminus \Supp(D-v)$.
\end{enumerate}
\end{lemma}

\begin{proof}
Assume that $G$ has genus $2$, $u \not\sim v$, and $\Delta(D) < 0$.
First, we show that $\deg D \geq 2$.
Recall that $\Delta(D) <0$ if and only if $r(D) = r(D-u) = r(D-v) = 1+r(D-u-v)$. Suppose for contradiction that $\deg D \leq 1$. Then $D-u-v$ has negative degree, hence rank $-1$, so $D-u$ and $D-v$ both have rank $0$. Since they both have degree at most $0$, both must be principal divisors, so $D \sim u \sim v$, which is a contradiction. So $\deg D \geq 2$. Now, the Riemann--Roch formula implies that $\Delta(D) = \Delta(K_G-D+u+v)$, so $\deg (K_G-D+u+v) \geq 2$ as well. This implies $\deg D \leq 2$, so in fact $\deg D = 2$. By Riemann--Roch, $r(D) = r(K_G-D)+1$, which is $\delta(D \sim K_G)$ since $K-D$ has degree $0$. So to prove $r(D) = 0$, we must show that $D \not\sim K_G$. Suppose for contradiction that $D \sim K_G$. Then $\Delta(D) = \Delta(K_G-D+u+v) = \Delta(u+v) < 0$. But this implies that $r(u+v-u) = 1 + r(u+v-u-v)$, hence $r(v) = 1$, which is impossible. This completes the proof of 1).  

Moving to part 2), now assume $G$ has any genus, and that $r(D) = 0$. By construction we know that $r(D-v) = r(D-u) = 0$, so $u,v \in \Supp(D)$. Further, because $r(D-u-v) = -1$, $v \not \in \Supp(D-u)$ and $u \not \in \Supp(D-v)$. The other direction also follows directly from the definitions.
\end{proof}

\subsection{The Genus 2 Case}

We give a complete classification of theta graphs on which all divisors are submodular in \Cref{thm-NonSubmodGenus2} below, which is a more precise version of \Cref{thm:thetaSimple}.

\begin{rema}\label{rem-closePoints}
    As a heuristic, \Cref{thm-NonSubmodGenus2} says that the only twice-marked theta graphs with non-submodular divisors are those for which the marked points are ``too close together.'' This is prone to occur when the marked points lie on the same strand. The most extreme case of this is when  $u\sim v$.
\end{rema}

\begin{theorem}\label{thm-NonSubmodGenus2}
Let $G = \theta_{n_0,n_1,n_2}$, with $(G,u,v)$ a twice-marked theta graph with $u\not \sim v$. The following are equivalent.
\begin{enumerate}[label = \arabic*)]
    \item There exist divisors $D$ with $\Delta(D)<0$.
    \item The marked points $u,v$ are on the same strand, so without loss of generality $u = v_{\alpha,i},v = v_{\alpha,j}$. Further the set
    $$N_{(G,u,v)} := \set{v_{\alpha,k}:k \neq n_\alpha -i, k \neq j, j-i\leq k \leq j-i+n_\alpha}$$ is nonempty. In particular there is a bijection:
    \begin{align*}
        N_{(G,u,v)} &\to \set{[D]\in\pic{G}:\Delta(D)<0}\\
        v_{\alpha,k} & \mapsto [v_{\alpha,k}+v_{\alpha,i}].
    \end{align*}
\end{enumerate}
\end{theorem}

\begin{lemma}\label{lem-SameStrand}
    On a genus $g \ge 2$ banana graph $B_{n_0,\dots,n_g}$, if $r(v_{\alpha,i}+v_{\beta,j}-v_{\gamma,k}) = 0$, then one of the following holds:
    \begin{enumerate}[label = \arabic*)]
        \item $v_{\alpha,i} = v_{\gamma,k}$;
        \item $v_{\beta,j} = v_{\gamma,k}$;
        \item $\overline{v_{\alpha,i}} = v_{\beta,j}$; or
        \item $v_{\alpha,i},v_{\beta,j},$ and $v_{\gamma,k}$ are all on the same strand.
    \end{enumerate}
\end{lemma}

\begin{proof}
    Let $D = v_{\alpha,i}+v_{\beta,j}-v_{\gamma,k}$. Suppose $r(D) = 0$ and $1)$, $2)$ and $3)$ do not hold. If $\alpha \neq \beta$ then $D$ is $v_{\gamma,n}$-reduced and ineffective thus $\alpha = \beta$. If $\alpha \neq \gamma$ then since $3)$ does not hold then
    \begin{align*}
        D \sim \begin{cases}
            v_{\alpha,0}+v_{\alpha,i+j} - v_{\gamma,k} & i + j < n_\alpha\\
            v_{\alpha,i + j - n_\alpha}+v_{\alpha,n_\alpha}-v_{\gamma,k} & i +j > n_\alpha
        \end{cases}.
    \end{align*}
    Since these are both $v_{\gamma,k}$-reduced and ineffective we conclude that $\alpha = \beta = \gamma$.
\end{proof}

\begin{proof}[Proof of \Cref{thm-NonSubmodGenus2}]
This theorem is largely a consequence of \Cref{lemm-rank0supp}, and \Cref{lem-SameStrand}.

Suppose $D$ is a divisor on $(G,u,v)$ with $\Delta(D)<0$. By \Cref{lemm-rank0supp}, $r(D) = 0$ and $\deg D  = 2$. Note that $D-u$ is a rank $0$ degree $1$ divisor, so by \Cref{lem-bridgelessFacts} there is a unique vertex $w$ such that $D-u\sim w$. Thus $D-v$ is a rank $0$ divisor equivalent to $w+u-v$. The facts that $u \not \sim v$, that $v \not \in \Supp(D-u)$, and that $r(D) = 0$, rule out possibilities $1),2),$ and $3)$ respectively of \Cref{lem-SameStrand} so we can conclude that $u,v,w$ are on the same strand. Without loss of generality, let $w = v_{\alpha,k}, u = v_{\alpha,i}$. Note that because $r(D) = 0$, we get that $i+k \neq n_\alpha$. Applying \Cref{lemm-rank0supp}, we see that
\begin{align*}
    v \in \Supp(D)\setminus \Supp(D-u) = \set{v_{\alpha,j}:j \neq k,i+k-n_\alpha \leq j\leq i+k}.
\end{align*}
So we can write $v = v_{\alpha,j}$ with the conditions on $j$ as above. Reformulating this condition in terms of $k$, we see that $k \neq n_\alpha-i,j$ and $j-i\leq k \leq j-i+n_\alpha$, as desired.

For the other direction, if $v_{\alpha,k} \in N_{(G,u,v)}$, let $D = v_{\alpha,i}+v_{\alpha,k}$. Then we have that $r(D) = \delta(i+k = n_\alpha) = 0$ and that
\begin{align*}
    \Supp(D)\setminus\Supp(D-u) = \set{v_{\alpha,\ell}:\ell \neq k,i+k-n_\alpha\leq \ell \leq i+k} \ni v_{\alpha,j} = v;\\
    \Supp(D)\setminus\Supp(D-v) = \set{v_{\alpha,\ell}:\ell \neq i+k-j,i+k-n_\alpha\leq \ell \leq i+k} \ni v_{\alpha,i} = u.
\end{align*}
By \Cref{lemm-rank0supp} this implies $\Delta(D)<0$. These constructions are evidently inverses to one another.
\end{proof}

\begin{coro}\label{cor-allSubmodSameStrand}
    Given a theta graph $(G,v_{\alpha,i},v_{\beta,j})$ every divisor is submodular if and only if either
    \begin{enumerate}[label = \arabic*)]
        \item The marked vertices are not on the same strand, or
        \item the marked vertices are on the same strand, so without loss of generality $\alpha = \beta$ and $i<j$, and either 
        \begin{enumerate}[label = \alph*)]
            \item $i = 0, j \in \set{n_\alpha-1,n_\alpha}$, or
            \item $i = 1, j = n_\alpha$.
        \end{enumerate}
    \end{enumerate}
\end{coro}

The characterization of non-submodularity on theta graphs obtained above can also be applied to the semidegenerate case of the chain of two loops. Here we denote the lengths of the loops by $n_0, n_1$ and label vertices on loop $\alpha$ by $v_{\alpha,0}, \ldots, v_{\alpha,n_\alpha}$, where $v_{\alpha, 0} = v_{\alpha,n_\alpha}$ is the glued vertex, just as if working on a banana graph $B_{n_0, n_1, 0}$ with one strand of length $0$.

\begin{prop}\label{rem-degenerateTheta}
    On a chain of two loops, every divisor is submodular if and only if the marked points are on distinct loops (and neither is the glued vertex), or if $n_\alpha = 2$ and $\set{u,v} = \set{v_{\alpha,0},v_{\alpha,1}}$.
\end{prop}

\begin{proof}
    If $u = v$ then $\Delta(u) = -1$, so assume $u \ne v$.
    The case of marked points on distinct loops is precisely what arises from gluing two twice-marked cycles as in \cite{Pfl22} which was shown to be submodular in Lemma 2.5 and Theorem 3.11 of that paper. If the marked points are on the same loop of length at least $3$, choose a vertex $w$ distinct from $v$ and the reflection of $u$ through the glued vertex. Then $\Delta(u+w) = -1$. On the other hand if the $\alpha$-th loop is length $2$ and without loss of generality $(u,v) = (v_{\alpha,1},v_{\alpha,0})$ then any putative non-submodular divisor must satisfy $D \sim u+w$. Thus by \Cref{lemm-rank0supp}, $v \in \Supp(D)\setminus \Supp(D-u) = \set{u}$, a contradiction. 
\end{proof}

Before moving on, we mention another immediate consequence of \Cref{lem-SameStrand} that will be useful later.

\begin{coro}\label{cor:suppUV}
If $g \ge 2$ and vertices $u,v$ on $B_{n_0, \ldots, n_g}$ do not lie on the same strand, then $\Supp(u+v) = \{u,v\}$.
\end{coro}

\subsection{Banana Graphs of Genus \texorpdfstring{$\geq$}{at least} 3}
Unfortunately, the picture is not as nice for higher genus banana graphs. As the following result shows, the presence of additional strands allows non-submodular divisors to appear for most banana graphs of higher genus. We prove the following, which implies the simplified \Cref{thm:bananaSimple} from the introduction.

\begin{theorem}\label{thm-NSMForBanana}
Let $(G,u,v) = (B_{n_0,\dots,n_g},v_{\alpha,i},v_{\beta,j})$ be a banana graph of genus $g \geq 3$. 

Then we have the following classification for non-submodular divisors on $G$. Either
\begin{enumerate}[label = \arabic*)]
    \item up to exchanging $u$ and $v$, one of the following five possibilities holds :
    \begin{enumerate}[label = \alph*)]
        \item $i=0$ and $j = n_\beta$,
        \item $i=1 < n_\alpha$ and $j = n_\beta$,
        \item $i=0$ and $1 \le j = n_\beta-1$,
        \item $\alpha \neq \beta$, $n_\alpha \ge 2$, $n_\beta \ge 2$, and $(i,j) = (1,n_\beta-1)$,
        \item $\alpha \neq \beta$, $n_\alpha = 2$, $i=1$, and $1 \le j \le n_\beta -1$,
    \end{enumerate}
    \item[or] 
    \item There exist divisors $D \in \opdiv{G}$ such that $\Delta(D)<0$.
\end{enumerate}
\end{theorem}

\begin{figure}
    \centering
    \begin{tikzpicture}[scale=0.8]
        \def\h{2} %
        \def\spacing{4} %
        \def\lbend{100} %
        \def\rbend{100} %
        \def\dotsXloc{0.45} %
        \def\offset{0.1} %
        \node (T) at (0,\h) {$\bullet$};
        \node (B) at (0,0) {$\bullet$};
        \node (dots) at (\dotsXloc,\h/2) {$\dots$};
        \draw[thick] (T) to (B);
        \draw[thick] (T) to [bend right = \rbend] (B);
        \draw[thick] (T) to [bend left = \lbend] (B);
        \node (cap) at (0,-1) {$\set{v_{0,0},v_{0,n_0}}$};
        \begin{scope}[shift = {(\spacing,0)}]
            \node (T) at (0,\h) {$\cdot$};
            \node (B) at (0,0) {$\bullet$};
            \node (dots) at (\dotsXloc,\h/2) {$\dots$};
            \draw[thick] (T) to (B);
            \draw[-dot-=\offset,thick] (T) to [bend right = \rbend] (B);
            \draw[thick] (T) to [bend left = \lbend] (B);
            \node (cap) at (0,-1) {$\set{v_{0,1},v_{0,n_0}}$};
        \end{scope}
        \begin{scope}[shift = {(2*\spacing,0)}]
            \node (T) at (0,\h) {$\bullet$};
            \node (B) at (0,0) {$\cdot$};
            \node (dots) at (\dotsXloc,\h/2) {$\dots$};
            \draw[thick] (T) to (B);
            \draw[-dot-=1-\offset,thick] (T) to [bend right = \rbend] (B);
            \draw[thick] (T) to [bend left = \lbend] (B);
            \node (cap) at (0,-1) {$\set{v_{0,0},v_{0,n_0-1}}$};
        \end{scope}
         \begin{scope}[shift = {(3*\spacing,0)}]
            \node (T) at (0,\h) {$\cdot$};
            \node (B) at (0,0) {$\cdot$};
            \node (dots) at (\dotsXloc,\h/2) {$\dots$};
            \draw[-dot-=1-\offset,thick] (T) to (B);
            \draw[-dot-=\offset,thick] (T) to [bend right = \rbend] (B);
            \draw[thick] (T) to [bend left = \lbend] (B);
            \node (cap) at (0,-1) {$\set{v_{0,1},v_{1,n_1-1}}$};
        \end{scope}
    \end{tikzpicture}
    \caption{The first four exceptional configurations of \Cref{thm-NSMForBanana}, for banana graphs of genus $\geq 3$ and strands of length at least $3$. The large dots indicate the locations of the marked points and the ellipses indicate the possible presence of additional strands. Not shown: the length-$2$ midpoint exception stated in \Cref{thm-NSMForBanana} 1(e).}

    \label{fig:banaNonSubMod}
\end{figure}

The first four possibilities enumerated in part $1$ of the theorem are euphemistically illustrated in \Cref{fig:banaNonSubMod}.

\begin{proof} 
Consider two possibilities depending whether $u,v$ are on the same strand. In each case we show that if the conditions in part $1)$ do not apply then we can construct a non-submodular divisor.

First suppose that $u,v$ are not on the same strand. If $n_\alpha = 2$ or $n_\beta = 2$ then, perhaps after relabeling, $i = 1$ and we are in case 1(e). Otherwise, up to relabelling, we may assume that $i<n_\alpha-1$. If $j<n_\beta-1$ consider the divisor $D:= v_{\alpha,1}+v_{\alpha,i}+v_{\beta,n_\beta-1}$. Then $D \sim v_{\alpha,0} + v_{\alpha,i+1} + v_{\beta,n_\beta-1}$, which is $v_{\alpha,n_\alpha}$-reduced, so $D$ is rank $0$.  Further, we have that
\begin{align*}
    r(D-u) = r(v_{\alpha,1}+v_{\beta,n_\beta-1})  = 0;\\
    r(D-v) = r(v_{\alpha,i+1}+v_{\beta,n_\beta-1-j}) = 0;\\
    r(D-u-v) = r(v_{\alpha,i+1}+v_{\beta,n_\beta-1-j}-v_{\alpha,i}) = -1.
\end{align*}
On the other hand if $j = n_\beta-1$ and $i \neq 1$ then consider the divisor $D:= v_{\alpha,i}+v_{\alpha,n_\alpha-1}+v_{\beta,1}$. 
Then we have $D \sim v_{\alpha,i-1} + v_{\alpha,n_\alpha} + v_{\beta,1}$, which is $v_{\alpha, 0}$-reduced, so $D$ is rank $0$. 
Further, $D-u$ is effective, and $D-v \sim v_{\alpha,i-1} + v_{\beta,2}$, which is effective, but $D-u-v \sim v_{\alpha,i-1} + v_{\beta,2} - v_{\alpha,i}$ is ineffective and $v_{\alpha,i}$-reduced. Therefore
\begin{align*}
    r(D-u) = r(v_{\alpha,n_\alpha-1}+v_{\beta,1}) = 0;\\
    r(D-v) = r(v_{\alpha,i-1}+v_{\beta,2}) = 0;\\
    r(D-u-v) = r(v_{\alpha,i-1}+v_{\beta,2}-v_{\alpha,i}) = -1.
\end{align*}

Thus if $u,v$ are on different strands we can construct non-submodular divisors except when $i = 1,j = n_\beta-1$ (case 1(d)), or in case 1(e).

On the other hand, if $\alpha = \beta$ then up to swapping $u,v$ we may assume that $i \leq j$. Then, the same analysis as \Cref{cor-allSubmodSameStrand}, extended to higher genus banana graphs, shows that unless $(i,j) \in \set{(0,n_\alpha),(1,n_\alpha),(0,n_\alpha-1)}$ we can always construct a non-submodular divisor.
\end{proof}

\begin{rema}
    As the above theorem suggests, it is sometimes the case that a graph has non-submodular divisors arising from non-submodular divisors on some subgraph. Clearly it is not the case that non-submodular divisors on a subgraph always result in non-submodular divisors on the larger graph, but an essentially identical argument to the one used to prove \Cref{cor-allSubmodSameStrand} can be used to find non-submodular divisors on a much larger family of graphs containing a theta-graph as a minor. This in turn suggests a possible ``forbidden-minor'' approach to identifying graphs with non-submodular divisors, reflecting the fact that failure of submodularity is largely sensitive to the structure of the graph \textit{near the marked points}. We do not attempt to push this characterization further as we do not need it for our results here, but note that it may be a fruitful approach to understanding submodularity more generally.
\end{rema}

\section{k-General Transmission in Banana Graphs}\label{sec:kgt}

In this section we first give an explicit description of transmission permutations which arise on theta graphs, framed in terms of effective representatives of twists in degree $2$. Next we give an abstract condition called non-recurrence which we show is equivalent to $k$-general transmission in genus $2$. We then use non-recurrence to deduce general transmission for a family of well behaved theta graphs. Lastly, we analyze general transmission on banana graphs of genus $\geq 3$, concluding that they in general fail to have general transmission.

\begin{rema}\label{rem-PermInvol}
    A natural question to ask about $k$-general transmission is whether it depends on the order of the pair of marked vertices. Clearly $\Delta(D)$ is symmetric in $u,v$ so the first condition does not and by Riemann--Roch we have that
    \begin{align*}
        \delta(\tau_D^{u,v}(b) = a) = \delta(\tau_{K_G-D+u+v}^{v,u}(a) = b).
    \end{align*}
    Thus permuting the marked vertices merely permutes the set of transmission permutations, so $k$-general transmission is invariant under such a swap. The somewhat odd looking factor of $u+v$ in this involution is a result of choices made in the definition of transmission permutations which were motivated by a desire to keep statement \cite[Theorem 3.11]{Pfl22} as clean as possible. 
\end{rema}

\subsection{Torsion order and general transmission}

Throughout this section, we will always take $k$ to be the torsion order of $(G,u,v)$ when studying $k$-general transmission. Here we justify that decision.

\begin{lemma} \label{lem:kgtImpliesTorsionOrder}
If $(G,u,v)$ is a twice-marked graph with $g \ge 1$ and $k$-general transmission, then $k$ is the torsion order of $(G,u,v)$.
\end{lemma}
\begin{proof}
Let $m$ be the torsion order of $(G,u,v)$, and let $D = 0$ be the zero divisor. Since $r(D) = 0$, \Cref{lem:tauChars} implies that there exists a unique $\ell \geq 0$ such that $\tauD(\ell) \leq 0$. This in turn implies that $r(D + \tauD(\ell) u - \ell v) = 0$. Effective divisors have nonnegative degree, so in fact $\ell = 0$ and $\tauD(0) = 0$. Applying the other part of \Cref{lem:tauChars} and Riemann--Roch, $r(K_G) = g-1$ implies that there exists a set $A$ of $g$ distinct integers $a < 0$ such that $\tauD(a) > 0$. Each such $a$ gives a distinct $k$-inversion $[(a,0)]$ (the square brackets indicate the $k$-equivalence class) of $\tauD$. Since $\inv_k(\tauD) \leq g$, these are \emph{all} of the $k$-inversions of $\tauD$. Hence every inversion $(a,b)$ of $\tauD$ satisfies $b \equiv 0 \mod k$.

The fact that $mu \sim mv$ implies that $\tauD(n + m) = \tauD(n)+m$ for all $n \in \Z$. Therefore for any $a \in A$, $(a+m,m)$ is also an inversion of $\tauD$. Therefore $m \equiv 0 \mod k$, i.e. $k \mid m$.

Conversely, $k$-general transmission implies that $\tauD \in \eaf{k}$, so $\tauD(0) = 0$ implies $\tauD(k) = k$, and hence $r(ku-kv) \geq 0$. The only degree $0$ effective divisor is $0$, so $k(u-v) \sim 0$ and therefore $m \mid k$. So in fact $m = k$.
\end{proof}

The converse of \Cref{lem:kgtImpliesTorsionOrder} is quite false, but a partial converse does hold for $k=2$ only.

\begin{lemma}\label{lem-TO2GenTrans}
    If $(G,u,v)$ has torsion order $2$ and every divisor is submodular then $G$ has $2$-general transmission.
\end{lemma}

\begin{proof}
    Let $D$ be any divisor, and $\tau = \tauD$. We will show that $\inv_2(\tau) \leq g$. We may assume $\inv_2(\tau) \geq 1$. Then $\tau$ has at least one inversion, and therefore at least one \emph{adjacent} inversion $\tau(b) > \tau(b+1)$. Any inversion $(i,j)$ is equivalent to one for which $i \in \{b,b+1\}$, and the identity $\tau(i+2k) = \tau(i)+2k$ implies that the set of such inversions is $\{(b,b+1+2m):\ 0 \leq m < \frac12(\tau(b) - \tau(b+1)) \}$. Now, Equation \eqref{eq-RRTauBounds} implies that 
    \[ \tau(b) - \tau(b+1) \leq \left[ 2g+b - \deg D \right] - \left[ b+1 - \deg D \right] = 2g-1,\]
    hence there are at most $g$ such inversions.
\end{proof}

\subsection{Computing Transmission Permutations on Banana Graphs}
In this section we make explicit the relationship between ``interesting'' twists of a divisor and factors of the corresponding transmission permutation. To that end we give an explicit formula for transmission permutations of theta graphs in the most generic case.

\begin{defi}
    A twice-marked graph $(G,u,v)$ is \textit{rigidly marked} if every divisor $D \in \opdiv{G}$ is submodular and $r(u + v) = 0$.
\end{defi}

A \textit{rigidly} marked graph satisfies $u \not \sim v$, and in the context of banana graphs $u \not \sim \Bar{v}$ and $v \not \sim \Bar{u}$. In the genus $2$ case the second condition is equivalent to $u+v \not \sim K_G$. This definition is useful for excluding some annoying corner cases where points of interest on our twice-marked graphs happen to coincide.

\subsubsection{Theta Graphs}
\begin{prop}\label{prop-thetaTransChar}
Let $(G,u,v)$ be a rigidly marked theta graph. Let $D$ be any degree $2$ divisor on $G$. For $t \in \Z$, Define $D'_t$ to be $D+t(u-v)$. Then
\begin{align*}
    \tau_D^{u,v}(t) = \begin{cases}
    t-2 & D'_t \sim 2u\\
    t-1 & D'_t \sim u +w \text{ for some } w \not \sim u,v\\
    t+1 & D'_t \sim v+w \text{ for some } w \not \sim \Bar{u},\Bar{v}\\
    t+2 & D'_t \sim v+\Bar{u}\\
    t & \text{ else}.
    \end{cases}
\end{align*}

\end{prop}

\begin{proof}
Note that the rigidly marked condition ensures that $\tauD$ exists and that the five cases above do not overlap with one another. We can examine each of these cases with \Cref{cor-BanaRankComp}.

If $D'_t \sim 2u$ then
\begin{align*}
    \Delta(D + (t-2)u-tv) = \Delta(0) = 0-(-1)-(-1)+(-1) = 1.
\end{align*}

If $D_t' \sim u+w$ for some $w$ other than $u,v$, then
\begin{align*}
    \Delta(D+(t-1)u-tv) =\Delta(w) = 0-(-1)-(-1)+(-1) = 1.
\end{align*}

If $D_t' \sim v+w$ for some $w$ other than $\Bar{u},\Bar{v}$, then
\begin{align*}
    \Delta(D+(t+1)u-tv) = \Delta(u+v+w) = 1-0-0+0 = 1.
\end{align*}

If $D_t' \sim v+\Bar{u}$, then
\begin{align*}
    \Delta(D+(t+2)u-tv) = \Delta(K_G+u+v) = 2-1-1+1 = 1.
\end{align*}

Lastly, if none of the other four conditions hold, then we know that $r(D'_t) = 0$ and $r(D'_t-u) = r(D'_t-v) = -1$ so we get that
\begin{align*}
    \Delta(D+tu-tv) = 0-(-1)-(-1) +(-1) = 1.
\end{align*}
\end{proof}

The upshot of \Cref{prop-thetaTransChar} is that in the rigidly marked case, we can always split transmission permutations into commuting factors of the form $\sigma_t, \sigma_t\sigma_{t+1}$ or $\sigma_{t+1}\sigma_t$, where $\sigma_t$ denotes the simple transposition of $t$ and $t+1$. In particular, inversions can be counted by examining twists of a fixed degree, which we exploit in the next section.

\subsection{Non-Recurrence}

The purpose of this section is to prove the following criterion for $k$-general transmission in genus $2$. The idea behind this criterion is: in genus $2$, the only possible special divisors are those of degrees $0,1$, and $2$. Degrees $0$ and $2$ are predictable, so the complexity occurs in degree $1$. Therefore, in studying $\inv_k(\tauD)$ for a divisor $D$, one is primarily interested in how many degree $1$ twists of $D$ are special, i.e. linearly equivalent to an effective divisor. These twists form an arithmetic progression of common difference $[u-v]$, and in ``typical'' circumstances one does not expect more than two special divisors in such a progression.

\begin{defi}
Let $G$ be a graph, and $[D] \in \Pic^0(G)$ be a degree $0$ divisor class, with order $k$. Call $[D]$ \emph{non-recurrent} if for every $v \in V(G)$, there is at most one integer $n \in \{1,\ldots,k-1\}$ such that $|nD + v| \neq \emptyset$.
\end{defi}

The following equivalent characterization of non-recurrence in genus $2$ is a direct consequence of Riemann--Roch, and is the reason for our choice of the word ``recurrent.''

\begin{lemma}\label{lem:nonrecDisjoint}
If $G$ has genus $2$ and $[D] \in \operatorname{Pic}^0(G)$ has order $k$, then $[D]$ is non-recurrent if and only if all the sets $\Supp(K_G - nD)$ for $1 \le n \le k-1$
are pairwise disjoint.
\end{lemma}

Our main result on non-recurrence is the following.

\begin{theorem}
\label{thm:kgtThetas}
Suppose that $(G,u,v)$ is a rigidly marked graph of genus $2$ and torsion order $k$. Then $(G,u,v)$ has $k$-general transmission if and only if $[u-v] \in \Pic^0(G)$ is non-recurrent.
\end{theorem}

The proof of this theorem requires some groundwork, namely \Cref{lem:invtau} below, and begins on page \pageref{proof:kgtThetas}.

\begin{defi}
Let $D$ be a divisor on a twice-marked graph $(G,u,v)$. Denote by $T^d_D$ the set of divisor classes
$$T^d_{D} = \left\{ [D + a u - bv]: a,b \in \Z,\ \deg D + a - b = d \right\} \subseteq \Pic^d(G).$$
We call this the set of \emph{degree $d$ twists} of $D$. Note that $\#T^d_D = k$, where $k$ is the torsion order of $(G,u,v)$.
\end{defi}

\begin{lemma}
\label{lem:invtau}
Suppose that $D$ is a submodular divisor on a twice-marked graph $(G,u,v)$ of genus $2$. Then
$$\inv_k (\tau_D^{u,v}) = 
\# \{ [D'] \in T^1_D:\ |D'| \neq \emptyset \} + \delta( 0 \in T^0_D \mbox{ and } u+v \sim K_G).$$
\end{lemma}

As the reader will see, the proof of this lemma is somewhat intricate and not particularly intuitive. It does not begin until page \pageref{proof:invtau}, after we have laid some groundwork. Therefore some remarks on the original intuition are in order. First, inversions of $\tauD$ can be used to detect the twists of $D$ that are special divisors, i.e. have larger-than-expected rank; see \Cref{lem:tauChars}. Now, the genus of $G$ gives us an advantage: with the exception of $0$ and $K_G$, all special divisors have degree $1$. Therefore, we can hope to count inversions of $\tauD$ by counting the degree $1$ twists of $D$ that are special, which is the first term in the sum above. The second term is a correction that is necessary to account for the special divisors that are not degree $1$.

To formalize this intuition, we first draw a precise link between the special degree-$1$ twists and the inversions of $\tauD$. This link will not be a bijection, as we will see, but we will be able to correct the discrepancy. 
There is a function
$$\{ [D'] \in T^1_D:\ |D'| \neq \emptyset \} \to \Inv_k(\tauD),$$
which we will describe shortly. Although this function need not be a bijection, both injectivity and surjectivity fail in ways that can be understood.

To describe this function, observe the following identity: for any twist $D' = D + au-bv$, \Cref{lem:tauChars} shows that
\begin{eqnarray*}
\left( r(D') + 1 \right) \left( r(K_G-D') + 1 \right)
&=& \# \{ (m,n) \in \Z^2:\ m < b \leq n,\;\tauD(m) > a \geq \tauD(n) \}.
\end{eqnarray*}
Assuming $G$ has genus $2$ and $\deg D' = 1$, Riemann--Roch implies that $r(D') +1= r(K_G-D')+1$, and both numbers are either $0$ or $1$, depending on whether or not $|D'| \neq \emptyset$. Therefore, each degree $1$ twist $D' = D + au-bv$ such that $|D'| \neq \emptyset$ determines a unique inversion $(m,n)$ such that $m < b \leq n$ and $\tauD(m) > a \geq \tauD(n)$. Furthermore, two twists are linearly equivalent if and only if the two choices of $(a,b) \in \Z^2$ differ by a multiple of $(k,k)$, so the \emph{divisor class} $[D']$ determines a unique $k$-equivalence class in $\Inv_k(\tauD)$. This describes the desired function.

The proof of \Cref{lem:invtau} therefore follows this strategy: the number $\inv_k(\tauD)$ is expressed as an inclusion-exclusion calculation, which corrects for both the possible non-injectivity and non-surjectivity of the function described above. Almost all terms in this calculation will cancel, to leave \Cref{lem:invtau}.

Next, we state a general formula that expresses the number of $k$-inversions of $\tauD$ in terms of ranks of twists of $D$. The formula is valid in every genus, and is conveniently stated using the following shorthand. This shorthand will be used to count the sums of sizes of certain sets of inversions that will arise in our inclusion-exclusion argument.

\begin{defi}
Let $D,E$ be two divisors on a twice-marked graph $(G,u,v)$ of any genus. For any two integers $d,e$ such that $d+e = \deg E$, define 
$$S^{d,e}_D(E) = \sum_{[D'] \in T^d_D} \left( r(D') + 1 \right) \left( r(E-D') + 1\right).$$
Denote also $S_D(E) = \sum_{d \in \Z} S^{d,\deg E - d}_D(E)$.
\end{defi}

The number $e$ is redundant in this notation since it is determined uniquely by $d$ and $\deg E$; we include it as a reminder of the degrees of the divisors considered in the second factor, and to more clearly highlight the following symmetry:
$$S^{d,e}_D(E) = S^{e,d}_{E-D}(E).$$

One special case is important in our analysis: when $d=0$, we have the formula
\begin{equation}
\label{eq:S0e}
S^{0,e}_D(E) = \delta(0 \in T^0_D) \left( r(E) + 1 \right),
\end{equation}
since $D' = 0$ is the only possible choice for which the factor $(r(D')+1)$ does not vanish. By like reasoning, the case $e=0$ has the formula
\begin{equation}
\label{eq:Sd0}
S^{d,0}_D(E) = \delta(E \in T^d_D) \left( r(E) + 1 \right).
\end{equation}

The discussion above shows that, \emph{when $G$ has genus $2$}, the ``main term'' from \Cref{lem:invtau} is one of these numbers: 
$$
\# \{ [D'] \in T^1_D:\ |D'| \neq \emptyset \} = S^{1,1}_D(K_G).
$$
Furthermore, the reader may verify that the ``correction term'' in \Cref{lem:invtau} is also one of these numbers:
$$
\delta( 0 \in T^0_D \mbox{ and } u+v \sim K_G) = S^{0,0}_D(K_G-u-v).
$$

So \Cref{lem:invtau} is equivalent to saying that, in genus $2$, we have $\inv_k(\tauD) = S^{1,1}_D(K_G) + S^{0,0}_D(K_G-u-v)$. We prove this by first proving the following more general fact (valid in all genera), and then showing that many terms vanish or cancel in genus $2$. Although the proof of this lemma is written in an algebraic way for convenience, the reader should observe that the formula resembles an inclusion-exclusion calculation, and indeed the proof may be reformulated explicitly using the inclusion-exclusion principle with a bit more work.

\begin{lemma}
\label{lem:invtauGeneral}
Let $D$ be a submodular divisor on a twice-marked graph $(G,u,v)$ of any genus with torsion order $k$. Then
$$
\inv_k(\tauD) = S_D(K_G) - S_D(K_G-u) - S_D(K_G-v) + S_D(K_G-u-v).
$$
\end{lemma}

\begin{proof}
First, observe that, since every equivalence class of inversion has a unique representative $(a,b)$ such that $0 \leq b < k$,  we may write
$$
\inv_k(\tauD)
=
\sum_{b=0}^{k-1} \# \{ n < b:\ \tauD(n) > \tauD(b) \}.
$$

It will be convenient to turn this into a double sum, introducing a new variable $a$ that will serve as the value $\tauD(b)$. We accomplish this by writing

\[ \# \{ n < b:\ \tauD(n) > \tauD(b) \} 
= \sum_{a \in \Z} \delta(\tauD(b) = a) \cdot \# \{n < b:\ \tauD(n) > a\}
\]

Now, the indicator $\delta(\tauD(b) = a)$ is equal to $\Delta( D + au-bv)$ by \Cref{def-tauD}, and $\# \{n < b:\ \tauD(n) > a\}$ is equal to the quantity $r(K_G - (D+au-bv)) + 1$ by \Cref{lem:tauChars}. Therefore we may rewrite

\begin{eqnarray*}
\inv_k(\tauD) &=&
\sum_{b=0}^{k-1} \sum_{a \in \Z} \delta(\tauD(b) = a) \cdot \# \{n < b:\ \tauD(n) > a\}\\
&=& \sum_{b=0}^{k-1} \sum_{a \in \Z} \Delta(D + au - bv) \left( r(K_G-(D+au-bv)) + 1 \right).\\
\end{eqnarray*}
Since $[u-v]$ has order $k$ in $\jac{G}$, the divisor $D' = D + au - bv$ in this last sum ranges over a system of representatives for the divisor classes of twists of $D$ (of any degree). Therefore the sum may be rewritten
$$
\inv_k(\tauD) = \sum_{d \in \Z} \sum_{[D'] \in T^d_D} \Delta( D') \left( r(K_G-D') + 1 \right).
$$
We may write $\Delta(D')$ as 
$\displaystyle \sum_{E \in \{0,u,v,u+v\}} (-1)^{\deg E} \left( r(D' - E) + 1 \right)$, and hence rewrite the equation above as follows.
\begin{eqnarray*}
\inv_k(\tauD) &=& \sum_{d \in \Z} \sum_{[D'] \in T^d_D} \sum_{E \in \{0,u,v,u+v\}} (-1)^{\deg E} \left( r(D'-E) + 1 \right) \left( r(K_G-D') + 1 \right)\\
&=& \sum_{E \in \{0,u,v,u+v\}} (-1)^{\deg E} \sum_{d \in \Z} \sum_{[D'] \in T^d_D} \left( r(D'-E) + 1 \right) \left( r(K_G-D') + 1 \right).
\end{eqnarray*}
In this sum, for a fixed choice of $E$, the fact that $E$ is a $\Z$-linear combination of $u$ and $v$ means that the divisor $D' - E$ ranges over all twists of $D$ of any degree. So we may write more simply 
$$
\sum_{d \in \Z} \sum_{[D'] \in T^d_D} \left( r(D'-E) + 1 \right) \left( r(K_G-D') + 1 \right) = S_D(K_G - E).
$$
The lemma follows.
\end{proof}

\begin{proof}[Proof of \Cref{lem:invtau}] \label{proof:invtau}
Observe that $S^{d,e}_D(E) = 0$ whenever $d<0$ or $e<0$. So if $G$ has genus $2$, and thus $\deg K_G = 2$, the four terms in \Cref{lem:invtauGeneral} each expand to a relatively small number of terms, as follows.
\begin{eqnarray*}
S_D(K_G) &=& S^{2,0}_D(K_G) + S^{1,1}_D(K_G) + S^{0,2}_D(K_G);\\
S_D(K_G-u) &=& S^{1,0}_D(K_G-u) + S^{0,1}_D(K_G-u);\\
S_D(K_G-v) &=& S^{1,0}_D(K_G-v) + S^{0,1}_D(K_G-v);\\
S_D(K_G-u-v) &=& S^{0,0}_D(K_G-u-v).
\end{eqnarray*}
As discussed in the paragraph above \Cref{lem:invtauGeneral}, we wish to show that, when the formula in \Cref{lem:invtauGeneral} is expanded using these equations, all terms cancel except $S^{1,1}_D(K_G)$ and $S^{0,0}_D(K_G-u-v)$.

Equations \eqref{eq:S0e} and \eqref{eq:Sd0}, Riemann--Roch, and the fact that $r(u) = r(v) = 0$, give the following formulas.

\begin{eqnarray*}
S^{2,0}_D(K_G) &=& \delta( K_G \in T^2_D) \left( r(K_G) + 1 \right)\\
&=& 2 \delta( K_G \in T^2_D);\\
S^{1,0}_D(K_G-u) &=& \delta(K_G - u \in T^1_D) \left( r(K_G-u) + 1 \right)\\
&=& \delta(K_G \in T^2_D);\\
S^{1,0}_D(K_G-v) &=& \delta( K_G -v \in T^1_D) \left( r(K_G-v) + 1 \right) \\
&=& \delta( K_G \in T^2_D).
\end{eqnarray*}

From this we deduce that
$$S^{2,0}_D(K_G) - S^{1,0}_D(K_G-u) - S^{1,0}_D(K_G-v) = 0.$$
The identity $S^{d,e}_D(E) = S^{e,d}_{E-D}(E)$ shows similarly that
$$S^{0,2}_D(K_G) - S^{0,1}_D(K_G-u) - S^{0,1}_D(K_G-v) = 0.$$
Putting this together with \Cref{lem:invtauGeneral} shows that
\begin{align*}
\inv_k(\tauD) &= S_D(K_G) - S_D(K_G-u) - S_D(K_G-v) + S_D(K_G-u-v) \\&= S^{1,1}_D(K_G) + S^{0,0}_D(K_G-u-v).
\end{align*}
As discussed above, these terms are equal to $\# \{ [D'] \in T^1_D:\ |D'| \neq \emptyset \}$ and $\delta( 0 \in T^0_D \mbox{ and } u+v \sim K_G)$, respectively.
\end{proof}

\begin{proof}[Proof of \Cref{thm:kgtThetas}] \label{proof:kgtThetas}
First suppose that $[u-v] \in \Pic^0(G)$ is non-recurrent. Take any $[D] \in \pic{G}$ and let $\tau = \tauD$ (which exists since all divisors are submodular by assumption). Let $I  =\{ [D'] \in T^1_D:\ |D'| \neq \emptyset \}$. By \Cref{lem:invtau}, $\inv_k(\tau) = \# I$, so if $I$ is empty then we are done. Otherwise take $[D'] \in I$. Note that this is a degree $1$ rank $0$ divisor class, so by \Cref{lem-bridgelessFacts} there is a unique $w \in V(G)$ such that $D' \sim w$. Now note that if $[D+au-bv], [D+a'u-b'v] \in I$ then since they have the same degree $a-a' = b-b'$ and thus $D+a'u-b'v = (D+au-bv)+(a'-a)(u-v)$. Thus every divisor class in $I$ has a representative of the form $[w+n(u-v)]$. If we enforce that $0\leq n < k$ then these are all distinct. Since $[u-v]$ is non-recurrent, we conclude that there is at most one nonzero value of $n$ such that $[w+n(u-v)] \in I$ and thus $\inv_k(\tau) \leq 2$.

Now suppose that $(G,u,v)$ has $k$-general transmission. We claim this implies non-recurrence of $[u-v]$. Take any vertex $w \in V(G)$. Note that any degree $1$ twist of $w$ must be of the form $w + n(u-v)$. Every such class of twists has a representative satisfying $0\leq n <k$. Thus we have that
\begin{align*}
    \#\left\{[w+au-bv] \in T^1_w:\right.&\left.|w+au-bv|\neq \emptyset\right\} \\&= \#\set{[w+n(u-v)]:\ 0\leq n<k,\ |w+n(u-v)|\neq \emptyset }\\
    &= 1 + \#\set{n \in \set{1,\dots, k-1}:|w+n(u-v)| \neq \emptyset}.
\end{align*}
Applying \Cref{lem:invtau} again to the permutation $\tau_{w}^{u,v}$, the above equation then implies that
\begin{align*}
    \#\set{n \in \set{1,\dots, k-1}:|w+n(u-v)| \neq \emptyset} \leq 1.
\end{align*}
Thus $[u-v]$ is non-recurrent.
\end{proof}

\begin{theorem}\label{thm:g2general}
    If $(G,u,v)$ is a twice-marked bridgeless graph of genus $2$ and torsion order $k$, with $u \ne v$, then $G$ has $k$-general transmission if and only if:
    \begin{enumerate}[label = \arabic*)]
        \item $(G,u,v)$ is the vertex gluing of two twice-marked cycles, each of which has torsion order $k$;
       \item $G$ is a vertex gluing of two cycles, one of which has length $2$, and the marked points are the two vertices on that cycle;
       \item $G$ is a theta graph with $[u-v]$ non-recurrent and either
       \begin{enumerate}[label = \alph*)]
           \item $\set{u,v} = \set{v_{\alpha,0}, v_{\alpha,n_\alpha-1}}$
           \item $\set{u,v} = \set{v_{\alpha,1}, v_{\alpha,n_\alpha}}$
           \item $\set{u,v} = \set{v_{\alpha,i}, v_{\beta,j}}$ with $\alpha \neq \beta$ and $0<i<n_\alpha,0<j<n_\beta$
       \end{enumerate}
    \end{enumerate}
\end{theorem}

We will defer one statement from the proof to \Cref{prop-TriangleInversionNumber}; the rest of the proof is given below.

\begin{proof}
    First suppose that $G$ is genus $2$ and has $k$-general transmission. Then $G$ is either a vertex gluing of two cycles or a theta graph. If $G$ is a vertex gluing of two cycles, with marked points on distinct loops, then \cite[Lemma 2.5, Theorem 3.11]{Pfl22} implies that every divisor is submodular. If the marked points are on distinct loops, then $(G,u,v)$ is the vertex gluing of $(G_1,u_1,v_1)$ and $(G_2,u_2,v_2)$ where $G_1,G_2$ are cycles. Let $k_i$ denote the torsion order of $(G_i,u_i,v_i)$. If $k_1 \neq k_2$ then without loss of generality $k_1<k_2$. Note that this implies that
    \begin{align*}
        k_1(u-v)+v \sim k_1u_2 + (1-k_1)v_2
    \end{align*}
    Considering this as a divisor on $G_2$ it clearly must be effective since it is degree $1$ on a genus $1$ graph. Thus we have that $(u-v)+v, k_1(u-v)+v$ are rank $0$, so $[u-v]$ is recurrent. However by \Cref{thm:kgtThetas} this contradicts the assumption that $G$ has $k$-general transmission so we can assume $k_1 = k_2$ giving case $1)$.

    If the marked points are on the same cycle then by \Cref{rem-degenerateTheta} $k$-general transmission implies $2)$ above.
    
    On the other hand, if $G$ is a theta graph then by \Cref{cor-allSubmodSameStrand} we have that $G$ must be rigidly marked unless $\set{u,v} = \set{v_{0,0},v_{0,n_0}}$. This is incompatible with $k$-general transmission, the proof of which we defer to \Cref{prop-TriangleInversionNumber}. The remaining possibilities are the cases $3)$ $ a),$ $b)$ and $c)$ above, which are all rigidly marked and thus by \Cref{thm:kgtThetas} $[u-v]$ is non-recurrent.

    For the other direction, $1)$ implies $k$-general transmission by \cite[Theorem A]{Pfl22}. $2)$ implies $k$-general transmission by \Cref{lem-TO2GenTrans}. Lastly $3)$ implies $k$-general transmission since in each case the graph is rigidly marked and thus \Cref{thm:kgtThetas} gives the result.
\end{proof}

\begin{defi}\label{defn:evenlyMarked}
    Let $(\theta_{n_0,n_1,n_2},v_{\alpha,i},v_{\beta,j})$ be a twice-marked theta graph. We say that such a graph is \textit{evenly marked} if $\frac{i}{n_\alpha} = \frac{j}{n_\beta}, \alpha \neq \beta$ and $0<i<n_\alpha$.
\end{defi}

\begin{lemma}
If $(\theta_{n_0,n_1,n_2},v_{\alpha,i},v_{\beta,j})$ is evenly marked, then the class $[v_{\alpha,i}-v_{\beta,j}]$ is non-recurrent, with order $\frac{n_\alpha}{\gcd(n_\alpha,i)} = \frac{n_\beta}{\gcd(n_\beta, j)}$ in $\jac{\theta_{n_0,n_1,n_2}}$.
\end{lemma}

\begin{proof}
Without loss of generality, let $(\alpha,\beta) = (0,1)$. For ease of notation let $a:= n_0, b:= n_1, c:=n_2$ and $d:= \gcd(i,a),d' = \gcd(j,b)$. In this notation, the evenly marked assumption means that $aj = bi$. First we claim that the torsion order is $\frac{a}{d}$. This can be seen using the isomorphism in \Cref{prop-JacTheta}:
\begin{align*}
    \frac{a}{d}[i,-j] = \left[\frac{ai }{d},\frac{-aj}{d}\right] = \left[0,\frac{bi - aj}{d}\right] = 0.
\end{align*}
Thus the torsion order divides $\frac{a}{d}$. Further if $n[i,-j] = 0$ then there exists some $x,y \in \Z$ such that $n(i,-j) = x(a,-b) + y(a+c,c)$. The fact that $\frac{a}{b} = \frac{i}{j}$ means that we can assume $y = 0$. Thus we have that $a|ni$ so $\frac{a}{d}|n$ and we can conclude that $\frac{a}{d}$ is the torsion order, call it $k$. By a symmetric argument, $k =\frac{b}{d'}$ as well.

We now obtain a formula for each support complex $\Supp(K_G - n(v_{0,i}-v_{1,j}))$ for $0<n<k$. 
We follow the convention: for any $\ell \in \Z, m \in \Z$, $\ell\mod m$ denotes the nonnegative residue of $\ell$ modulo $m$, i.e. $\ell - \lfloor \frac{\ell}{m} \rfloor m$.
Using the identification of \Cref{prop-JacBanana}, and the equation $\lfloor \frac{ni}{a} \rfloor = \lfloor \frac{nj}{b} \rfloor$, we see that 
$$[ni,-n j ] = \left[ni - \left\lfloor \frac{ni}{a}\right\rfloor a,-n j + \left\lfloor \frac{ni}{a} \right\rfloor b \right] = \left[ ni \mod a, - (nj \mod b )\right].$$

Using the isomorphism in \Cref{prop-JacBanana} again, we deduce 
\begin{align}\label{eq:multDiffMarkedPts}
n(v_{0,i}-v_{1,j}) \sim v_{0,ni \mod a} - v_{1,nj \mod b},
\end{align}

and therefore

$$
K_G - n(v_{0,i} - v_{1,j}) \sim \overline{v}_{0, ni \mod a} + v_{1,nj \mod b}.
$$

Note the overline on the first vertex. Since we are assuming $0 < n < k$, we have $ni \not\equiv 0 \mod a$ and $nj \not\equiv 0 \mod b$. Therefore neither of the two vertices mentioned above are multivalent, so they do not lie on the same strand, and \Cref{cor:suppUV} implies that
$$\Supp \left( K_G - n(v_{0,i} - v_{1,j}) \right)
= \left\{ \overline{v}_{0, ni \mod a},\ v_{1,nj \mod b} \right\}.
$$

Since $k$ is the additive order of $i + a \Z$ in $\Z / a \Z$, and also of $j + b \Z$ in $\Z / b \Z$, it follows that the vertices $\{ \overline{v}_{0, ni \mod a}:\ 0 < n < k \}$ are $k-1$ distinct vertices on strand $0$, and $\{ v_{1,nj \mod b}:\ 0 < n < k \}$ are $k-1$ distinct vertices on strand $1$. So indeed these supports are pairwise disjoint, and $[v_{0,i} - v_{1,j}]$ is non-recurrent by \Cref{lem:nonrecDisjoint}.
\end{proof}

\begin{rema}
The simplicity of Equation \eqref{eq:multDiffMarkedPts} is the essential reason why evenly marked theta graphs present such a tractable case for analyzing recurrence; multiples of the difference of the marked points do not involve the third strand at all so these theta graphs ``act like'' cycles.
\end{rema}

Note that evenly marked theta graphs are rigidly marked by \Cref{cor-allSubmodSameStrand} so these two lemmas together establish the following corollary, which was stated in the introduction as the second part of \Cref{thm:thetaSimple}.

\begin{coro}\label{cor:evenlyMarkedKGT}
An evenly marked theta graph $(\theta_{n_0,n_1,n_2}, v_{\alpha,i}, v_{\beta,j})$  has $k$-general transmission, where $k = \frac{n_\alpha}{\gcd(n_\alpha, i)} = \frac{n_\beta}{\gcd(n_\beta, j)} $.
\end{coro}

\subsection{Banana Graphs Are Increasingly Special As Genus Increases}\label{ssec:mostBananas}

For almost all banana graphs of genus $\geq 3$ we can rule out general transmission using techniques we develop in \Cref{sec:mixedTorsion}, provided sufficiently high torsion order. Indeed, we saw in \Cref{thm-NSMForBanana} that $k$-general transmission is ruled out for most banana graphs because we can construct non-submodular divisors for many choices of markings.  In this section we show that as the genus increases, banana graphs become less and less general in the sense that the maximal number of inversions of a permutation on a given banana graph grows at least quadratically with the genus. Along the way we prove lower bounds on torsion orders of these graphs which we will use in \Cref{sec:mixedTorsion}.

Throughout this section $G = B_{n_0,\dots,n_g}$ is a banana graph of genus $g\geq 2$. Our general strategy will be to compute enough of the transmission permutations to get a lower bound on the number of $k$-inversions. In particular, if we are clever with our choice of divisor $D$, part of the permutation $\tauD$ can be determined with only the local data of the lengths of the strands supporting the marked points. Without loss of generality we assume these are $n_0$ and $n_1$. Let $D = gv_{0,n_0}, \tau = \tauD, k$ denote the torsion order, and $M$ the maximum number of $k$-inversions a transmission permutation of any divisor on $(G,u,v)$. With this notation, the main results are as follows.

\begin{theorem}\label{thm-quadInvGrowth}
    If $(G,u,v)$ is a twice-marked banana graph with genus $\geq 3$ where every divisor is submodular and the marked strands are sufficiently long, then $M$ is at least quadratic in $g$.
\end{theorem}

As shown in \Cref{thm-NSMForBanana}, we must consider four cases to establish this theorem: marking at the multivalent vertices, marking at $v_{0,0}$ and $v_{0,n_0-1}$, marking at $v_{0,1}$ and $v_{1,n_1-1}$, and marking at the midpoint $v_{0,1}$ in the case $n_0=2$. We establish the theorem in each of the first three cases separately, in the following three results: \Cref{prop-TriangleInversionNumber}, \Cref{prop-oneOffNotGeneral}, and \Cref{cor-bothOffMax}. The midpoint case 1(e) is excluded since the marked strand is not sufficiently long. In addition to this asymptotic result, our analysis also gives the following.

\begin{prop}\label{prop-bananTorsion}
    If $(G,u,v)$ is a twice-marked banana graph of genus $\geq 3$ and torsion order $k$ where every divisor is submodular then either
    \begin{enumerate}[label = \arabic*)]
        \item up to reordering the strands and swapping $u$ and $v$, $n_0 = 2$, $n_1$ is even, and $(G,u,v) = (G,v_{0,1},v_{1,n_1/2})$, and thus $k = 2$; or
        \item The torsion order is at least the genus, i.e. $k \geq g$.
    \end{enumerate}
\end{prop}

Again, this proposition is proved by considering each of the cases from \Cref{thm-NSMForBanana} in turn. The relevant bounds are proved, for each of these four cases, in \Cref{lem-TriangleInversionII}, \Cref{lem-BananOneOff}, \Cref{lem-bothOffTorOrder}, and \Cref{lem-midpointTorsion}. The bounds in those lemmas will each be at least $g$ except in the case of $n_0 = 2$, $n_1$ even, and $(u,v) = (v_{0,1},v_{1,n_1/2})$.

We now study the four cases identified in \Cref{thm-NSMForBanana} in turn.

\subsubsection{Marking the multivalent vertices}\label{subsub-g+1}
\begin{lemma}\label{lem-TriangleInversionII}
    With $(G,u,v) = (B, v_{0,0},v_{0,n_0})$, $D = g v_{0,n_0}$, and $\tau = \tauD$, for all $0 \leq b \leq g$ we have that
    \begin{align*}
        \tau(b) = g-b.
    \end{align*}
    As a consequence this yields $k \geq g$.
\end{lemma}

\begin{proof}
    This follows fairly directly from \Cref{cor-BananaDeltaComps} part $1)$ as follows:
    \begin{align*}
        \Delta(D+bu-(g-b)v) &= \Delta(bu+bv) = 1.
    \end{align*}
    The torsion order consequence follows from the fact that $\tau \in \eaf{k}$.
\end{proof}

\begin{prop}\label{prop-TriangleInversionNumber}
    With $(G,u,v)$ as above, we have $M \geq \binom{g+1}{2}$.
\end{prop}

\begin{proof}
    By \Cref{lem-TriangleInversionII}, the permutation associated to the divisor $D$ has at least $\binom{g+1}{2}$ $k-$inversions given by the pairs $\set{(i,j):i,j \in \set{0,\dots,g}, i<j}$.
\end{proof}

\begin{rema}
    As a consequence this entirely completes the picture for describing $k$-general transmission in genus $2$, which we proved except for the above case in \Cref{thm:g2general}. By the above proposition, $M\geq 3$, ruling out $k$-general transmission in such cases as well.
\end{rema}

\subsubsection{\texorpdfstring{Marking at $v_{0,0},v_{0,n_0-1}$}{Marking at v\_0,0 and v\_0,n0-1}}
Note that in order for such markings to make sense we require $n_0>1$. To make some results easier to state we use $f(g)$ to denote $\floor*{ \frac{g}{n_0-1}}$ in this section.

\begin{lemma}\label{lem-BananOneOff}
    Let $(G,u,v) = (G,v_{0,0},v_{0,n_0-1})$, $D = gv_{0,n_0}$ and $\tau = \tauD$. If $0\leq b \leq \frac{n_0}{n_0-1}g$ then
    \begin{align*}
        \tau(b) = \begin{cases}\frac{b}{n_0} & b \equiv 0 \mod n_0\\
        g+\frac{b+1}{n_0} & b \equiv -1 \mod n_0\\
        g + 2 \left\lfloor \frac{b}{n_0} \right\rfloor -b + 1 & b \not\equiv 0,-1 \mod n_0
        \end{cases}.
    \end{align*}
    As a consequence, we get that $k > \frac{n_0}{n_0-1}g$.
\end{lemma}

This lemma says that for sufficiently small values of $b$, $\tau(b)$ is very well described by the residues of $b$ modulo $n_0$. In \Cref{fig:one-off-perm} the lemma describes the behavior of a permutation for $0 \leq b \leq 11$.

\begin{figure}
    \centering
    \includegraphics[width = 10 cm]{PermGraphLegendBorder.pdf}
    \caption{The permutation $\tau_{gv_{0,n_0}}^{u,v}$ on the graph $(B_{5,4,4,3,3,3,3,3,3,3},v_{0,0},v_{0,n_0-1})$, a graph with genus $9$ and torsion order $91$.  Values are colored by residues modulo $n_0$ as indicated in the legend. The values predicted by \Cref{lem-BananOneOff} are those where $b\leq 11$ which are well sorted by color. Notice also that although the torsion order is $91$, the permutation appears ``quasiperiodic'' at greater frequency, in a manner we do not attempt to define precisely.
    \label{fig:one-off-perm}}
\end{figure}

\begin{proof}
    This all essentially follows from part $2)$ of \Cref{cor-BananaDeltaComps}. If $b = mn_0$ then
    \begin{align*}
        \Delta(D+mu-mn_0v) &= \Delta((g-(n_0-1)m)v_{0,n_0}) = 1.\\
    \end{align*}

    Now suppose $b = mn_0 -1$. Then we have
    \begin{align*}
        \Delta(D+(g+m)u-(mn_0-1)v) &= \Delta(gv_{0,0}+v+(g-m(n_0-1))v_{0,n_0}) = 1.
    \end{align*}

    Lastly, suppose that $b = mn_0 + n$ with $0<n<n_0-1$ and $m = \left\lfloor \frac{b}{n_0} \right\rfloor$. Then the claim is that $\tau(b) = g + 2m - b + 1$, and we can compute
    \begin{align*}
        \Delta(D + (g + 2m - b +1)u - bv) & = \Delta((g+m-b + n)v_{0,n_0} \\&+ (g+m-b+1)u-(b-mn_0)v)\\
        &= \Delta((g-(b-m))v_{0,n_0} + (g-(b-m))v_{0,0} + v_{0,n}) \\&= 1.
    \end{align*}
\end{proof}

\begin{prop}
    With the same notation as above, $M \geq \binom{(n_0-2)f(g)}{2}$.
\end{prop}

\begin{proof}
    By \Cref{lem-BananOneOff}, every pair of integers in the set $$\set{b:0<b<\frac{n_0}{n_0-1}g, b \not \equiv 0,-1\mod n_0}$$ gives a distinct $k$-inversion. The number of elements in this set is lower bounded by $(n_0-2)f(g)$ thus giving the result.
\end{proof}

We can improve this bound with a more careful analysis of the implications of \Cref{lem-BananOneOff}.

\begin{prop}\label{prop-oneOffNotGeneral}
    With the same notation as above, and letting $h(g) = f(g)(n_0-2) + \min{n_0-2,f(n_0g)-n_0f(g)}$,
    \begin{align*}
        M \geq \binom{f(g)+1}{2} + f(g)h(g) + \binom{h(g)}{2}.
    \end{align*}
\end{prop}

The idea for this proof is to split the inversions predicted by \Cref{lem-BananOneOff} into four different types depending on the residues of each coordinate of each inversion. In the notation of \Cref{fig:one-off-perm} these types are 
\begin{align*}
    (\text{light blue}, \text{yellow}),(\text{light blue}, \text{dark blue}),(\text{dark blue}, \text{yellow}),(\text{dark blue}, \text{dark blue}). 
\end{align*}

\begin{proof}
    First note that if $X = \set{b \in \Z: 0<b\leq \frac{n_0}{n_0-1}g}$ then
    \begin{align*}
        A:= \set{b \in X:b \equiv 0 \mod n_0}, \quad \#A &= f(g);\\
        B:= \set{b \in X: b \equiv -1 \mod n_0, b \neq \floor*{\frac{n_0g}{n_0-1}}}, \quad \#B &= f(g);\\
        C:= \set{b \in X:b\not \equiv 0,-1 \mod n_0}, \quad \\\#C = f(g)(n_0-2) + \min{n_0-2,f(n_0g) -n_0 f(g)} = h(g).
    \end{align*}
    We are aiming to lower bound $\inv_k(\tau)$ by counting elements of $I:=\Inv_k(\tau) \cap (X \times X)$. We can write this latter set as the disjoint union:
    \begin{align*}
        I = [I\cap (B\times A)] \cup [I \cap (B\times C)]\cup [I\cap (C\times A)] \cup [I \cap (C\times C)].
    \end{align*}
    It is a straightforward computation using \Cref{lem-BananOneOff} that
    \begin{align*}
        I \cap (B\times A) = \set{(b,b') \in B\times A:b<b'}; \quad I \cap (B\times C) = \set{(b,b') \in B\times C:b<b'};\\
        I \cap (C\times A) = \set{(b,b') \in C\times A:b<b'}; \quad 
        I \cap (C\times C) = \set{(b,b') \in C\times C:b<b'}.
    \end{align*}
    Another simple counting argument shows:
    \begin{align*}
    \begin{split}
        \#(I \cap (B\times A)) &= \binom{\#A+1}{2};\\
        \#(I \cap (C\times A)) &= (n_0-2)\binom{\#A+1}{2};
    \end{split}
    \begin{split}
        \#I \cap (B \times C) &= \#B\#C - (n_0-2)\binom{\#B+1}{2};\\
        \#I \cap (C\times C) &=\binom{\#C}{2}.
    \end{split}
    \end{align*}
    Putting this all together, we get that
    \begin{align*}
        \#I = \binom{\#A+1}{2} + \#B\#C + \binom{\#C}{2}.
    \end{align*}
\end{proof}

Note that this implies that $M$ is always lower bounded by a function which is quadratic in $g$. This is fairly easy to see when $f(g) >0$. If $f(g) = 0$, i.e. $g<n_0 -1$, then
\begin{align*}
    h(g) = \min{n_0-2,f(n_0g)} \geq \min{g-1,\frac{n_0}{n_0-1}g-1} = g-1.
\end{align*}

\begin{example}
    Examining the permutation in \Cref{fig:one-off-perm}, the proposition says that
\begin{align*}
    \inv_k(\tauD) \geq \#I = \binom{3}{2} + 2\cdot 7 + \binom{7}{2} = 38.
\end{align*}
This lower bound certainly exceeds the genus of $9$, but still misses most of the $217$ total $k$-inversions of this permutation.
\end{example}

\subsubsection{\texorpdfstring{Marking at $v_{0,1},v_{1,n_1-1}$}{Marking at v\_0,1 and v\_1,n1-1}}\label{subsub-bothOff}
The next class of markings we look at is the least well behaved of the three since neither of our two marked points are the multivalent vertices. In this context we are assuming that $n_0,n_1>1$.

We first give a lower bound on the torsion order for these markings.
\begin{lemma}\label{lem-bothOffTorOrder}
    The torsion order $k$ of $(G,v_{0,1},v_{1,n_1-1})$ is at least $g$ unless $n_0 = n_1 = 2$ in which case $k = 2$.
\end{lemma}

\begin{proof}
    The case where $n_0 = n_1 = 2$ is straightforward so suppose this does not hold. Take $a \in \Z$ such that $1\le a<g$. We aim to show that $a(v_{0,1}-v_{1,n_1-1})$ is not equivalent to the zero divisor. We use throughout that $D \sim 0$ if and only if $-D \sim 0$. Let $a = q_0n_0+r_0 = q_1n_1+r_1$ where $q_i,r_i \in \Z_{\ge 0}$ and $0\leq r_i < n_i$. Let $ m = q_0(n_0-1)-q_1, n = q_1(n_1-1)-q_0$. Then we have that
    \begin{align*}
        a(v_{0,1}-v_{1,n_1-1}) \sim mv_{0,0}-nv_{0,n_0}+r_0v_{0,1}-r_1v_{1,n_1-1}.
    \end{align*}
    \textbf{Case I:}
    If $n>0$ then we have that $a(u-v)$ has rank no greater than that of
    \begin{align*}
        mv_{0,0}-nv_{0,n_0}+r_0v_{0,1} \sim (m+r_0-1)v_{0,0}-nv_{0,n_0}+v_{0,r_0}.
    \end{align*}
    Since $m+r_0-1 \leq a - 1 <g-1$ this is $v_{0,n_0}$-reduced and not effective, thus our original divisor was not effective. 
    \textbf{Case II:}
    On the other hand if $n<0$ then we can see that $-a(u-v)$ has rank no greater than that of
    \begin{align*}
        -mv_{0,0} + (n+r_1-1)v_{0,n_0} + v_{1,n_1-r_1}.
    \end{align*}
    Since $n+m \geq 0$ and $n<0$ we have that $-m<0$ and by a symmetric argument as above $n+r_1-1 <g-1$ so this divisor is ineffective and $v_{0,0}$-reduced.\\
    \textbf{Case III:}
    This leaves the remaining case of $n = 0$. Note that this implies that $m\geq 0$. Thus either both $q_0$ and $q_1$ are zero or neither is. If $n_0 = n_1$ then $q_0 = q_1$ so we get that $q_0(n_0-2) = 0$, so either $q_0 = 0$ or $n_0 = 2$. Since we are assuming $n_0 = n_1$ and are excluding the case $n_0 = n_1 = 2$, this implies that $q_0 = 0$, and thus $r_0 = r_1\geq 1$. So we have that $a(u-v)$ has rank no greater than that of
    \begin{align*}
        r_0v_{0,1}-v_{1,n_1-1} \sim (r_0-1)v_{0,0}+v_{0,r_0} -v_{1,n_1-1}.
    \end{align*}
    Because $r_0-1 < a<g$ this is $v_{1,n_1-1}$-reduced and ineffective. Lastly we consider if $n_0 \neq n_1$. 
    
    Suppose for contradiction that $r_0 = r_1 = 0$.
    Then $q_0 n_0 = a = q_1 n_1$, hence $q_0$ and $q_1$ are nonzero, and $m = a - q_0 - q_1 = n = 0$, implying $q_1 = q_0 (n_0-1) = q_1 (n_1-1)(n_0-1)$.
    Since $q_1 \ne 0$, this implies $(n_1-1)(n_0-1) = 1$, and thus $n_0 = n_1 = 2$, contradicting our assumptions. Thus one of $r_0,r_1$ is at least $1$. If $r_0 \ge 1$ then 
    \begin{align*}
        -a(v_{0,1}-v_{1,n_1-1}) 
        &\sim -mv_{0,0} - r_0 v_{0,1}+r_1v_{1,n_1-1} 
        \\&\sim -mv_{0,0} - r_0 v_{0,1} + (r_1-1)v_{1,n_1}+v_{1,n_1-r_1}
        \\&\le -v_{0,1}+(r_1-1)v_{1,n_1}+v_{1,n_1-r_1}.
    \end{align*}
    This last divisor is effective away from $v_{0,1}$ (in the edge case $r_1 = 0$, the term $v_{1,n_1-r_1}$ is equal to $v_{1,n_1}$ and cancels $(r_1-1)v_{1,n_1})$,
    and applying Dhar's burning algorithm starting at $v_{0,1}$ will burn strand $0$, then vertex $v_{0,0}$, and then the $g$ strands other than strand $1$. Since $r_1 \le a < g$, the vertex $v_{1,n_1}$ burns, too, and thus the whole graph. So this divisor is $v_{0,1}$-reduced and ineffective, and thus $a(v_{0,1}-v_{1,n_1-1})$ cannot be linearly equivalent to $0$. 
    On the other hand, if $r_0 = 0$, then $r_1 \ge 1$ and we have $a(v_{0,1} -v_{1,n_1-1}) \sim m v_{0,0} - r_1 v_{1,n_1-1} \le m v_{0,0} - v_{1,n_1-1}$, which is $v_{1,n_1-1}$-reduced since $m \le a < g$, so again $a(v_{0,1} - v_{1,n_1-1}) \not\sim 0$.
\end{proof}

\begin{lemma}\label{lem-topOffBottomOffSimple}
    For $\max{2,g+3-n_0}\leq b \leq \min{g-1,n_1-2}$ and $D = gv_{0,n_0}$ we have that $\tauD(b) = g-b+2$.
\end{lemma}

\begin{proof}
    First note that
    \begin{align*}
        gv_{0,n_0}+(g-b+2)v_{0,1}-bv_{1,n_1-1} \sim (g-b)v_{0,0}+(g-b)v_{0,n_0}+v_{0,g-b+2}+v_{1,b}.
    \end{align*}
    Then the key observations are that because $b \leq g-1$, we have that $g-b \geq 1$ and because $b \geq 2$, $g-b \leq g-2$. Further because $b\geq g+3-n_0$ we have that $g-b +2\leq g-(g+3-n_0) + 2 = n_0-1$. This establishes the conditions for part $3)$ of \Cref{cor-BananaDeltaComps} thus proving the lemma.
\end{proof}

\begin{coro}\label{cor-bothOffMin}
    If $\min{n_0,n_1} \geq g+1$, then $M \geq \binom{g-2}{2}$.
\end{coro}

\begin{proof}
    Consider the set $A_{n_0,n_1} = \set{b \in \Z: \max{2,g+3-n_0}\leq b \leq \min{g-1,n_1-2}}$. Let $m = \min{n_0,n_1}$. Then we have the following
    \begin{align*}
        \#A_{n_0,n_1} \geq \#A_{m,m} = \begin{cases}
            0 & m < \frac{4+g}{2}\\
            2m-(g+4) & \frac{4+g}{2}\leq m \leq g\\
            g-2 & m\geq g+1
        \end{cases}.
    \end{align*}
    Thus by \Cref{lem-topOffBottomOffSimple} we see that
    \begin{align*}
        \inv_k(\tauD) \geq \binom{\#A_{n_0,n_1}}{2} \geq \binom{g-2}{2}.
    \end{align*}
\end{proof}

\begin{lemma}\label{lem-topOffBottomOff}
    If $D = gv_{0,n_0}, \tau = \tauD$ then
    \begin{enumerate}[label = \arabic*)]
        \item If $1 \leq b \leq \min{\frac{n_1}{n_1-1}g,n_1(n_0-1)}$ and $b \equiv 0 \mod n_1$ then $\tau(b) = \frac{b}{n_1}+1$.
        \item If $2 \leq b \leq \min{\frac{n_1}{n_1-1}g,n_1(n_0-1-g)-1}$,  and $b \equiv -1 \mod n_1$ then $\tau(b) = g+\frac{b+1}{n_1}$.
        \item If $2 \leq b\leq \frac{n_1}{n_1-1}g$, $2\floor*{\frac{b}{n_1}}-b \leq n_0-3-g$ and $b \equiv -n \mod n_1$ with $2 \le n \le n_1 - 1$, then $\tau(b) = g+2\frac{b+n}{n_1}-b$.
    \end{enumerate}
\end{lemma}

\begin{proof}
    We proceed similarly to the proof of \Cref{lem-BananOneOff}. In all cases, the constraints on $b$ above are chosen to keep the various indices ``in bounds'' and to satisfy criteria of \Cref{cor-BananaDeltaComps}. If $b = mn_1$ then
    \begin{align*}
        gv_{0,n_0}+(m+1)v_{0,1}-mn_1v_{1,n_1-1} \sim v_{0,m+1} + (g-m(n_1-1))v_{0,n_0}.
    \end{align*}
    This satisfies the conditions of \Cref{cor-BananaDeltaComps} part $3)$ so we are done. If $b = mn_1-1$ then
    \begin{align*}
        gv_{0,n_0}+(m+g)v_{0,1}&-(mn_1-1)v_{1,n_1-1} \\&\sim (g-1)v_{0,0} + (g-m(n_1-1))v_{0,n_0}+v_{0,m+g}+v_{1,n_1-1}.
    \end{align*}
    Again this reduces us to checking the conditions of \Cref{cor-BananaDeltaComps}.
    Lastly if $b = mn_1 +n$ where $0<n<n_1-1$ then we have
    \begin{align*}
        gv_{0,n_0}+(g+2m-b+2)v_{0,1}&-(mn_1+n)v_{1,n_1-1} \\&\sim (g+m-b)(v_{0,0}+v_{0,n_0}) + v_{0,g+2m-b+2} + v_{1,n}.
    \end{align*}
    Applying \Cref{cor-BananaDeltaComps} a final time completes the argument.
\end{proof}

\begin{coro}\label{cor-bothOffMax}
    When $n_0$ is sufficiently large relative to the genus, then we get a lower bound on $M$ which is quadratic in $g$.
\end{coro}

\begin{proof}
    For sufficiently large $n_0$ the bounds appearing on $b$ in \Cref{lem-topOffBottomOff} match those of \Cref{lem-BananOneOff}. Although the permutation of $gv_{0,n_0}$ on $(G,v_{0,0},v_{1,n_1-1})$ does not agree with the permutation of $gv_{0,n_0}$ of $(G,v_{0,1},v_{1,n_1-1})$, they have the same number of inversions with coordinates in the set $\set{2,3,\dots,\frac{n_1}{n_1-1}g}$, thus \Cref{prop-oneOffNotGeneral}, restricted to $b \in \{2, \dots, g+1\}$, gives the lower bound for $M$.
\end{proof}

\begin{rema}
    A symmetric result to \Cref{cor-bothOffMax} could be developed by studying the divisor $gv_{0,0}$. The point is that this gives a lower bound on $M$ when $\max{n_0,n_1}$ is sufficiently large. Note also that this lower bound on $\max{n_0,n_1}$ generally exceeds the lower bound on $\min{n_0,n_1}$ investigated in \Cref{cor-bothOffMin}.
\end{rema}

\subsubsection{\texorpdfstring{Marking at $v_{0,1}$ when $n_0=2$}{Marking at v\_0,1 when n0=2}}\label{ssec:midpointfamily}

The last class of markings to consider is the exception in case 1(e) of \Cref{thm-NSMForBanana}: suppose that $n_0=2$ and we mark at the midpoint $v_{0,1}$, while the other marked point is interior to a distinct strand. The case where the other marked point is also the midpoint of its strand is excluded by the statement of \Cref{prop-bananTorsion}. Therefore, without loss of generality, we may assume that the other marked point is $v_{1,j}$, where $1 \le j < \frac12 n_1$. The following lemma shows that in this case, the torsion order is at least $g$, as desired.

\begin{lemma} \label{lem-midpointTorsion}
Suppose that $n_0 = 2$ and $1 \le j < \frac12 n_1$. Then for every positive \emph{even} integer $k \le 2g-2$, we have $k v_{0,1} \not\sim k v_{1,j}$. In turn, if $k$ is any positive integer less than $g$, odd or even, then $k v_{0,1} \not\sim k v_{1,j}$.
\end{lemma}

\begin{proof}
Let $q = \left\lfloor \frac{kj}{n_1} \right\rfloor$ and let $r$ be the remainder $kj - q n_1$. Then observe that $k v_{1,j} \sim (k-1-q) v_{1,0} + v_{1,r} + q v_{1,n_1}$, which may also be written $(k-1-q) v_{0,0} + v_{1,r} + q v_{0,2}$. Similarly, using the assumption that $k$ is even, $k v_{0,1} \sim \frac{k}{2} v_{0,0} + \frac{k}{2} v_{0,2}$. Combining these, we obtain
\[ k v_{1,j} - k v_{0,1} \sim \left( \frac{k}{2} - q - 1 \right) v_{0,0} + v_{1,r} - \left( \frac{k}{2} - q\right) v_{0,2}. \]
Denote by $D$ the divisor on the right side of the display above. The assumption $j < \frac{n_1}{2}$ implies that $q < \frac{k}{2}$, and the definition of $r$ implies that $v_{1,r} \ne v_{0,2}$, so $D$ is effective away from $v_{0,2}$ but has negative coefficient at $v_{0,2}$. Now, suppose we apply Dhar's burning algorithm, beginning a fire at $v_{0,2}$. Then the $g$ strands other than strand $1$ all burn. There are at most $\frac{k}{2} - q$ chips at $v_{0,0}$ (equality being possible when $r=0$), which is at most $g-1$ by assumption on $k$, so $v_{0,0}$ burns as well, and in turn so must strand $1$. So $D$ is $v_{0,2}$-reduced and nonzero, hence it is not linearly equivalent to $0$. This shows the first sentence of the lemma; the second follows immediately by first doubling $k$.
\end{proof}

While discussing these length-$2$-midpoint exceptions, we mention another interesting feature of them: they provide an unexpected additional case where banana graphs have only submodular divisors.

\begin{lemma} \label{lem-midpointSubmodularity}
Suppose that $(B, v_{0,1}, v_{1,j})$ is a twice-marked banana graph of genus $g \ge 1$, where $n_0 = 2$ and $j$ is interior to its strand, i.e. $1 \le j \le n_1-1$. Every divisor is submodular on $(B, v_{0,1}, v_{1,j})$.
\end{lemma}

\begin{proof}
Write $u = v_{0,1}, v = v_{1,j}$, and let $D$ be any divisor on $B$. We wish to show that $\Delta(D) \ge 0$.
Since $\Delta(D) = \Delta(K_B + u + v - D)$, we may assume without loss of generality that $\deg(D) \le g$. We may also assume that $r(D) \ge 0$. We will apply \Cref{cor-BanaRankComp}, so let us also assume without loss of generality that $D$ is $v_{0,0}$-reduced. Then using \Cref{lem-v00reduced}, $D = a v_{0,0} + bv_{0,2} + E$, where $a,b \ge 0$ and $E$ is effective, supported on interior vertices, with at most one chip on each strand, and with at most $g-b$ chips total. 

First, consider the edge case $a = b= 0$. In this case, $r(D) = 0$, but Dhar's burning algorithm shows that for any vertex $w$ not in the support of $D$, $r(D-w) = -1$. So if $D-u$ and $D-v$ both have rank $0$, then $D \ge u + v$ and thus $D-u-v$ has rank $0$ as well. So we may assume that at least one of $a,b$ is strictly positive.

Since $a + b + \deg E = \deg D \le g$, Corollary \ref{cor-BanaRankComp} applies. If $E$ has a chip at $u$, it shows immediately that $r(D) = r(D-u)$. On the other hand, if $E$ has no chip at $u$, then $D-u \sim D - v_{0,0} - v_{0,2} + u$, 
which has the standard form $(a-1) v_{0,0} + (b-1) v_{0,2} + E + u$. In this standard form, at least one of $a-1,b-1$ is nonnegative by the previous paragraph, so Corollary \ref{cor-BanaRankComp} applies and
shows that $r(D-u) = r(D) - 1$. Suppose for contradiction that $\Delta(D) = -1$. It follows that $r(D-u) = r(D)$ but $r(D-u-v) \ne r(D-v)$. So $D$ has a chip at $u$. But note that, using the equivalence $v_{1,k} - v_{1,\ell} \sim v_{1,k+1} - v_{1,\ell+1}$ repeatedly, we can move the chips on strand $1$ alone to move $D-v$ into a linearly equivalent divisor $D'$ of the form $a' v_{0,0} + b' v_{0,2} + E'$, where $E'$ matches $E$ on the interior of all strands except $1$, $a' \ge a-1$, $b' \ge b-1$, and at least one of $a', b'$ is nonnegative. Since $\deg D' = \deg D - 1 \le g-1$, this new divisor again satisfies the hypotheses of \Cref{cor-BanaRankComp}, and still has a chip at $u$, so $r(D'-u) = r(D')$. Therefore $r(D-u-v) = r(D-v)$, which is a contradiction.
\end{proof}

\section{Symmetries and Quasi-Symmetries of Transmission Permutations}\label{ssec:mpAutos}

Apart from \eqref{eq-RRTauBounds}, used in \Cref{lem-TO2GenTrans}, this section is not needed for our main results, but is included to shine a light on some intriguing patterns in the examples discussed above, namely some symmetries of the permutations obtained. Perhaps more tantalizing are some ``quasi-symmetries'' in the permutations obtained, for which we have no formal definition, but which we wish to draw the reader's eye to.

In practice, one finds that transmission permutations often have more symmetry than expected for elements of $\eaf{k}$. In this section we give a brief accounting of where these additional sources of symmetry arise. In some cases, these symmetries can be brought to bear to yield information about general transmission on these graphs, although we are somewhat constrained by the fact that the symmetries in question do not necessarily impact the transmission permutation of \textit{every} divisor on a given graph.

Throughout this section, $D$ is assumed to be submodular on the twice-marked graph $(G,u,v)$.

The most obvious constraint on transmission permutations comes from the fact that the Riemann--Roch formula bounds the values of a transmission permutation. Namely
\begin{align}\label{eq-RRTauBounds}
    b-\deg D \leq \tauD(b) \leq 2g +b-\deg D.
\end{align}
For certain graphs we get an additional source of symmetry.

\begin{defi}
    For a twice-marked graph $(G,u,v)$ a \textit{marked point automorphism} $\phi$ is a pair $\phi_V:V(G)\to V(G)$ and $\phi_E:E(G)\to E(G)$ such that $\phi_V,\phi_E$ are bijections and if $e \in E(G)$ connects vertices $w_1,w_2$ then $\phi_E(e)$ connects $\phi(w_1)$ and $\phi(w_2)$. We further require that $\phi_V$ restricts to a bijection of the marked points.
\end{defi}

Given any such automorphism $\phi$ we get an induced automorphism on $\pic{G}$ given by $(\phi(D))(w) = D(\phi(w))$. Thus we get a bijection between the set of transmission permutations of $(G,u,v)$ and those of $(\phi(G),\phi(u),\phi(v))$ given by
\begin{align*}
    \tauD \mapsto \tau_{\phi(D)}^{\phi(u),\phi(v)}.
\end{align*}

A related phenomenon is the involution $\iota(D) = K_G-D+u+v$. This permutation of the elements of $\pic{G}$ does not arise from a marked point automorphism, but nonetheless induces constraints on the transmission permutations on $(G,u,v)$. We summarize the picture in the following lemma.

\begin{lemma}\label{lem:mpIds}
    If $\phi$ is a marked point automorphism of $(G,u,v)$ then the following are equivalent:
    \begin{enumerate}[label = \arabic*)]
        \item $\tauD(b) = a$;
        \item $\tau_D^{v,u}(-a) = -b$;
        \item $\tau_{\iota(D)}^{v,u}(a) = b$;
        \item $\tau_{\phi(D)}^{\phi(u),\phi(v)}(b) = a$.
    \end{enumerate}
    Note the change in the order of the marked points in cases $2)$, $3)$, and $4)$.
\end{lemma}

Thus relationships between a divisor $D$ and divisors obtained by applying combinations of marked point automorphisms and the involution $\iota$ can force a transmission permutation to have more than expected symmetry.

\begin{lemma}\label{lem-tauSyms}
    Let $(G,u,v)$ be a twice-marked graph and $\phi$ a marked point automorphism which transposes $u$ and $v$.
    \begin{enumerate}[label = \arabic*)]
        \item If $\phi(D) + D \sim K_G + u+v$ then $\delta(\tauD(b) = a) = \delta(\tauD(a) = b)$, i.e. $\left(\tauD\right)^2 = \id$.
        \item If $\phi(D) -D \sim n(u-v)$ for some $n \in \Z$ then $\delta(\tauD(b) = a)=\delta(\tauD(n-a) = n-b)$.
    \end{enumerate}
\end{lemma}

\begin{proof}
    For the first statement note that using the equivalence of $1),$ $3)$ and $4)$ of \Cref{lem:mpIds}
    \begin{align*}
        \delta(\tauD(b) = a) &= \delta(\tau_{\phi(D)}^{v,u}(b) = a)\\
        &=\delta(\tau_{K_G-\phi(D)+u+v}^{u,v}(a) = b)\\
        &=\delta(\tau_{D}^{u,v}(a) = b).
    \end{align*}
    For the second we can use the equivalence of $1)$, $2)$ and $4)$ of \Cref{lem:mpIds} to get
    \begin{align*}
        \delta(\tauD(b) = a) & = \delta(\tau_{\phi(D)}^{v,u}(b) = a)\\
        &=\delta(\tau_{D+n(u-v)}^{u,v}(-a) = -b)\\
        &=\delta(\tau_{D}^{u,v}(n-a) = n-b).
    \end{align*}
\end{proof}

\begin{example}
    For the $(g+1)-$valently marked banana graphs discussed in \Cref{subsub-g+1} there is a marked point automorphism $\phi$ given by $v_{\alpha,i}\mapsto v_{\alpha,n_\alpha-i}$ which transposes the marked points. For the divisor $D = gv_{0,n_0}$ we have that $\phi(D) = gv_{0,0}$ so it satisfies condition $2)$ of \Cref{lem-tauSyms}.

    The other situation  occurs in the markings discussed in \Cref{subsub-bothOff}. If $n_0 = n_1$ then we get an automorphism $\phi$ defined by
    \begin{align*}
        v_{\alpha,i}\mapsto \begin{cases}
            v_{0,n_0-i} & \alpha  = 1\\
            v_{1,n_1-i} & \alpha  = 0\\
            v_{\alpha,n_\alpha-i} & \alpha>1
        \end{cases}.
    \end{align*}
    Intuitively this is the automorphism which reverses every strand and then transposes the first two. Again here \Cref{lem-tauSyms} forces divisors such as $(g-1)v_{0,0}+u$ to be self-inverse.
\end{example}

These symmetries can be used to bound the number of inversions of the transmission permutations of some divisors. Notably, these techniques are also not restricted to banana graphs.
\begin{prop}
    If $\phi$ is a marked point automorphism of $(G,u,v)$ transposing $u$ and $v$ and $D \in \opdiv{G}$ such that $\phi(D) + D \sim K_G + u + v$, then
    \begin{align*}
        \inv_k(\tauD) \geq \sum_{M \in [k]} \Big[ r(D+(M-1)u-Mv) - r(D+(M-2)u-Mv) \Big].
    \end{align*}
\end{prop}

The basic idea here is that when our permutations are self-inverse, we can obtain a lower bound on the number of inversions by counting the number of points of the permutation which fall below the identity permutation (equivalently we could count those above).

\begin{proof}
    If $\tauD(b)<b$ then the pair $(\tauD(b),b)$ is an increasing pair such that $(\tauD(\tauD(b)),\tauD(b)) = (b,\tauD(b))$ is decreasing, thus an inversion. Since $b \in [k]$, these inversions are all distinct as $k$-inversions, thus giving the lower bound. Thus
    \begin{align*}
        \inv_k(\tauD) \geq \# \set{b \in [k]:\tauD(b)<b}.
    \end{align*}
    To count this latter set we use \Cref{lem:tauChars}. Let $B_\tau(i):= \set{b\geq i :\tau(b)\leq i-1}$. By an inclusion-exclusion argument,
    \begin{align*}
         \#\cup_{i \in [k]} B_\tau(i) &= \sum_{\emptyset \ne J \subseteq [k]}(-1)^{|J|+1}\#\bigcap_{i \in J} B_\tau(i)\\
         &=\sum_{M \in [k]}\sum_{m = 0}^M\sum_{\substack{J \subseteq [k]\\ \max{J} = M\\
         \min{J} = m}}(-1)^{|J|+1}\#\bigcap_{i \in J} B_\tau(i)\\
         &=\sum_{M \in [k]}\sum_{m = 0}^M\sum_{\substack{J \subseteq [k]\\ \max{J} = M\\
         \min{J} = m}}(-1)^{|J|+1}\# \set{b \geq M:\tau(b) \leq m-1}\\
         &=\sum_{M \in [k]}\sum_{m = 0}^M\# \set{b \geq M:\tau(b) \leq m-1}\sum_{\substack{J \subseteq [k]\\ \max{J} = M\\
         \min{J} = m}}(-1)^{|J|+1}\\
         &=\sum_{M \in [k]}\sum_{m = 0}^M\# \set{b \geq M:\tau(b) \leq m-1}\begin{cases}
             1 & m = M\\
             -1 & m = M-1\\
             0 & m<M-1
         \end{cases}\\
         &= \#\set{b\geq 0:\tau(b)\leq -1} \\&+ \sum_{M = 1}^{k-1} \Big[ \#\set{b \geq M :\tau(b) \leq M-1}-\# \set{b \geq M:\tau(b) \leq M-2} \Big].
    \end{align*}
    Thus in order to count subdiagonal points we add an error term:
    
    \begin{align*}
        \# \set{b \in [k]:\tau(b)<b} &= \#\cup_{i \in [k]} B_\tau(i) - \#\set{b \geq k:\tau(b) \leq k-2}\\
        &=\sum_{M \in [k]} \Big[ \#\set{b \geq M :\tau(b) \leq M-1}-\# \set{b \geq M:\tau(b) \leq M-2}\Big].
    \end{align*} 
    
    By \Cref{lem:tauChars} this is equivalent to
    \begin{align*}
        \sum_{M \in [k]} \Big[ r(D+(M-1)u-Mv) - r(D+(M-2)u-Mv) \Big].
    \end{align*}
\end{proof}

\begin{rema}
    Another notable feature is a kind of quasi-symmetry in which the permutation seems almost to belong to $\eaf{\ell}$ for some $\ell<k$, with a little bit of added noise. This pattern is particularly striking in \Cref{fig:one-off-perm} where despite a torsion order of $91$, the permutation almost seems to obey $\tau(b+12) = \tau(b)+12$ with some kind of error term. In some sense this suggests that $12(u-v)$ is ``close'' to being equivalent to $0$, but we do not yet have a framework for understanding this phenomenon.
\end{rema}

\section{Chains of mixed torsion orders}
\label{sec:mixedTorsion}

We prove in this section several criteria for Brill--Noether generality of graphs obtained by vertex gluing, or by chains thereof. We wish to prove criteria that are valid when twice-marked graphs of different torsion orders are chained together. The main purpose of this section is to prove \Cref{thm:bngChain} from the introduction, but the inductive tools developed here may be more broadly applicable.

\subsection{Brill--Noether generality from k-general transmission}

We begin with a result not involving any chains. For sufficiently large $k$, $k$-general transmission directly implies Brill--Noether generality of the underlying graph (forgetting the marked points).

\begin{prop}\label{prop:kgt-bngenl}
If $(G,u,v)$ is a twice-marked graph of genus $g$ with $k$-general transmission, and $k \geq \frac12 g + 1$, then $G$ is Brill--Noether general (as an unmarked graph).
\end{prop}

\begin{proof}
For contradiction, suppose that $D$ is a divisor on $G$ of degree $d$ and rank $r$ such that $\rho(g,r,d) < 0$, i.e. $(r+1)(g-d+r) > g$. 
Note that \Cref{lem:tauChars} implies the following two equations.
\begin{eqnarray*}
r+1 &=& \# \{ b \geq 0:\ \tauD(b) \leq 0 \};\\
g-d+r &=& r(K_G-D) +1 = \# \{ a < 0: \tauD(a) > 0 \}.
\end{eqnarray*}

It follows that the number $(r+1)(g-d+r)$ is equal to the number of inversions $(a,b)$ of $\tauD$ such that $a < 0, \tau(a) >0, b \geq 0$, and $\tau(b) \leq 0$.
Since there are more than $g$ such inversions, and $\inv_k(\tauD) \leq g$, some two such inversions $(a,b), (a',b')$ are $k$-equivalent. Without loss of generality, assume $a < a'$. We will use these values to construct more than $g$ distinct $k$-inversions, which will lead to a contradiction.

Let $I$ be the set of integers $i$ such that $a < i \leq b'$ and $a \not\equiv i \pmod{k}$. For each $i \in I$, associate an inversion $$f(i) = \begin{cases}
(a,i) & \mbox{ if } \tauD(i) \leq 0,\\
(i,b') & \mbox{ if } \tauD(i) > 0.
\end{cases}$$
We claim that for $i,j \in I$, if $i \neq j$ then $f(i)$ and $f(j)$ are not $k$-equivalent. To see this, observe that $(a,i),(a,j)$ cannot be $k$-equivalent, nor can $(i,b'),(j,b')$, since $i \neq j$. By symmetry, we need only check that $(a,i)$ cannot be $k$-equivalent to $(j,b')$. This follows from $a \not\equiv j \pmod{k}$, which is part of the definition of the set $I$.

Therefore $\inv_k(\tauD) \geq |I|$. Now, the number of integers $a < i \leq b'$ is $b'-a$. Since $b' > b$ and $b' \equiv b \pmod{k}$, we have $b' \geq b+k$. Similarly $a \leq a'-k$, so $b'-a \geq b-a' + 2k$. Since $a' < 0 \leq b$, we have $b-a' \geq 1$, so $b'-a \geq 2k+1$. It follows that $I$ definitely contains all the elements $a+1, a+2, \ldots, a+2k+1$ except $a+k$ and $a+2k$, so $|I| \geq 2k-1$. Since $k \geq \frac12 g + 1$, it follows that $\inv_k(\tauD) \geq |I| \geq g+1$, which contradicts $k$-general transmission.
\end{proof}

\begin{rema}
The bound $k \geq \frac12 g + 1$ in \Cref{prop:kgt-bngenl} is sharp. This is because one can construct, for any $g \geq 1$ and $k \geq 2$, twice-marked graphs $(G,u,v)$ such that $ku \sim kv$ and $r(ku) \ge 1$ (for example, a chain of loops glued at points differing by $k$-torsion). The bound $k \geq \frac12 g + 1$ is equivalent to $\rho(g,1,k) \geq 0$, so if $k < \frac12 g + 1$ there exist twice-marked graphs with $k$-general transmission that are not Brill--Noether general.
\end{rema}

As proved in \cite[Theorem A]{Pfl22}, chaining twice-marked graphs of the same torsion order preserves $k$-general transmission, so we can deduce an easy but somewhat restrictive criterion for Brill--Noether generality of chains.

\begin{coro}
Let $(G_i, u_i, v_i)$, for $i=1,2, \ldots, \ell$, be a sequence of $\ell$ twice-marked graphs, and $(G,u,v) = (G,u_1, v_\ell)$ the iterated vertex gluing. Let $g_i$ be the genus of $G_i$. If each $(G_i,u_i,v_i)$ has $k$-general transmission for the same value of $k$, and $k \geq \frac12 (g_1 + \ldots + g_\ell) + 1$, then $G$ is Brill--Noether general.
\end{coro}

\Cref{prop:kgt-bngenl} also provides the last ingredient needed to prove a slightly strengthened form of \Cref{thm:bananas}.

\begin{coro} \label{cor:bananasWithKGT}
    The only twice-marked banana graphs of genus $\geq 3$ which have $k$-general transmission are $$(B_{n_0, \ldots, n_g},v_{\alpha,n_\alpha / 2},v_{\beta,n_\beta / 2})$$ with $\alpha \neq \beta$, both $n_\alpha, n_\beta$ even, and either $n_\alpha = 2$ or $n_\beta = 2$. These examples have $2$-general transmission.
\end{coro}

\begin{proof}
Suppose that $(B_{n_0, \ldots, n_g},v_{\alpha,i},v_{\beta,j})$ has $k$-general transmission. Then all divisors are submodular, so
\Cref{prop-bananTorsion} shows that either we are in the torsion order $2$ case described in the statement, or $k \geq g$. Suppose for contradiction that we are in the latter case. Since $g \geq 3$, this implies $k \geq \frac12 g + 1$, so \Cref{prop:kgt-bngenl} implies that the banana graph $B_{n_0, \ldots, n_g}$ is Brill--Noether general. But  this contradicts \Cref{lem:g12}, since we assume $g \geq 3$.

It remains to show that the exceptional twice-marked graphs described do indeed have $2$-general transmission. Relabeling if needed, \Cref{lem-midpointSubmodularity} shows that every divisor is submodular. Since $2 u \sim v_{0,0} + v_{0,2} \sim 2 v$, the torsion order is $2$. Therefore \Cref{lem-TO2GenTrans} shows that they have $2$-general transmission.
\end{proof}

\begin{rema}
    Although we defined the notion of ``evenly marked'' only in genus $2$, it's noteworthy that the higher genus banana graphs with $2$-general transmission are also evenly marked in the same sense. Moreover, by \Cref{thm-NSMForBanana}, we get that a banana graph of genus $\geq 3$ has $k$-general transmission if and only if it is evenly marked and every divisor is submodular. But it is important to note that these hypotheses are extremely restrictive; the only evenly marked banana graphs of genus $\geq 3$ on which every divisor is submodular are the torsion order $2$ graphs discussed in the corollary.
\end{rema}

\subsection{Attaching a twice-marked graph to a one-marked graph}

Attaching a twice-marked graph with $k$-general transmission preserves Brill--Noether generality of a once-marked graph, as long as $k$ is large enough.

\begin{theorem}\label{thm:glueBNGtoKGT}
Let $(G_1, u_1, v_1), (G_2, u_2, v_2)$ be two twice-marked graphs of genera $g_1,g_2$ on which all divisors are submodular, and denote by $(G,u,v) = (G, u_1, v_2)$ their vertex gluing.
Suppose that $(G_1,v_1)$ is Brill--Noether general as a marked graph, and $(G_2, u_2,v_2)$ has $k$-general transmission, where $k > g_1 + g_2$. Then $(G,v)$ is Brill--Noether general.
\end{theorem}

\begin{rema}
The marked point $u_1$ is seemingly extraneous here, since it is only the divisor theory of $(G_1,v_1)$ that we care about. Indeed, we are quite confident that \Cref{thm:glueBNGtoKGT} remains true if $u_1$ is not mentioned, and indeed without the ``all divisors submodular on $(G_1,u_1,v_1)$'' hypothesis. However, we include this assumption to simplify exposition, as it allows the machinery of Demazure products to be used directly.
\end{rema}

We establish this theorem in the course of this subsection. Before doing so, we note that we can take $G_1$ to be a single vertex (genus $0$) in \Cref{thm:glueBNGtoKGT} to obtain Brill--Noether generality of once-marked graphs obtained by forgetting one marked point on a twice-marked graph with $k$-general transmission, provided that $k$ is large enough.

\begin{coro}
If $(G,u,v)$ is a twice-marked graph of genus $g$ with $k$-general transmission, and $k > g$, then $(G,v)$ is a Brill--Noether general once-marked graph.
\end{coro}

Of course, this corollary gives a weaker torsion bound on unmarked graphs than \Cref{prop:kgt-bngenl}.

We now proceed to the proof of \Cref{thm:glueBNGtoKGT}. We must first relate Weierstrass partitions to transmission permutations; the key mechanism is the following count.

\begin{defi}
A \emph{sign-changing inversion} of a permutation $\alpha$ is a pair $(u,v) \in \Z^2$ with $u<v$ and $\alpha(u) > 0 \geq \alpha(v)$. Denote the number of sign-changing inversions by $\sci(\alpha)$.
\end{defi}

\begin{prop}\label{prop:sciLambda}
If $D$ is a submodular divisor on a twice-marked graph $(G,u,v)$, then
$$\sci(\tauD) = \left| \lambda(D,v) \right|.$$
\end{prop}

\begin{proof}
To simplify notation, write $s_i$ instead of $s_i(D,v)$.
The numbers $s_0, s_1, \ldots$ are precisely the integers $\ell$ such that $r(D+\ell v) > r(D + (\ell-1)v)$. By definition of $\tauD$, these are the integers $\ell$ such that $\tauD(-\ell) \leq 0$. Hence the sign-changing inversions of $\tauD$ are the pairs $(a,-s_i)$, where $i \geq 0$, $a < -s_i$, and $\tauD(a) > 0$.
It follows that

\begin{eqnarray*}
\sci(\tauD) &=& \sum_{i=0}^\infty \# \{ a < -s_i:\ \tauD(a) > 0 \}\\
&=& \sum_{i=0}^\infty \left( r(K_G - D - s_i v) + 1\right)\\
&=& \sum_{i=0}^\infty \left( r(D + s_i v) - \deg( D + s_i v) +g\right) \mbox{ (by Riemann--Roch)}\\
&=& \sum_{i=0}^\infty \left( i - \deg D - s_i + g \right).\\
\end{eqnarray*}
This last sum is equal to $\displaystyle \sum_{i=0}^\infty \lambda_i = |\lambda|$.
\end{proof}

Before proceeding, we require some facts about a certain associative operation on permutations, the \emph{Demazure product}, which we denote by $\star$. For our purposes, this operation is defined on the group $\asp$ of \emph{almost-sign-preserving permutations}, consisting of all bijections $\alpha: \Z \to \Z$ for which $n$ and $\alpha(n)$ have the same sign for all but finitely many $n$. To each $\alpha \in \asp$ we associate a function $s_\alpha: \Z^2 \to \Z_{\geq 0}$ given by
$$s_\alpha(a,b) = \# \{\ell \geq b:\ \alpha(\ell) < a\}.$$
The Demazure product is characterized by the following min-plus matrix multiplication formula.
$$s_{\alpha \star \beta}(a,b) = \min{
s_\alpha(a,\ell) + s_\beta(\ell,b):\ \ell \in \Z
}.$$
Theorem A of \cite{PflDemProd} proves that this formula uniquely determines a well-defined associative operation on $\asp$. We also use a handy formula for specific computations. Following the notation of \cite{PflDemProd}, we will write, for a set $S \subseteq \Z$ containing no two consecutive integers, $\sigma_S$ for the permutation exchanging $n$ and $n+1$ for all $n \in S$, but fixing all other integers. By \cite[Theorem 8.7]{PflDemProd}, we have for all $\alpha \in \asp$ and sets $S$ as above,

\begin{equation}\label{eq:starSigma}
\alpha \star \sigma_S =  \alpha \sigma_T
\mbox{ where }
T = \{ \ell \in S:\ \alpha(\ell) < \alpha(\ell+1)\}.
\end{equation}

The relation between $\star$ and divisors on graphs is established in \cite[Theorem 3.11]{Pfl22}: if $(G,u,v)$ is the vertex gluing of $(G_1,u_1, v_1)$ and $(G_2, u_2, v_2)$, $D_1$ is a submodular divisor on $G_1$, and $D_2$ is a submodular divisor on $G_2$, then $D = D_1 + D_2$, regarded as a divisor on $G$, is also submodular and has transmission permutation given by

\begin{equation}\label{eq:tauGlued}
\tauD = \tauDi{1} \star \tauDi{2}.
\end{equation}

\begin{rema}
Because \cite{Pfl22} was written before \cite{PflDemProd}, the results in \cite{Pfl22} are phrased in a way that does not assume $\star$ is well-defined for all pairs of permutations, but they quickly imply the above claims in light of the results of \cite{PflDemProd}.
\end{rema}

\begin{lemma}\label{lem-SciSimpleRefl}
Let $k \geq 2$ be an integer, and $\alpha \in \asp$ a permutation with $\sci(\alpha) \leq k-2$. 
Denote by $\sigma^k_n$ the permutation $\sigma_{n + k \Z}$.
For any integer $n$,
$$\sci(\alpha \star \sigma^k_n) \leq \sci(\alpha)+1.$$
\end{lemma}

\begin{proof}
By Equation \eqref{eq:starSigma}, we have
$$\alpha \star \sigma^k_n = \alpha \sigma_S,
\mbox{ where }
S = \{ \ell \in n + k \Z:\ \alpha(\ell) < \alpha(\ell+1)\}.$$
For any pair $(u,v)$ with $u < v$, $(u,v)$ is a sign-changing inversion of $\alpha \sigma_S$ if and only if either
\begin{enumerate}[label = \arabic*)]
    \item $(u,v)$ is \emph{not} an inversion of $\sigma_S$ and $(\sigma_S(u), \sigma_S(v))$ is a sign-changing inversion of $\alpha$, or
    \item $(u,v)$ \emph{is} an inversion of $\sigma_S$, and $\alpha \sigma_S (u) > 0 \geq \alpha \sigma_S(v)$.
\end{enumerate}
Both of these statements can be simplified. The definition of $S$ implies that $\alpha$ and $\sigma_S$ have no inversions in common, so the first phrase of case $1)$ is redundant. This gives an embedding of the sign-changing inversions of $\alpha$ into those of $\alpha \sigma_S$. The remaining sign-changing inversions of $\alpha \sigma_S$ are those described in case $2)$. The only inversions of $\sigma_S$ are $(\ell,\ell+1)$ for $\ell \in S$, so we may write
$$\sci(\alpha \sigma_S) = \sci(\alpha) + \# \{ \ell \in S:\ \alpha(\ell) \leq 0 < \alpha(\ell+1) \}.$$
Therefore it suffices to demonstrate that there is at most one integer $\ell \in S$ such that $\alpha(\ell) \leq 0 < \alpha(\ell+1)$. We establish this by contradiction. Suppose that $m_1, m_2 \in S$ are two such integers, with $m_1 < m_2$. Then $m_2 - m_1 \geq k$, since they are congruent modulo $k$. Now, for each integer $u$ with $m_1+2 \leq u \leq m_2-1$, either $(m_1+1,u)$ or $(u,m_2)$ is a sign-changing inversion of $\alpha$, depending on whether or not $\alpha(u) \leq 0$. This accounts for $m_2-m_1-2  \geq k-2$ sign-changing inversions of $\alpha$. Furthermore, $(m_1+1,m_2)$ is an additional sign-changing inversion. So $\alpha$ has at least $k-1$ sign-changing inversions. This contradiction completes the proof.
\end{proof}

\begin{prop}
\label{prop:sciInvStar}
Suppose $\alpha \in \asp$ and $\beta \in \eaf{k}$ satisfy
$$k > \sci(\alpha) + \inv_k(\beta).$$
Then
$$\sci(\alpha \star \beta) \leq \sci(\alpha) + \inv_k(\beta).$$
\end{prop}

\begin{proof}
We proceed by induction on $\inv_k(\beta)$. If $\inv_k(\beta) = 0$, then $\beta$ is a shift permutation $\iota_d$ and thus $\sci(\alpha\star \beta) = \sci(\alpha\beta) = \sci(\alpha)$.
Now suppose that $\inv_k(\beta) >0$, and that the result holds for all choices of $\beta$ with fewer $k$-inversions.
Then for some $n$, we can factor $\beta$ as $\sigma^k_n \beta'$, where $\inv_k(\beta') = \inv_k(\beta)-1$. In fact, we must also have $\beta = \sigma^k_n \star \beta'$ by Equation \eqref{eq:starSigma}.
Using this and the associativity of $\star$, we have
\begin{align*}
    \alpha \star \beta = \alpha \star (\sigma_n^k \star \beta') = (\alpha \star \sigma_n^k) \star \beta'.
\end{align*}
Now, \Cref{lem-SciSimpleRefl} implies that $\sci(\alpha \star \sigma^k_n) \leq \sci(\alpha)+1$, and thus $k > \sci(\alpha \star \sigma^k_n) + \inv_k(\beta')$. Therefore the inductive hypothesis applies:
\begin{align*}
    \sci(\alpha \star \beta) &= \sci((\alpha \star \sigma_n^k) \star \beta') \leq \sci(\alpha \star \sigma_n^k) + \inv_k(\beta') \\&\leq \sci(\alpha) + 1 + \inv_k(\beta) -1 = \sci(\alpha) + \inv_k(\beta),
\end{align*}
which completes the induction.
\end{proof}

\begin{proof}[Proof of \Cref{thm:glueBNGtoKGT}]
Assume $(G_1, v_1)$ is Brill--Noether general and has all divisors submodular, and that $(G_2, u_2, v_2)$ has $k$-general transmission, where $k > g_1 + g_2$. 
Then the vertex gluing has all divisors submodular; this follows from Equation \eqref{eq:tauGlued}. 
Let $D$ be any divisor on $G$, and split $D$ into $D = D_1 + D_2$, where $D_1$ is a divisor on $G_1$ and $D_2$ is a divisor on $G_2$. Then by Equation \eqref{eq:tauGlued}, $D$ is submodular, and
$$\tau_D^{u_1,v_2} = \tau_{D_1}^{u_1, v_1} \star \tau_{D_2}^{u_2,v_2}.$$

Since $(G_1, v_1)$ is Brill--Noether general, \Cref{prop:sciLambda} implies that $\sci(\tau_{D_1}^{u_1, v_1}) = |\lambda(D_1, v_1)| \leq g_1$. Since $(G_2, u_2, v_2)$ has $k$-general transmission, $\inv_k(\tau_{D_2}^{u_2,v_2}) \leq g_2$. Therefore $k > g_1 + g_2$ implies $k > \sci(\tau_{D_1}^{u_1,v_1}) + \inv_k(\tau_{D_2}^{u_2,v_2})$, and \Cref{prop:sciInvStar} implies
$$ |\lambda(D,v)| = \sci(\tauD) = \sci(\tau_{D_1}^{u_1,v_1} \star \tau_{D_2}^{u_2,v_2}) \leq \sci(\tau_{D_1}^{u_1,v_1}) + \inv_k(\tau_{D_2}^{u_2,v_2}) \leq g_1 + g_2.$$
So $(G,v)$ is Brill--Noether general.
\end{proof}

\subsection{Attaching two once-marked graphs}

Attaching two Brill--Noether general once-marked graphs behaves as one would hope.

\begin{prop}\label{prop:glueMarked}
If $(G_1,v_1)$ and $(G_2,v_2)$ are two Brill--Noether general marked graphs of genera $g_1,g_2$, and $G$ is the genus $g = g_1+g_2$ graph obtained by gluing  $v_1$ to $v_2$, then $G$ is Brill--Noether general.
\end{prop}

\begin{proof}
Fix a divisor $D$ on $G$, and split $D$ as a sum $D = D_1 + D_2$, where $D_1$ is a divisor on $G_1$ and $D_2$ is a divisor on $G_2$. Abbreviate $\deg D_1, \deg D_2$ by $d_1,d_2$.  By \cite[Prop. 3.15]{Pfl22}, 
$$r_G(D) = \min{ r_{G_1}(D_1 + \ell v_1) + r_{G_2}(D_2 - (\ell+1) v_2) + 1:\ \ell \in \Z}.$$
In this formula, subscripts of $r$ indicate the graph considered when rank is computed, e.g. $r_{G_1}(D_1)$ refers to the rank of $D_1$ as a divisor on $G_1$, not as a divisor on $G$. This notation should not be confused with the notation $r_W$ used in the discussion of rank-determining sets in \Cref{lem-BananaRDS}.

We will use this formula for $r(D)$ to obtain a lower bound on the Weierstrass partitions of $D_1$ and $D_2$. For simplicity of notation, we write $\lambda(D_1,v_1)$ to refer to the Weierstrass partition of $D_1$ as a divisor on $(G_1,v_1)$, and use the notation $s(D_1, v_1)$ similarly. The same remarks apply to $D_2$ on $G_2$. 

The formula for $r_G(D)$ implies that for all $\ell \in \Z$,
$$r_{G_1}(D_1 + \ell v_1) + r_{G_2}(D_2 - (\ell+1) v_2) \geq r-1.$$
Fix an integer $i \in \{0, 1, \ldots, r\}$, and let $\ell = s_i(D_1,v_1)-1$ in the formula above. Then $r_{G_1}(D_1 + \ell v_1) = i-1$ by definition, so the inequality is equivalent to
$r_{G_2}(D_2 - (\ell+1) v_2) \geq r - i$. Equivalently, $-(\ell+1) \geq s_{r-i}(D_2, v_2)$. By our choice of $\ell$, this is equivalent to 
$$0 \geq s_i(D_1,v_1) + s_{r-i}(D_2,v_2).$$
Now, using the definition of Weierstrass partitions, this inequality is equivalent to
$$ 0 \geq i + g_1 - d_1 - \lambda_i(D_1,v_1) + r-i + g_2 - d_2 - \lambda_{r-i}(D_2,v_2).$$
Upon writing $g = g_1 + g_2$ and $d = d_1 + d_2$, this is equivalent to
$$\lambda_i(D_1, v_1) + \lambda_{r-i}(D_2,v_2) \geq g-d+r.$$
This inequality on Weierstrass partitions must hold for all $i \in \{0, 1, \ldots, r\}$. Summing gives
$$ |\lambda(D_1,v_1)| + |\lambda(D_2,v_2)| \geq (r+1)(g-d+r).$$
Since we assumed that $(G_1,v_1)$ and $(G_2,v_2)$ are Brill--Noether general marked graphs, it follows that $|\lambda(D_1,v_1)| + |\lambda(D_2,v_2)| \leq g_1 + g_2 = g$, and the result follows.
\end{proof}

\begin{rema}
The proof above can be reorganized slightly to give the following formula for $r(D)$.
$$
r(D) = \operatorname{max} \left\{ r \in \Z:\ 
\lambda_i(D_1,v_1) + \lambda_{r-i}(D_2,v_2) \geq g-d+r
\mbox{ for all $i \in \{0,1,\ldots,r\}$ }
\right\}.
$$
Note that the condition in this set of possible $r$ is vacuous when $r = -1$, so $r(D) \geq -1$ for all $D$ (as it should be).
This formula is reminiscent of the ``compatibility condition'' for limit linear series on nodal algebraic curves (see for example \cite[Definition 5.33]{HarrisMorrison} and the surrounding discussion).
\end{rema}

We now prove \Cref{thm:bngChain} from the introduction.

\begin{coro}[\Cref{thm:bngChain}]
Let $(G_i, u_i, v_i)$, for $i=1,2, \ldots, \ell$, be a sequence of $\ell$ twice-marked graphs, and $(G,u,v) = (G,u_1, v_\ell)$ the iterated vertex gluing. Let $g_i$ and $k_i$ be the genus of $G_i$ and torsion order of $(G_i, u_i, v_i)$, respectively. Suppose that each $(G_i, u_i, v_i)$ has $k_i$-general transmission.
\begin{enumerate}[label = \arabic*)]
    \item If $k_i > g_1 + g_2 + \ldots + g_i$ for all $i$, then $(G,v)$ is a Brill--Noether general marked graph.
    \item If $k_i > \min{g_1 + g_2 + \ldots + g_i, g_i + g_{i+1} + \ldots + g_\ell }$ for all $i$, then $G$ is a Brill--Noether general graph.
\end{enumerate}
\end{coro}

\begin{proof}
Part $1)$ follows by induction on $\ell$, using \Cref{thm:glueBNGtoKGT} for the inductive step. Part $2)$ follows from \Cref{prop:glueMarked} upon splitting the chain into two once-marked chains of genera $g_1 + \ldots + g_j$ and $g_{j+1} + \ldots + g_\ell$, where $j$ is the maximum index such that $g_1 + \ldots + g_j \leq g_{j} + g_{j+1} + \ldots + g_\ell$.  This choice of $j$ means that both halves are Brill--Noether general marked graphs, by part $1)$.
\end{proof}

\bibliographystyle{plain}
\bibliography{main}
\end{document}